\documentclass[12pt]{amsart}%
\usepackage{amsfonts}
\usepackage{amsmath}
\usepackage{amssymb}
\usepackage{graphicx}%
\providecommand{\U}[1]{\protect\rule{.1in}{.1in}}
\newtheorem{theorem}{Theorem}
\theoremstyle{plain}

\newtheorem{corollary}{Corollary}

\newtheorem{example}{Example}

\newtheorem{lemma}{Lemma}

\newtheorem{proposition}{Proposition}
\newtheorem{remark}{Remark}

\numberwithin{equation}{section}
\numberwithin{equation}{section}
\numberwithin{theorem}{section}
\numberwithin{lemma}{section}
\numberwithin{remark}{section}
\numberwithin{example}{section}
\numberwithin{proposition}{section}
\numberwithin{definition}{section}
\numberwithin{corollary}{section}
\begin{document}
\title[Algebraic spectral theory]{Algebraic spectral theory and Serre multiplicity formula}
\author{Anar Dosi (Dosiev)}
\address{Middle East Technical University Northern Cyprus Campus, Guzelyurt, KKTC,
Mersin 10, Turkey}
\email{(dosiev@yahoo.com), (dosiev@metu.edu.tr)}
\date{April 8, 2021}

\begin{abstract}
The present paper is devoted to an algebraic treatment of the joint spectral
theory within the framework of Noetherian modules over an algebra finite
extension of an algebraically closed field. We prove the spectral mapping
theorem and analyze the index of tuples in purely algebraic case. The index
function over tuples from the coordinate ring of a variety is naturally
extended up to a numerical Tor-polynomial. Based on Serre's multiplicity
formula, we deduce that Tor-polynomial is just the Samuel polynomial of the
local algebra.

\end{abstract}
\maketitle

\section{Introduction}

The joint spectral theory plays a central role in operator theory and the
complex analytic geometry. Its origin goes back to Gelfand's commutative
Banach algebras and their representations. The joint spectral theory was
mainly developed within the context of holomorphic functional calculus problem
by J. Taylor \cite{Tay1}, \cite{TayGF}. An interesting link to the complex
analytic geometry was found by M. Putinar \cite{Put}, \cite{PutE} by
constructing Taylor's holomorphic functional calculus within the context of
Stein spaces. Taylor's spectral theory has a strong homological background
developed independently by J. L. Taylor in \cite{TayGF} and A. Ya. Helemskii
in \cite{Hel}, that allows to have a further generalization of the theory in
the noncommutative setting. That played a crucial role in foundations of
noncommutative complex analytic geometry (see \cite{Dizv}, \cite{DIZV},
\cite{DComA} and \cite{Pir3}).

It is well known that many key results from the complex analytic geometry have
their analogs in schemes such as the fundamental theorem of Serre on vanishing
\cite[3.3.7]{Harts}. More interesting sight is to have a scheme-theoretic
analog of the joint spectra found in \cite{DCRM11} and \cite{DMMJ} that
indicates to the fundamental nature of the joint spectral theory. Furthermore,
Taylor's multivariable functional calculus has a scheme version suggested in
\cite{DMMJ}. Namely, Putinar spectrum can be defined for a module over a
scheme and it plays the same central role in the functional calculus problem
over Noetherian schemes as it does in the complex analytic geometry. Actually,
in the affine case the same spectrum occurred in \cite{NY1} and \cite{NY2}
based on the result of A. Neeman \cite{Ne} that there is a bijection between
subsets of $\operatorname{Spec}\left(  R\right)  $ and localizing
subcategories of the derived category of complexes of $R$-modules. Taylor's
functional calculus through derived categories was proposed in \cite{PirM} by
A. Yu. Pirkovskii, which is a complex analytic version of Neeman's result.

In the paper we intend to develop a spectral theory in purely algebraic case
with all its key properties and their links to algebraic geometry. As a basic
tool we are exploiting many results from commutative algebra. Actually, some
key results of commutative algebra can be retreated from point of view the
joint spectral theory. For example, associated primes of a Noetherian module
play the role of eigenvalues whereas support primes are spectral values. That
approach inherits Koszul homology groups of an algebraic variety and the
related index. In the case of a variety the index is reduced to the
multiplicity from the local theory \cite{Se}, and we use the multiplicity
formula of Serre for calculation of Koszul homology groups.

Fix a commutative $k$-algebra $R/k$ with an $n$-tuple $x$ in $R$, and an
$R$-module $M$. It turns out that $x$ is a family of commuting linear
transformations acting on the $k$-vector space $M$, and we have their
\textit{Taylor spectrum} $\sigma\left(  x,M\right)  $ to be the set of those
$a\in\mathbb{A}^{n}$ such that the Koszul complex $\operatorname{Kos}\left(
x-a,M\right)  $ fails to be exact. If the homology groups $H_{p}\left(
x-a,M\right)  $ of $\operatorname{Kos}\left(  x-a,M\right)  $ are finite
dimensional $k$-vector spaces we have the index $i\left(  x-a\right)
=\sum_{p=0}^{n}\left(  -1\right)  ^{p}\dim_{k}\left(  H_{p}\left(
x-a,M\right)  \right)  $ of the tuple $x-a$. The last homology group
$H_{n}\left(  x-a,M\right)  $ responds to the submodule of all joint
eigenvectors in $M$ whenever $a$ is a joint eigenvalue of the tuple $x$. The
set of all eigenvalues of $x$ is called \textit{the point spectrum}
$\sigma_{\operatorname{p}}\left(  x,M\right)  $ of $x$. The central result of
Section \ref{Sec1} is the spectral mapping formula
\begin{equation}
\sigma\left(  p\left(  x\right)  ,M\right)  =p\left(  \sigma\left(
x,M\right)  \right)  \label{1}%
\end{equation}
for all tuples $x$ and $p\left(  x\right)  $ (a polynomial tuple in $x$) from
$R$ whenever $R/k$ is an algebra finite extension of an algebraically closed
field $k$ and $M$ is a Noetherian $R$-module. Moreover, $i\left(  x\right)
<\infty$ for an $n$-tuple $x$ from $R$ whenever $R^{\prime}=k\left[  x\right]
\subseteq R$ is an integral extension. In this case, $i\left(  x\right)
=i\left(  y\right)  $ for every $n$-tuple $y$ from the subalgebra $R^{\prime}$
generating the same maximal ideal $\left\langle x\right\rangle \subseteq
R^{\prime}$. It is an analog of the index stability result from analysis.

In Section \ref{Sec2}, we analyze the general case of a scheme $\left(
\mathfrak{X},\mathcal{O}_{\mathfrak{X}}\right)  $ and a module $M$ over the
global sections $R=\Gamma\left(  \mathfrak{X},\mathcal{O}_{\mathfrak{X}%
}\right)  $. The spectrum $\sigma\left(  \mathfrak{X},M\right)  $ of the
$R$-module $M$ over the scheme $\left(  \mathfrak{X},\mathcal{O}%
_{\mathfrak{X}}\right)  $ is defined as the complement to the set of those
$x\in\mathfrak{X}$ such that $\operatorname{Tor}_{i}^{R}\left(  \mathcal{O}%
_{\mathfrak{X}}\left(  U\right)  ,M\right)  =0$, $i\geq0$ for an affine
neighborhood $U$ of $x$. This spectrum was introduced in \cite{DMMJ} as a
scheme-theoretic analog of Putinar spectrum of analytic sheaves \cite{Put}. In
the case of an affine scheme $\mathfrak{X=}\operatorname{Spec}\left(
R\right)  $ the spectrum $\sigma\left(  \mathfrak{X},M\right)  $ is reduced to
the closure of the support $\operatorname{Supp}_{R}\left(  M\right)  $ of the
$R$-module $M$. If $M$ is a finitely generated $R$-module then we come up with
the support $\operatorname{Supp}_{R}\left(  M\right)  $ rather than its
closure. The point spectrum $\sigma_{\operatorname{p}}\left(  \mathfrak{X}%
,M\right)  $ is defined to be the set $\operatorname{Ass}_{R}\left(  M\right)
$ of all associated primes of $M$. The spectral mapping theorem for modules
over schemes was proposed in \cite{DMMJ}. Its affine version states that
\begin{equation}
\sigma\left(  \mathfrak{Y},M\right)  =f\left(  \sigma\left(  \mathfrak{X}%
,M\right)  \right)  ^{-} \label{2}%
\end{equation}
whenever $\mathfrak{X}=\operatorname{Spec}\left(  R\right)  $ and
$\mathfrak{Y}=\operatorname{Spec}\left(  R^{\prime}\right)  $ are affine
schemes, $f=\varphi^{\ast}:\mathfrak{X}\rightarrow\mathfrak{Y}$ is a morphism
responding to a ring map $\varphi:R^{\prime}\rightarrow R$, and $M\in
R$-$\operatorname{mod}$, which is $R^{\prime}$-module through $\varphi$
either. For a ring extension $\iota:R^{\prime}\subseteq R$ and a finitely
generated $R$-module $M$ we obtain that $\iota^{\ast}\left(
\operatorname{Supp}_{R}\left(  M\right)  \right)  $ is dense in
$\operatorname{Supp}_{R^{\prime}}\left(  M\right)  $. In the case of the point
spectrum we prove more valuable equality without the related closure in
(\ref{2}) (being so weak, the closure of Zariski topology covers up too much).
If $\iota:R^{\prime}\subseteq R$ is a ring extension with Noetherian $R$, and
$M\in R$-$\operatorname{mod}$, then (see Theorem \ref{propEx1})
\[
\sigma_{\operatorname{p}}\left(  \mathfrak{X}^{\prime},M\right)  =\iota^{\ast
}\left(  \sigma_{\operatorname{p}}\left(  \mathfrak{X},M\right)  \right)  .
\]
Notice that similar result for the spectrum (or support) is not true (see
below Remark \ref{rem00}). But that is the case if $R^{\prime}\subseteq R$ is
integral and we come with classics Krull-Cohen-Seidenberg Theory \cite[Ch.
14]{AK}, \cite[5.2]{BurCA}.

The spectrum over a scheme and Taylor spectrum linked with each other when
$R=k\left[  x\right]  $ is an algebra finite extension. Namely, $\mathfrak{X=}%
\operatorname{Spec}\left(  R\right)  \subseteq\mathbb{A}_{k}^{n}$ up to a
homeomorphism, $\sigma\left(  \mathfrak{X},M\right)  =\sigma\left(
\mathbb{A}_{k}^{n},M\right)  $, $\sigma_{\operatorname{p}}\left(
\mathfrak{X},M\right)  =\sigma_{\operatorname{p}}\left(  \mathbb{A}_{k}%
^{n},M\right)  $, and $\sigma_{\operatorname{p}}\left(  x,M\right)
=\sigma_{\operatorname{p}}\left(  \mathbb{A}_{k}^{n},M\right)  \cap
\mathbb{A}^{n}$. If $M$ is a Noetherian $R$-module then $\sigma\left(
x,M\right)  =\sigma\left(  \mathbb{A}_{k}^{n},M\right)  \cap\mathbb{A}^{n}$
and it is a nonempty closed subset of $\mathbb{A}^{n}$ whose closure in the
scheme $\mathbb{A}_{k}^{n}$ is reduced to $\sigma\left(  \mathfrak{X}%
,M\right)  $. If $k\left[  p\left(  x\right)  \right]  \subseteq R$ is
integral then%
\[
\sigma_{\operatorname{p}}\left(  p\left(  x\right)  ,M\right)  =p\left(
\sigma_{\operatorname{p}}\left(  x,M\right)  \right)  .
\]
Actually, (\ref{1}) and (\ref{2}) can be driven from the classics whenever
$k\left[  p\left(  x\right)  \right]  \subseteq R$ is integral and $M$ is
Noetherian. But the spectral mapping formula (\ref{1}) is much stronger than
classics, which is true for all tuples. Thus Taylor spectrum with all its
properties have an independent value.

In Section \ref{Sec3} we consider the case of $M=R$ of the coordinate ring $R$
of a variety $Y\subseteq\mathbb{A}^{n}$. In this case, the tuple $x\subseteq
R$ consists of the coordinate functions and $\sigma\left(  x,R\right)  =Y$
whereas $\sigma_{\operatorname{p}}\left(  x,R\right)  =\varnothing$. In the
case of an algebraic set $Y$ the point spectrum $\sigma_{\operatorname{p}%
}\left(  x,R\right)  $ is the set of all isolated points of $Y$. For every
point $a\in\mathbb{A}^{n}$ we have $i\left(  x-a\right)  =0$ due to Serre's
multiplicity formula from the local algebra. For a singular point $a\in Y$ the
dimension of $H_{1}\left(  x-a,R\right)  $ could get high integers depending
on the depth of singularity (see below Lemma \ref{lemMinfi}). The link between
the index and the dimension of $Y$ is obtained through the numerical
polynomial $p:\mathbb{Z}\rightarrow\mathbb{Z}$, $p\left(  r\right)
=\sum_{i=1}^{n}\left(  -1\right)  ^{i}\dim_{k}\left(  \operatorname{Tor}%
_{i}^{P}\left(  P/\left\langle X_{1},\ldots,X_{n}\right\rangle ^{r},R\right)
\right)  $ called the $\operatorname{Tor}$-polynomial, which is reduced to
Hilbert-Samuel polynomial of the localization $R_{\left\langle
x-a\right\rangle }$.

Finally, I wish to thank G. G. Amosov, O. Yu. Aristov, B. Bilich and A. Yu.
Pirkovskii for their interest to the paper, useful discussions and to draw my
attention to the papers \cite{NY1} and \cite{NY2}.

\section{The projection property of spectrum\label{Sec1}}

In the present section we prove the projection property of spectrum based on
homology groups of the Koszul complex. The technical back up is the homology
of commutative rings related to Koszul homology groups.

\subsection{Lemma Bourbaki}

Let $R$ be a (unital) commutative ring, $K_{i}\in R$-$\operatorname{mod}$,
$i=0,1$, and let $\alpha:K_{1}\rightarrow K_{0}$ be an $R$-linear map (module
morphism). The chain complex $0\leftarrow K_{0}\overset{\alpha}{\longleftarrow
}K_{1}$ is denoted by $\mathcal{K}$. If $\mathcal{C}$ is a chain complex
$0\leftarrow C_{0}\overset{d_{0}}{\longleftarrow}C_{1}\overset{d_{1}%
}{\longleftarrow}\cdots$ in $R$-$\operatorname{mod}$ then $\mathcal{K\otimes
}_{R}\mathcal{C}$ is the following complex
\[%
\begin{array}
[c]{ccccccc}
&  & \left(  \mathcal{K\otimes}_{R}\mathcal{C}\right)  _{p-1} &  & \left(
\mathcal{K\otimes}_{R}\mathcal{C}\right)  _{p} &  & \\
&  & \parallel &  & \parallel &  & \\
\cdots & \overset{\partial_{p-2}}{\longleftarrow} & \left(  K_{0}%
\mathcal{\otimes}_{R}C_{p-1}\right)  \oplus\left(  K_{1}\mathcal{\otimes}%
_{R}C_{p-2}\right)  & \overset{\partial_{p-1}}{\longleftarrow} & \left(
K_{0}\mathcal{\otimes}_{R}C_{p}\right)  \oplus\left(  K_{1}\mathcal{\otimes
}_{R}C_{p-1}\right)  & \overset{\partial_{p}}{\longleftarrow} & \cdots
\end{array}
\]
with the differential
\begin{equation}
\partial_{p-1}\left(  z_{0,p},z_{1,p-1}\right)  =\left(  \left(  1\otimes
d_{p-1}\right)  \left(  z_{0,p}\right)  +\left(  -1\right)  ^{p-1}\left(
\alpha\otimes1\right)  \left(  z_{1,p-1}\right)  ,\left(  1\otimes
d_{p-2}\right)  \left(  z_{1,p-1}\right)  \right)  , \label{dif}%
\end{equation}
where $z_{i,j}\in K_{i}\mathcal{\otimes}_{R}C_{j}$. The canonical embedding
$i_{p}$ and projection $\pi_{p}$ define the exact sequence%
\[
0\leftarrow K_{1}\mathcal{\otimes}_{R}C_{p-1}\overset{\pi_{p}}{\longleftarrow
}\left(  K_{0}\mathcal{\otimes}_{R}C_{p}\right)  \oplus\left(  K_{1}%
\mathcal{\otimes}_{R}C_{p-1}\right)  \overset{i_{p}}{\longleftarrow}%
K_{0}\mathcal{\otimes}_{R}C_{p}\leftarrow0
\]
which splits. Since the diagram
\[%
\begin{array}
[c]{ccccccc}
&  & 0 &  & 0 &  & \\
&  & \downarrow &  & \downarrow &  & \\
\cdots & \overset{1\otimes d_{p-2}}{\longleftarrow} & K_{0}\mathcal{\otimes
}_{R}C_{p-1} & \overset{1\otimes d_{p-1}}{\longleftarrow} & K_{0}%
\mathcal{\otimes}_{R}C_{p} & \overset{1\otimes d_{p}}{\longleftarrow} &
\cdots\\
&  & \downarrow^{i_{p-1}} &  & \downarrow^{i_{p}} &  & \\
\cdots & \overset{\partial_{p-2}}{\longleftarrow} & \left(  \mathcal{K\otimes
}_{R}\mathcal{C}\right)  _{p-1} & \overset{\partial_{p-1}}{\longleftarrow} &
\left(  \mathcal{K\otimes}_{R}\mathcal{C}\right)  _{p} & \overset{\partial
_{p}}{\longleftarrow} & \cdots\\
&  & \downarrow^{\pi_{p-1}} &  & \downarrow^{\pi_{p}} &  & \\
\cdots & \overset{1\otimes d_{p-3}}{\longleftarrow} & K_{1}\mathcal{\otimes
}_{R}C_{p-2} & \overset{1\otimes d_{p-2}}{\longleftarrow} & K_{1}%
\mathcal{\otimes}_{R}C_{p-1} & \overset{1\otimes d_{p-1}}{\longleftarrow} &
\cdots\\
&  & \downarrow &  & \downarrow &  & \\
&  & 0 &  & 0 &  &
\end{array}
\]
commutes, there is an exact sequence
\[
0\leftarrow K_{1}\mathcal{\otimes}_{R}\mathcal{C}\overset{\pi}{\longleftarrow
}\mathcal{K\otimes}_{R}\mathcal{C}\overset{i}{\longleftarrow}K_{0}%
\mathcal{\otimes}_{R}\mathcal{C}\leftarrow0
\]
of complexes with $\deg\left(  \pi\right)  =-1$ and $\deg\left(  i\right)
=0$. The sequence in turn generates a long exact sequence of homology groups
\begin{equation}
\cdots\longleftarrow H_{p-1}\left(  K_{0}\mathcal{\otimes}_{R}\mathcal{C}%
\right)  \overset{\delta_{p-1}}{\longleftarrow}H_{p-1}\left(  K_{1}%
\mathcal{\otimes}_{R}\mathcal{C}\right)  \leftarrow H_{p}\left(
\mathcal{K\otimes}_{R}\mathcal{C}\right)  \leftarrow H_{p}\left(
K_{0}\mathcal{\otimes}_{R}\mathcal{C}\right)  \overset{\delta_{p}%
}{\longleftarrow}\cdots\label{h}%
\end{equation}
with the connecting morphisms $\delta_{p}$, $p\geq0$. If $\mathcal{K}$ is flat
(that is, $K_{i}$ are flat modules) then $H_{p}\left(  K_{i}\mathcal{\otimes
}_{R}\mathcal{C}\right)  =K_{i}\otimes_{R}H_{p}\left(  \mathcal{C}\right)  $
for all $i$ and $p$.

\begin{lemma}
\label{lemB1}The morphism $\delta_{p-1}:H_{p-1}\left(  K_{1}\mathcal{\otimes
}_{R}\mathcal{C}\right)  \rightarrow H_{p-1}\left(  K_{0}\mathcal{\otimes}%
_{R}\mathcal{C}\right)  $ is acting by the rule $\delta_{p-1}\left(
\omega^{\sim}\right)  =\left(  -1\right)  ^{p-1}\left(  \left(  \alpha
\otimes1\right)  \omega\right)  ^{\sim}$ for all $\omega\in\ker\left(
1\otimes d_{p-2}\right)  $, that is, $\delta_{p-1}=\left(  -1\right)
^{p-1}\left(  \alpha\otimes1\right)  ^{\sim}$ for all $p$. In the case of a
flat complex $\mathcal{K}$ the morphism $\delta_{p-1}$ is reduced to the
following morphism $\delta_{p-1}:K_{1}\otimes_{R}H_{p-1}\left(  \mathcal{C}%
\right)  \rightarrow K_{0}\otimes_{R}H_{p-1}\left(  \mathcal{C}\right)  $,
$\delta_{p-1}=\left(  -1\right)  ^{p-1}\alpha\otimes1$.
\end{lemma}

\begin{proof}
Take $\omega\in\ker\left(  1\otimes d_{p-2}\right)  $. Then $\left(
0,\omega\right)  \in\left(  K_{0}\mathcal{\otimes}_{R}C_{p}\right)
\oplus\left(  K_{1}\mathcal{\otimes}_{R}C_{p-1}\right)  =\left(
\mathcal{K\otimes}_{R}\mathcal{C}\right)  _{p}$ and
\[
\partial_{p-1}\left(  0,\omega\right)  =\left(  \left(  -1\right)
^{p-1}\left(  \alpha\otimes1\right)  \left(  \omega\right)  ,0\right)
=i_{p-1}\left(  \left(  -1\right)  ^{p-1}\left(  \alpha\otimes1\right)
\left(  \omega\right)  \right)
\]
(see (\ref{dif})). Hence $\left(  -1\right)  ^{p-1}\left(  \alpha
\otimes1\right)  \left(  \omega\right)  \in\ker\left(  1\otimes d_{p-2}%
\right)  $ and $\delta_{p-1}\left(  \omega^{\sim}\right)  =\left(  -1\right)
^{p-1}\left(  \left(  \alpha\otimes1\right)  \omega\right)  ^{\sim}$.

Finally, if $\mathcal{K}$ is flat and $\omega=x_{1}\otimes y_{p-1}$ with
$x_{1}\in K_{1}$, $y_{p-1}\in\ker\left(  d_{p-2}\right)  $ then
\[
\delta_{p-1}\left(  x_{1}\otimes y_{p-1}^{\sim}\right)  =\delta_{p-1}\left(
\omega^{\sim}\right)  =\left(  -1\right)  ^{p-1}\left(  \alpha\left(
x_{1}\right)  \otimes y_{p-1}\right)  ^{\sim}=\left(  -1\right)  ^{p-1}%
\alpha\left(  x_{1}\right)  \otimes y_{p-1}^{\sim},
\]
that is, $\delta_{p-1}=\left(  -1\right)  ^{p-1}\alpha\otimes1$.
\end{proof}

The following key assertion is taken from \cite[9.5, Lemma 3 ]{BurHA}.

\begin{lemma}
\label{lemB2}(N. Bourbaki) If $\alpha:K_{1}\rightarrow K_{0}$ is a morphism of
flat $R$-modules and $\mathcal{C}$ is a complex in $R$-$\operatorname{mod}$
then the long homology sequence generates the following short exact sequences
\[
0\leftarrow\ker\left(  \alpha\right)  \otimes_{R}H_{p-1}\left(  \mathcal{C}%
\right)  \longleftarrow H_{p}\left(  \mathcal{K\otimes}_{R}\mathcal{C}\right)
\longleftarrow\operatorname{coker}\left(  \alpha\right)  \otimes_{R}%
H_{p}\left(  \mathcal{C}\right)  \leftarrow0
\]
of $R$-modules, $p\geq1$.
\end{lemma}

\begin{proof}
Using the long homology sequence (\ref{h}) and Lemma \ref{lemB1}, we obtain
the following exact sequence
\[
K_{0}\otimes_{R}H_{p-1}\left(  \mathcal{C}\right)  \overset{\left(  -1\right)
^{p-1}\alpha\otimes1}{\longleftarrow}K_{1}\otimes_{R}H_{p-1}\left(
\mathcal{C}\right)  \leftarrow H_{p}\left(  \mathcal{K\otimes}_{R}%
\mathcal{C}\right)  \leftarrow K_{0}\otimes_{R}H_{p-1}\left(  \mathcal{C}%
\right)  \overset{\left(  -1\right)  ^{p}\alpha\otimes1}{\longleftarrow}%
K_{1}\otimes_{R}H_{p}\left(  \mathcal{C}\right)
\]
of $R$-modules. Using again the flatness, we deduce that $\ker\left(  \left(
-1\right)  ^{p-1}\alpha\otimes1\right)  =\ker\left(  \alpha\right)
\otimes_{R}H_{p-1}\left(  \mathcal{C}\right)  $ and $\operatorname{coker}%
\left(  \left(  -1\right)  ^{p}\alpha\otimes1\right)  =\operatorname{coker}%
\left(  \alpha\right)  \otimes_{R}H_{p}\left(  \mathcal{C}\right)  $. The rest
is clear.
\end{proof}

As a practical use of Lemma \ref{lemB2}, let us consider the case of
$K_{0}=K_{1}=R$ and $\alpha=x$ is an element of the ring $R$, which is acting
over all $R$-modules as a multiplication operator. Moreover, $K_{i}\otimes
_{R}H_{p}\left(  \mathcal{C}\right)  =H_{p}\left(  \mathcal{C}\right)  $ and
$\alpha\otimes1$ is the same action $x:H_{p}\left(  \mathcal{C}\right)
\rightarrow H_{p}\left(  \mathcal{C}\right)  $ of $x$ over the $R$-module
$H_{p}\left(  \mathcal{C}\right)  $ denoted by $x|H_{p}\left(  \mathcal{C}%
\right)  $. The complex $\mathcal{K\otimes}_{R}\mathcal{C}$ is reduced to the
cone $\operatorname{Con}\left(  x,\mathcal{C}\right)  $ of the morphism
$x:\mathcal{C}\longrightarrow\mathcal{C}$. Namely, $\operatorname{Con}\left(
x,\mathcal{C}\right)  $ is the following complex
\[
\cdots\longleftarrow C_{p-1}\oplus C_{p-2}\overset{\partial_{p-1}%
}{\longleftarrow}C_{p}\oplus C_{p-1}\leftarrow\cdots
\]
with the differential $\partial_{p-1}\left(  c_{p},c_{p-1}\right)  =\left(
d_{p-1}\left(  c_{p}\right)  +\left(  -1\right)  ^{p-1}xc_{p-1},d_{p-2}\left(
c_{p-1}\right)  \right)  $.

\begin{corollary}
\label{corB}If $x\in R$ and $\mathcal{C}$ is a complex of $R$-modules then the
following canonical sequences
\[
0\leftarrow\ker\left(  x|H_{p-1}\left(  \mathcal{C}\right)  \right)
\longleftarrow H_{p}\left(  x,\mathcal{C}\right)  \longleftarrow
\operatorname{coker}\left(  x|H_{p}\left(  \mathcal{C}\right)  \right)
\leftarrow0
\]
are exact, where $H_{p}\left(  x,\mathcal{C}\right)  $ are homology groups of
the cone $\operatorname{Con}\left(  x,\mathcal{C}\right)  $.
\end{corollary}

\begin{proof}
Since $\ker\left(  \alpha\right)  \otimes_{R}H_{p-1}\left(  \mathcal{C}%
\right)  =\ker\left(  \alpha\otimes1\right)  =\ker\left(  x|H_{p-1}\left(
\mathcal{C}\right)  \right)  $ and $\operatorname{coker}\left(  \alpha\right)
\otimes_{R}H_{p}\left(  \mathcal{C}\right)  =\operatorname{coker}\left(
\alpha\otimes1\right)  =\operatorname{coker}\left(  x|H_{p}\left(
\mathcal{C}\right)  \right)  $, the result follows from Lemma \ref{lemB2}.
\end{proof}

In particular, $H_{p}\left(  x,\mathcal{C}\right)  =0$ iff $x:H_{p-1}\left(
\mathcal{C}\right)  \rightarrow H_{p-1}\left(  \mathcal{C}\right)  $ is an
injection and $x:H_{p}\left(  \mathcal{C}\right)  \rightarrow H_{p}\left(
\mathcal{C}\right)  $ is a surjection.

\subsection{Koszul homology groups}

Fix a field $k$, $R/k$ a commutative $k$-algebra, $x=\left(  x_{1}%
,\ldots,x_{n}\right)  $ an $n$-tuple in $R$, and let $M$ be an $R$-module,
which in turn is a $k$-vector space with an $n$-tuple $x$ of mutually
commuting linear transformations acting on it. If it is necessary one can
replace $R$ by an algebra finite extension $k\left[  x\right]  $ of the field
$k$. For every point $a$ from the affine space $\mathbb{A}^{n}$ we have the
tuple $x-a=\left(  x_{1}-a_{1},\ldots,x_{n}-a_{n}\right)  $ on $M$, which is
turn defines the Koszul complex $\operatorname{Kos}\left(  x-a,M\right)  $:%
\[
0\leftarrow M\overset{\partial_{0}}{\longleftarrow}M\otimes_{k}k^{n}%
\overset{\partial_{1}}{\longleftarrow}\cdots\overset{\partial_{p-2}%
}{\longleftarrow}M\otimes_{k}\wedge^{p-1}k^{n}\overset{\partial_{p-1}%
}{\longleftarrow}M\otimes_{k}\wedge^{p}k^{n}\overset{\partial_{p}%
}{\longleftarrow}\cdots\overset{\partial_{n-1}}{\longleftarrow}M\leftarrow0
\]
in $R$-$\operatorname{mod}$ with the differential
\[
\partial_{p-1}\left(  m\otimes v_{p}\right)  =\sum_{s=1}^{p}\left(  -1\right)
^{s+1}\left(  x_{i_{s}}-a_{i_{s}}\right)  m\otimes v_{p,s},
\]
where $m\in M$, $v_{p}=e_{i_{1}}\wedge\ldots\wedge e_{i_{p}}$, $v_{p,s}%
=e_{i_{1}}\wedge\ldots\wedge\widehat{e_{i_{s}}}\wedge\ldots\wedge e_{i_{p}}$
(the notation $\widehat{e_{i_{s}}}$ stands for skipping $e_{i_{s}}$ from the
$p$-vector) and $\left(  e_{1},\ldots,e_{n}\right)  $ is the standard basis
for $k^{n}$. The homology groups of $\operatorname{Kos}\left(  x-a,M\right)  $
are denoted by $H_{p}\left(  x-a,M\right)  $, $p\geq0$, which are $R$-modules.
We put
\[
i\left(  x\right)  =\sum_{p=0}^{n}\left(  -1\right)  ^{p+1}\dim_{k}\left(
H_{p}\left(  x,M\right)  \right)
\]
to be \textit{the index of the tuple} $x$ whenever $\dim_{k}\left(
H_{p}\left(  x,M\right)  \right)  <\infty$ for all $p$. In the latter case we
write $i\left(  x\right)  <\infty$. Recall that the index $i\left(  t\right)
$ of a single $k$-linear transformation $t:M\rightarrow M$ is given by
$i\left(  t\right)  =\dim_{k}\ker\left(  t\right)  -\dim_{k}%
\operatorname{coker}\left(  t\right)  $ whenever both dimensions are finite.
If $\dim_{k}\left(  M\right)  <\infty$ then $i\left(  t\right)  =0$ for every
$t$. The index of tuples of bounded linear operators acting on a Banach space
is a subject of Fredholm theory from analysis \cite{Gur}, \cite{FainF}.

\begin{lemma}
\label{lemB3}If $x^{\prime}=\left(  x_{1},\ldots,x_{n-1}\right)  $ then there
is an exact sequence
\[
0\leftarrow\ker\left(  x_{n}|H_{p-1}\left(  x^{\prime},M\right)  \right)
\longleftarrow H_{p}\left(  x,M\right)  \longleftarrow\operatorname{coker}%
\left(  x_{n}|H_{p}\left(  x^{\prime},M\right)  \right)  \leftarrow0
\]
of $R$-modules. In particular, if $H_{p}\left(  x,M\right)  \neq0$ then
$H_{p}\left(  x^{\prime},M\right)  \neq0$ or $H_{p-1}\left(  x^{\prime
},M\right)  \neq0$. Moreover, if $i\left(  x\right)  <\infty$ then $i\left(
x_{n}|H_{p}\left(  x^{\prime},M\right)  \right)  <\infty$ for every $p$, and
$i\left(  x\right)  =\sum_{p=0}^{n-1}\left(  -1\right)  ^{p}i\left(
x_{n}|H_{p}\left(  x^{\prime},M\right)  \right)  $. If $i\left(  x^{\prime
}\right)  <\infty$ then $i\left(  x\right)  =0$.
\end{lemma}

\begin{proof}
Put $\mathcal{C=}\operatorname{Kos}\left(  x^{\prime},M\right)  $ to be a
complex of $R$-modules. Then $x_{n}$ defines an endomorphism of $\mathcal{C}$,
which in turn generates the cone $\operatorname{Con}\left(  x_{n}%
,\mathcal{C}\right)  $. The fact $\operatorname{Con}\left(  x_{n}%
,\mathcal{C}\right)  =\operatorname{Kos}\left(  x,M\right)  $ is know even in
the noncommutative case (see \cite[Lemma 1.5]{Fain}). One needs to use the
vector space isomorphism $\wedge^{p}k^{n-1}\oplus\wedge^{p-1}k^{n-1}%
\rightarrow\wedge^{p}k^{n}$, $\left(  v_{1},v_{2}\right)  \mapsto v_{1}%
+v_{2}\wedge e_{n}$, which generates $P$-module isomorphism $C_{p}\oplus
C_{p-1}\rightarrow M\otimes_{k}\wedge^{p}k^{n}$ for every $p$. If
$m_{p}=m\otimes v_{p}\in M\otimes_{k}\wedge^{p}k^{n}$ with $j_{p}\neq n$, then
$m_{p}$ is identified with an element $\left(  c_{p},0\right)  $ of
$C_{p}\oplus C_{p-1}$ and $\partial\left(  c_{p},0\right)  =\left(
\partial^{\prime}\left(  c_{p}\right)  ,0\right)  $, where $\partial^{\prime}$
is the differential of $\mathcal{C}$. If $j_{p}=n$ then $m_{p}$ is identified
with $\left(  0,c_{p-1}\right)  $, where $c_{p-1}=m\otimes v_{p-1}\in C_{p-1}%
$, $v_{p-1}\wedge e_{n}=v_{p}$. Moreover,
\begin{align*}
\partial\left(  0,c_{p-1}\right)   &  =\partial\left(  m_{p}\right)
=\sum_{s=1}^{p-1}\left(  -1\right)  ^{s+1}x_{i_{s}}m\otimes v_{p-1,s}\wedge
e_{n}+\left(  -1\right)  ^{p+1}x_{n}m\otimes v_{p-1}\\
&  =\partial^{\prime}\left(  c_{p-1}\right)  \wedge e_{n}+\left(  -1\right)
^{p+1}x_{n}c_{p-1}=\left(  \left(  -1\right)  ^{p+1}x_{n}c_{p-1}%
,\partial^{\prime}\left(  c_{p-1}\right)  \right)  .
\end{align*}
Hence $\partial\left(  c_{p},c_{p-1}\right)  =\left(  \partial^{\prime}\left(
c_{p}\right)  +\left(  -1\right)  ^{p-1}x_{n}c_{p-1},\partial^{\prime}\left(
c_{p-1}\right)  \right)  $ for all $\left(  c_{p},c_{p-1}\right)  \in
C_{p}\oplus C_{p-1}$, which means that $\partial$ is the morphism of
$\operatorname{Con}\left(  x_{n},\mathcal{C}\right)  $. It remains to use
Corollary \ref{corB}.

Finally, consider the case of $i\left(  x\right)  <\infty$. Since $\dim
_{k}\left(  H_{p}\left(  x,M\right)  \right)  <\infty$ for every $p$, it
follows that $\dim_{k}\ker\left(  x_{n}|H_{p-1}\left(  x^{\prime},M\right)
\right)  <\infty$ and $\dim_{k}\operatorname{coker}\left(  x_{n}|H_{p}\left(
x^{\prime},M\right)  \right)  <\infty$ for every $p$, which means that
$i\left(  x_{n}|H_{p}\left(  x^{\prime},M\right)  \right)  <\infty$ for every
$p$. Then
\begin{align*}
i\left(  x\right)   &  =\sum_{p=0}^{n}\left(  -1\right)  ^{p+1}\left(
\dim_{k}\ker\left(  x_{n}|H_{p-1}\left(  x^{\prime},M\right)  \right)
+\dim_{k}\operatorname{coker}\left(  x_{n}|H_{p}\left(  x^{\prime},M\right)
\right)  \right) \\
&  =\sum_{p=0}^{n}\left(  -1\right)  ^{p+1}\dim_{k}\operatorname{coker}\left(
x_{n}|H_{p}\left(  x^{\prime},M\right)  \right)  +\left(  -1\right)  ^{p}%
\dim_{k}\ker\left(  x_{n}|H_{p}\left(  x^{\prime},M\right)  \right) \\
&  =\sum_{p=0}^{n}\left(  -1\right)  ^{p}i\left(  x_{n}|H_{p}\left(
x^{\prime},M\right)  \right)  ,
\end{align*}
that is, $i\left(  x\right)  =\sum_{p=0}^{n-1}\left(  -1\right)  ^{p}i\left(
x_{n}|H_{p}\left(  x^{\prime},M\right)  \right)  $ (for $H_{n}\left(
x^{\prime},M\right)  =0$). In particular, if $i\left(  x^{\prime}\right)
<\infty$ then $\dim_{k}\left(  H_{p}\left(  x^{\prime},M\right)  \right)
<\infty$ and $i\left(  x_{n}|H_{p}\left(  x^{\prime},M\right)  \right)  =0$
for all $p$. Therefore $i\left(  x\right)  =0$.
\end{proof}

The Taylor spectrum $\sigma\left(  x,M\right)  $ of the operator tuple $x$ on
$M$ is defined to be a subset of $\mathbb{A}^{n}$ of those $a$ such that
$\operatorname{Kos}\left(  x-a,M\right)  $ is not exact. The point spectrum
$\sigma_{\operatorname{p}}\left(  x,M\right)  $ consists of those
$a\in\mathbb{A}^{n}$ such that $H_{n}\left(  x-a,M\right)  \neq0$. Taking into
account $H_{n}\left(  x-a,M\right)  =\ker\partial_{n-1}$, we conclude that
$a\in\sigma_{\operatorname{p}}\left(  x,M\right)  $ iff there is a nonzero
$m\in M$ such that $x_{i}m=a_{i}m$ for all $i$, that is, $a$ turns out to be a
joint eigenvalue and $m$ is the related joint eigenvector (see \cite{CTag}).

\begin{corollary}
\label{corProj1}Let $R/k$ be an algebra, $M\in R$-mod, $x$ an $n$-tuple from
$R$. Then $\pi\left(  \sigma\left(  x,M\right)  \right)  \subseteq
\sigma\left(  x^{\prime},M\right)  $, where $\pi:\mathbb{A}^{n}\rightarrow
\mathbb{A}^{n-1}$ is the canonical projection onto first $n-1$ coordinates.
\end{corollary}

\begin{proof}
If $a\in\sigma\left(  x,M\right)  $ then $H_{p}\left(  x-a,M\right)  \neq0$
for some $p$. By Lemma \ref{lemB3}, either $H_{p}\left(  x^{\prime}-a^{\prime
},M\right)  \neq0$ or $H_{p-1}\left(  x^{\prime}-a^{\prime},M\right)  \neq0$,
where $a^{\prime}=\pi\left(  a\right)  $. Whence $a^{\prime}\in\sigma\left(
x^{\prime},M\right)  $.
\end{proof}

Put $M_{i}=M/\left\langle x_{1},\ldots,x_{i}\right\rangle M$, $i\geq1$. Recall
\cite[9.6]{BurHA}, \cite[Ch. 23]{AK} that if $M_{n}\neq\left\{  0\right\}  $
and $x_{i}\notin\operatorname{zdiv}\left(  M_{i-1}\right)  $ (zero-divisors)
for all $i$, then $x$ is called a regular sequence. In the case of a nonzero
Noetherian module $M$ and $x\subseteq\operatorname{rad}\left(  M\right)  $ the
condition $M_{n}\neq\left\{  0\right\}  $ is satisfied automatically. Indeed,
otherwise $\left\langle x\right\rangle M=M$ implies that $\left(  1+a\right)
M=\left\{  0\right\}  $ for some $a\in\left\langle x\right\rangle $ (Nakayama
Lemma), which means that $1+a\in\operatorname{Ann}\left(  M\right)
\subseteq\operatorname{rad}\left(  M\right)  $ or $1\in\operatorname{rad}%
\left(  M\right)  $, a contradiction.

\begin{corollary}
\label{corRH}Let $R/k$ be a $k$-algebra, $M\in R$-$\operatorname{mod}$, and
let $x=\left(  x_{1},\ldots,x_{n}\right)  \subseteq R$ be a regular sequence
for $M$. Then $H_{i}\left(  x,M\right)  =0$ for all $i>0$.
\end{corollary}

\begin{proof}
We proceed by induction on $n$. If $n=1$ the result follows. In the general
case, we put $x^{\prime}=\left(  x_{1},\ldots,x_{n-1}\right)  $, which is
regular either. By induction hypothesis, $H_{i}\left(  x^{\prime},M\right)
=0$ for all $i>0$. By Lemma \ref{lemB3}, $H_{p}\left(  x,M\right)  =0$ for all
$p\geq2$. Moreover, $H_{1}\left(  x,M\right)  =\ker\left(  x_{n}|H_{0}\left(
x^{\prime},M\right)  \right)  $ and $H_{0}\left(  x^{\prime},M\right)
=M/\left\langle x^{\prime}\right\rangle M$. But $x_{n}:M/\left\langle
x^{\prime}\right\rangle M\rightarrow M/\left\langle x^{\prime}\right\rangle M$
is injective by assumption, therefore $H_{1}\left(  x,M\right)  =0$. Whence
$H_{i}\left(  x,M\right)  =0$ for all $i>0$.
\end{proof}

In the case of a local ring $R$ and a Noetherian module $M$ the statement of
Corollary \ref{corRH} turns out to be a criteria for regularity of a tuple
(see \cite{Dkth} for regularity in the noncommutative setting). One can easily
seen that the tuple $X=\left(  X_{1},\ldots,X_{n}\right)  $ is a regular
sequence for the polynomial algebra $P=k\left[  X\right]  $ to be a
$P$-module. By Corollary \ref{corRH}, $H_{i}\left(  X,P\right)  =0$ for all
$i>0$, and $H_{0}\left(  X,P\right)  =P/\mathfrak{t}=k$, where $\mathfrak{t=}%
\left\langle X\right\rangle $ is the maximal ideal of $P$. Thus $0\leftarrow
k\longleftarrow\operatorname{Kos}\left(  X,P\right)  $ is exact or
$\operatorname{Kos}\left(  X,P\right)  $ provides a free $P$-module resolution
for the $P$-module $k$. In particular, for every $P$-module $M$ we obtain that
$\operatorname{Kos}\left(  X,P\right)  \otimes_{P}M=\operatorname{Kos}\left(
x,M\right)  $ and $\operatorname{Tor}_{i}^{P}\left(  k,M\right)  =H_{i}\left(
\operatorname{Kos}\left(  X,P\right)  \otimes_{P}M\right)  =H_{i}\left(
x,M\right)  $, where $x$ is the $n$-tuple of linear transformations on $M$
given by $X$-action on $M$ (see also below Subsection \ref{SubsecAVS}).

Let $R/k$ be a $k$-algebra, $R^{\prime}=k\left[  x\right]  \subseteq R$ a
subalgebra with an $n$-tuple $x$. An $m$-tuple $y$ from $R^{\prime}$ is said
to be \textit{related to} $x$ if $x$ and $y$ generate the same (maximal) ideal
$\left\langle x\right\rangle =\left\langle y\right\rangle $ in $R^{\prime}$.
In particular, $x$ is always related to $x$ itself.

The following assertion is well know \cite{Tay1}. Below we provide its
modified (a bit) version with the detailed proof.

\begin{lemma}
\label{lemB4}If $R/k$ is an algebra, $x$ an $n$-tuple in $R$ and $y$ is an
$m$-tuple related to $x$, then the $k\left[  x\right]  $-module structure of
$H_{p}\left(  y,M\right)  $ is reduced to its $k$-vector space structure for
every $p\geq0$. In particular, the action of a polynomial $p\left(  x\right)
\in R$ on $H_{p}\left(  y,M\right)  $ is just the multiplication by $p\left(
0\right)  $ operator.
\end{lemma}

\begin{proof}
For every $i$ let us define the $P$-linear map
\[
\gamma_{i}:M\otimes_{k}\wedge^{p}k^{m}\rightarrow M\otimes_{k}\wedge
^{p+1}k^{m},\quad\gamma_{i}\left(  u_{p}\right)  =u\otimes\left(  e_{i}\wedge
v_{p}\right)  ,
\]
where $u_{p}=u\otimes v_{p}$, $e_{i}\wedge v_{p}=\left(  -1\right)
^{s-1}e_{i_{1}}\wedge\ldots\wedge e_{i}\wedge\ldots\wedge e_{i_{p}}=\left(
-1\right)  ^{s-1}v_{p+1}$ whenever $v_{p}=e_{i_{1}}\wedge\ldots\wedge
e_{i_{p}}$ and $i_{1}<\cdots<i_{s-1}<i<i_{s}<\cdots<i_{p}$ for some $s$. Then
\begin{align*}
\partial_{p}\gamma_{i}\left(  u_{p}\right)   &  =\left(  -1\right)  ^{s-1}%
\sum_{k<s}\left(  -1\right)  ^{k+1}y_{i_{k}}u\otimes v_{p+1,k}+x_{i}u\otimes
v_{p}+\left(  -1\right)  ^{s-1}\sum_{k\geq s}\left(  -1\right)  ^{k}y_{i_{k}%
}u\otimes v_{p+1,k}\\
&  =-\sum_{k<s}\left(  -1\right)  ^{k+1}y_{i_{k}}u\otimes e_{i}\wedge
v_{p,k}+\sum_{k\geq s}\left(  -1\right)  ^{k}y_{i_{k}}u\otimes e_{i}\wedge
v_{p,k}+x_{i}m\otimes v_{p}\\
&  =-\gamma_{i}\partial_{p-1}\left(  u_{p}\right)  +y_{i}u_{p}.
\end{align*}
If $u_{p}\in\ker\left(  \partial_{p-1}\right)  $ then $y_{i}m_{p}=\partial
_{p}\gamma_{i}\left(  u_{p}\right)  \in\operatorname{im}\left(  \partial
_{p}\right)  $, which means that the action of $y_{i}$ over $H_{p}\left(
y,M\right)  $ is trivial, that is, $yH_{p}\left(  y,M\right)  =\left\{
0\right\}  $. But $H_{p}\left(  y,M\right)  $ is an $R$-module, therefore
$\left\langle y\right\rangle H_{p}\left(  y,M\right)  =\left\{  0\right\}  $.
Taking into account that $\left\langle x\right\rangle =\left\langle
y\right\rangle $ in $R^{\prime}$, we deduce that $\left\langle x\right\rangle
H_{p}\left(  y,M\right)  =\left\{  0\right\}  $.

Finally, take $p\left(  x\right)  \in R^{\prime}$. Then $p\left(  x\right)
=p\left(  0\right)  +g$, $g\in\left\langle x\right\rangle $ and $p\left(
x\right)  H_{p}\left(  y,M\right)  =p\left(  0\right)  H_{p}\left(
y,M\right)  $.
\end{proof}

Since $R^{\prime}=k\left[  x\right]  =k\left[  x-a\right]  $ and $\left\langle
x-a\right\rangle \subseteq R^{\prime}$ is a maximal ideal, we conclude that
$k\left[  x\right]  $-module structure of $H_{p}\left(  x-a,M\right)  $ is
reduced to its $k$-vector space structure for every $p\geq0$. Actually, the
action of a polynomial $p\left(  x\right)  \in R$ on $H_{p}\left(
x-a,M\right)  $ is just the multiplication by $p\left(  a\right)  $ operator.

Again consider $R^{\prime}=k\left[  x\right]  \subseteq R$ a subalgebra with
an $n$-tuple $x$, and an $m$-tuple $y$ from $R^{\prime}$ related to $x$. If
$p\left(  x\right)  \in R^{\prime}$ then $z=\left(  y,p\left(  x\right)
\right)  $ is an $m+1$-tuple in $R$.

\begin{proposition}
\label{propB1}If $b=\left(  0,\lambda\right)  \in\sigma\left(  z,M\right)  $
for some $\lambda\in k$, then $0\in\sigma\left(  y,M\right)  $ and
$\lambda=p\left(  0\right)  $. Moreover, $\dim_{k}\left(  H_{p}\left(
z-b,M\right)  \right)  =\dim_{k}H_{p-1}\left(  y,M\right)  +\dim_{k}%
H_{p}\left(  y,M\right)  $ for all $p$, and $i\left(  z-b\right)  =0$ whenever
$i\left(  y\right)  <\infty$.
\end{proposition}

\begin{proof}
Suppose $b\in\sigma\left(  z,M\right)  $ and $H_{j}\left(  z-b,M\right)
\neq0$ for some $j$. Using Corollary \ref{corProj1}, we derive that
$0\in\sigma\left(  y,M\right)  $. By Lemma \ref{lemB3}, the following
sequence
\[
0\leftarrow\ker\left(  p\left(  x\right)  -\lambda|H_{j-1}\left(  y,M\right)
\right)  \longleftarrow H_{j}\left(  z-b,M\right)  \longleftarrow
\operatorname{coker}\left(  p\left(  x\right)  -\lambda|H_{j}\left(
y,M\right)  \right)  \leftarrow0
\]
turns out to be exact. By Lemma \ref{lemB4}, $p\left(  x\right)  H_{p}\left(
y,M\right)  =p\left(  0\right)  H_{p}\left(  y,M\right)  $ for every $p\left(
x\right)  \in R^{\prime}$. Thus the actions of $p\left(  x\right)  -\lambda$
over the homology groups $H_{i}\left(  y,M\right)  $ are reduced to the
constant multiplication operators by $p\left(  0\right)  -\lambda$. If
$\lambda\neq p\left(  0\right)  $ then $p\left(  x\right)  -\lambda$ defines
an invertible action over homology groups. Therefore the corners of the exact
sequence are vanishing, which means that $H_{j}\left(  z-b,M\right)  =0$, a
contradiction. Hence $\lambda=p\left(  0\right)  $. Using again Lemma
\ref{lemB3}, for every $p$ we obtain the following
\[
0\leftarrow H_{p-1}\left(  y,M\right)  \longleftarrow H_{p}\left(
z-b,M\right)  \longleftarrow H_{p}\left(  y,M\right)  \leftarrow0
\]
exact sequence. In particular, $\dim_{k}\left(  H_{p}\left(  z-b,M\right)
\right)  =\dim_{k}H_{p-1}\left(  y,M\right)  +\dim_{k}H_{p}\left(  y,M\right)
$ whenever $i\left(  y\right)  <\infty$. Finally,
\begin{align*}
i\left(  z-b\right)   &  =\sum_{p=0}^{n+1}\left(  -1\right)  ^{p+1}\dim
_{k}\left(  H_{p}\left(  z-b,M\right)  \right) \\
&  =\sum_{p=0}^{n+1}\left(  -1\right)  ^{p+1}\dim_{k}H_{p-1}\left(
y,M\right)  +\left(  -1\right)  ^{p+1}\dim_{k}H_{p}\left(  y,M\right)  =0,
\end{align*}
that is, $i\left(  z-b\right)  =0$.
\end{proof}

\begin{remark}
\label{remSw}The assertion just proven is equally true for the tuple
$z=\left(  p\left(  x\right)  ,y\right)  $ swapped. Namely, if $b=\left(
\lambda,0\right)  \in\sigma\left(  z,M\right)  $ for some $\lambda\in k$, then
$0\in\sigma\left(  y,M\right)  $ and $\lambda=p\left(  0\right)  $.
\end{remark}

\begin{corollary}
\label{corBB1}If $w=\left(  x,p\left(  x\right)  \right)  $ is an $n+1$-tuple
in $R$ and $c=\left(  a,\lambda\right)  \in\sigma\left(  w,M\right)  $ for
some $a\in\mathbb{A}^{n}$ and $\lambda\in k$, then $a\in\sigma\left(
x,M\right)  $ and $\lambda=p\left(  a\right)  $.
\end{corollary}

\begin{proof}
Put $y=x-a$, $R^{\prime}=k\left[  y\right]  $, $z=\left(  y,q\left(  y\right)
\right)  $ and $b=\left(  0,\lambda\right)  $, where $q\left(  y\right)
=p\left(  y+a\right)  $. Then $z-b=\left(  y,q\left(  y\right)  -\lambda
\right)  =w-c$, which means that $b\in\sigma\left(  z,M\right)  $. It remains
to use Proposition \ref{propB1} for $y$ instead of $x$. Then $0\in
\sigma\left(  y,M\right)  $ (or $a\in\sigma\left(  x,M\right)  $) and
$\lambda=q\left(  0\right)  =p\left(  a\right)  $.
\end{proof}

Now let us define the following one-to-one function
\[
f_{n}:\mathbb{Z}_{+}^{n+1}\rightarrow\mathbb{Z}_{+}^{n+2},\tau_{n}\left(
d_{0},\ldots,d_{n}\right)  =\left(  d_{0},d_{0}+d_{1},d_{1}+d_{2}%
,\ldots,d_{n-1}+d_{n},d_{n}\right)
\]
that generates Fibonacci numbers by increasing $n+1$-tuples of nonnegative
integers. If $x$ is an $n$-tuple from $R$ with $i\left(  x\right)  <\infty$
then $d\left(  x\right)  =\left(  d_{0},\ldots,d_{n}\right)  \in\mathbb{Z}%
_{+}^{n+1}$ is a tuple of dimensions $d_{p}=\dim_{k}H_{p}\left(  x,M\right)
$, $0\leq p\leq n$.

\begin{corollary}
\label{corind1}Let $R/k$ be an algebra, $R^{\prime}=k\left[  x\right]
\subseteq R$ a subalgebra with an $n$-tuple $x$, and related to $x$ an
$m$-tuple $y$. If $\left(  0,\lambda\right)  \in\sigma\left(  \left(
x,y\right)  ,M\right)  $ and both $i\left(  x\right)  $ and $i\left(
y\right)  $ are finite, then $\lambda=0$, $d\left(  x\right)  =f_{n-m-1}\ldots
f_{m}\left(  d\left(  y\right)  \right)  $ and $i\left(  x\right)
=\delta_{nm}i\left(  y\right)  $ whenever $m\leq n$.
\end{corollary}

\begin{proof}
Put $z=\left(  x,y\right)  $ to be $n+m$-tuple from $R$. Note that $y$ is a
tuple of polynomials from $R^{\prime}$ such that $y\subseteq\left\langle
x\right\rangle \subseteq R^{\prime}$, that is, $y\left(  0\right)  =0$. By
Corollary \ref{corBB1}, $0\in\sigma\left(  x,M\right)  $ and $\lambda=y\left(
0\right)  =0\in\sigma\left(  y,M\right)  $ (see Corollary \ref{corProj1}).
Thus $0\in\sigma\left(  z,M\right)  $ and $d\left(  z\right)  =f_{n+m-1}\ldots
f_{n}\left(  d\left(  x\right)  \right)  $ thanks Proposition \ref{propB1}. By
symmetry (see Remark \ref{remSw}), we deduce $d\left(  z\right)
=f_{n+m-1}\ldots f_{m}\left(  d\left(  y\right)  \right)  $ either. If $m\leq
n$ then $d\left(  x\right)  =f_{n-m-1}\ldots f_{m}\left(  d\left(  y\right)
\right)  $ due to the property to be a one-to-one function. If $n>m$ then
$i\left(  x\right)  =0$ (see Proposition \ref{propB1}). But if $m=n$ then
$d\left(  x\right)  =d\left(  y\right)  $ and $i\left(  x\right)  =i\left(
y\right)  $.
\end{proof}

\begin{remark}
\label{remIE}If $R^{\prime}=k\left[  x\right]  =k\left[  y\right]  $ for an
$m$-tuple $y\subseteq\left\langle x\right\rangle $ then $y$ is a tuple related
to $x$. Indeed, $x_{i}=q_{i}\left(  y\right)  \in k\left[  y\right]
=R^{\prime}$ for every $i$. It follows that $x_{i}=q_{i}\left(  0\right)
+h_{i}\left(  y\right)  $, $h_{i}\left(  y\right)  \in\left\langle
y\right\rangle \subseteq\left\langle x\right\rangle $. Hence $q_{i}\left(
0\right)  =0$, that is, $x_{i}\in\left\langle y\right\rangle $. Thus
$\left\langle x\right\rangle =\left\langle y\right\rangle $ in $R^{\prime}$.
\end{remark}

\subsection{Noetherian modules and triangular actions}

As above we have an algebra $R/k$ and an $R$-module $M$ with an $n$-tuple $x$
from $R$. If $R$ is Noetherian and $M$ is a finitely generated $R$-module then
it turns out to be Noetherian automatically. Conversely, a Noetherian
$R$-module has a finitely many generators. For brevity we say that $M$ is a
Noetherian $R$-module with an acting tuple $x$ on $M$. Since $M\otimes
_{k}\wedge^{i}k^{n}$ is a Noetherian $R$-module, it follows that so is
$H_{i}\left(  x,M\right)  $ for every $i$.

\begin{lemma}
\label{lemNP0}Let $R=k\left[  x\right]  $ be an algebra finite extension with
an $n$-tuple $x$, $\mathfrak{m=}\left\langle x\right\rangle \subseteq R$ a
maximal ideal, and let $M$ be an $R$-module. Then $H_{i}\left(  x,M\right)
_{\mathfrak{m}}=H_{i}\left(  x,M\right)  $ for all $i\geq0$. If $M$ is
Noetherian and $H_{p}\left(  x,M\right)  =0$ for some $p$, then $H_{j}\left(
x,M\right)  =0$ for all $j\geq p$.
\end{lemma}

\begin{proof}
By Lemma \ref{lemB4}, the action of every $p\left(  x\right)  \notin%
\mathfrak{m}$ on $H_{i}\left(  x,M\right)  $ is reduced to the diagonal
operator $p\left(  0\right)  $, which is invertible. Therefore $H_{i}\left(
x,M\right)  _{\mathfrak{m}}=H_{i}\left(  x,M\right)  $ (see \cite[12.1,
12.4]{AK}).

Now assume that $M$ is a Noetherian $R$-module and prove that $H_{p}\left(
x,M\right)  =0$ implies that $H_{j}\left(  x,M\right)  =0$ for all $j\geq p$.
If $R$ is a local $k$-algebra with its maximal ideal $\mathfrak{m}$ and
$x\subseteq\mathfrak{m}$ then the result follows thanks to Lemma \ref{lemB3}.
Namely, we proceed by induction on $n$ for a nonzero Noetherian $R$-module
$M$. Since $M/x_{1}M\neq0$ (Nakayama lemma), the result follows in the case of
$n=1$. Put $x^{\prime}=\left(  x_{1},\ldots,x_{n-1}\right)  $ and suppose
$H_{p}\left(  x,M\right)  =0$. By Lemma \ref{lemB3}, we obtain that
$\operatorname{coker}\left(  x_{n}|H_{p}\left(  x^{\prime},M\right)  \right)
=0$. Since $H_{p}\left(  x^{\prime},M\right)  $ is a Noetherian $R$-module and
$x_{n}\in\operatorname{rad}R$, it follows that $H_{p}\left(  x^{\prime
},M\right)  /x_{n}H_{p}\left(  x^{\prime},M\right)  \neq0$ whenever
$H_{p}\left(  x^{\prime},M\right)  \neq0$ (Nakayama lemma). Hence
$H_{p}\left(  x^{\prime},M\right)  =0$, and $H_{j}\left(  x^{\prime},M\right)
=0$, $j\geq p$ by induction hypothesis. Using again Lemma \ref{lemB3}, we
obtain the exact sequence
\[
0\leftarrow\ker\left(  x_{n}|H_{j}\left(  x^{\prime},M\right)  \right)
\longleftarrow H_{j+1}\left(  x,M\right)  \longleftarrow\operatorname{coker}%
\left(  x_{n}|H_{j+1}\left(  x^{\prime},M\right)  \right)  \leftarrow0
\]
for every $j\geq p$. It follows that $H_{j+1}\left(  x,M\right)  =0$ for all
$j\geq p$. Optionally one can use \cite[Theorem 17.6]{E}.

In the general case, $M_{\mathfrak{m}}$ is a Noetherian $R_{\mathfrak{m}}%
$-module and $H_{i}\left(  x,M\right)  =H_{i}\left(  x,M\right)
_{\mathfrak{m}}=H_{i}\left(  x/1,M_{\mathfrak{m}}\right)  $ (the localization
is an exact functor) $i\geq0$, where $x/1$ is the tuple representing $\left(
x_{1}/1,\ldots,x_{n}/1\right)  \subseteq\mathfrak{m}R_{\mathfrak{m}%
}=\operatorname{rad}R_{\mathfrak{m}}$. If $H_{p}\left(  x,M\right)  =0$ then
$H_{p}\left(  x/1,M_{\mathfrak{m}}\right)  =0$ and we come up with the local
$k$-algebra case $R_{\mathfrak{m}}$. Based on the fact just proven, we
conclude that $H_{j}\left(  x/1,M_{\mathfrak{m}}\right)  =0$ for all $j\geq
p$. Hence $H_{j}\left(  x,M\right)  =0$ for all $j\geq p$.
\end{proof}

Now we analyze tuples $y$ from an algebra finite extension $R/k$ which admit
finite index. Everywhere below we assume that $k$ is an algebraically closed
field. A tuple $y$ whose extension $k\left[  y\right]  \subseteq k\left[
x\right]  $ is integral plays a key role in this manner.

\begin{theorem}
\label{thNP11}Let $k\subseteq R^{\prime}\subseteq R$ be a ring extensions of
the field $k$ such that $R/k$ is an algebra finite extension and $R/R^{\prime
}$ is integral. If $M$ is a Noetherian $R$-module then $i\left(  y\right)
<\infty$ whenever $y$ is a tuple in $R$ generating a maximal ideal in
$R^{\prime}$. In this case, the $R$-module structure on every $H_{p}\left(
y,M\right)  $ is triangularizable (or it is a semilocal $R$-module) whereas
its $R^{\prime}$-module structure is diagonalizable.
\end{theorem}

\begin{proof}
Since $R/R^{\prime}$ is algebra finite and integral, it follows that $R$ is
module finite over $R^{\prime}$ \cite[10.18]{AK}. But $M$ is module finite
over $R$, therefore $M$ is module finite over $R^{\prime}$ (see \cite[10.16]%
{AK}). Moreover, $R^{\prime}/k$ is an algebra finite extension by Artin-Tate
\cite[16.17]{AK}, and it is Noetherian (Hilbert Basis). In particular, $M$ is
a Noetherian $R^{\prime}$-module either, and so are all homology groups
$H_{p}\left(  y,M\right)  $ for all tuples $y$ in $R^{\prime}$. Assume that
$y$ is an $m$-tuple generating a maximal ideal $\left\langle y\right\rangle $
in $R^{\prime}$. But $R^{\prime}=k\left[  y^{\prime}\right]  $ for a certain
$s$-tuple $y^{\prime}\subseteq R^{\prime}$, and the maximal ideals of
$R^{\prime}$ respond to points from $\mathbb{A}^{s}$ by Zariski
Nullstellensatz \cite[15.4]{AK}. In particular, $\left\langle y\right\rangle
=\left\langle y^{\prime}-a^{\prime}\right\rangle $ for some $a^{\prime}%
\in\mathbb{A}^{s}$. Thus we can assume that $R^{\prime}=k\left[  y^{\prime
}\right]  $ and $\left\langle y\right\rangle =\left\langle y^{\prime
}\right\rangle $, which means that $y$ is a tuple related to $y^{\prime}$. By
Lemma \ref{lemB4}, the $R^{\prime}$-module structure of $H_{p}\left(
y,M\right)  $ is diagonalizable, and it is just the $k$-vector space
structure. In particular, every ascending chain of vector subspaces in
$H_{p}\left(  y,M\right)  $ turns out to be a chain of $R^{\prime}$-submodules
which has to stabilize being a Noetherian $R^{\prime}$-module. Hence
$H_{p}\left(  y,M\right)  $ is a finite dimensional $k$-vector space, and
$i\left(  y\right)  <\infty$.

But $H_{p}\left(  y,M\right)  $ has also Noetherian $R$-module structure on.
In particular, $H_{p}\left(  y,M\right)  $ has the finite length as an
$R$-module, which means that there is a Jordan-H\"{o}lder series with the
quotients $R/\mathfrak{m}_{j}$ for some maximal ideals $\mathfrak{m}%
_{j}\subseteq R$. Thus $H_{p}\left(  y,M\right)  $ is a semilocal $R$-module,
that is, $\operatorname{Supp}_{R}\left(  H_{p}\left(  y,M\right)  \right)  $
is finite. As above $R=k\left[  x\right]  $ for some $n$-tuple $x$, and
maximal ideals respond to points from $\mathbb{A}^{n}$. Hence there is a chain
$0=V_{s}\varsubsetneq V_{s-1}\varsubsetneq\cdots\varsubsetneq V_{0}%
=H_{p}\left(  y,M\right)  $ of $R$-submodules such that $V_{j}/V_{j+1}%
=R/\mathfrak{m}_{j}=k$ with $\mathfrak{m}_{j}=\left\langle x-b^{\left(
j\right)  }\right\rangle $, $b^{\left(  j\right)  }\in\mathbb{A}^{n}$, $1\leq
j\leq s$. Thus $s=\dim_{k}\left(  H_{p}\left(  x,M\right)  \right)  $ and
\[
\operatorname{Supp}_{R}\left(  H_{p}\left(  y,M\right)  \right)
=\operatorname{Ass}_{R}\left(  H_{p}\left(  y,M\right)  \right)
=\operatorname{Max}_{R}\left(  H_{p}\left(  y,M\right)  \right)  =\left\{
\mathfrak{m}_{1},\ldots,\mathfrak{m}_{s}\right\}
\]
with their multiplicities (see \cite[19.4]{AK}). Based on the chain we can
easily construct a basis $\omega=\left(  \omega_{1},\ldots,\omega_{s}\right)
$ for $H_{p}\left(  y,M\right)  $ such that
\[
x_{i}|H_{p}\left(  y,M\right)  =\left[
\begin{array}
[c]{ccc}%
b_{i}^{\left(  1\right)  } &  & 0\\
& \ddots & \\
\ast &  & b_{i}^{\left(  s\right)  }%
\end{array}
\right]
\]
for all $i$, $1\leq i\leq n$ with respect to $\omega$. In particular,%
\[
p\left(  x\right)  |H_{p}\left(  y,M\right)  =\left[
\begin{array}
[c]{ccc}%
p\left(  b^{\left(  1\right)  }\right)  &  & 0\\
& \ddots & \\
\ast &  & p\left(  b^{\left(  s\right)  }\right)
\end{array}
\right]
\]
for every $p\left(  x\right)  \in R$, which means that $R$-action on
$H_{p}\left(  y,M\right)  $ is triangularizable.
\end{proof}

\begin{corollary}
\label{corkey1}Let $R/k$ be an algebra finite extension, and let
$\mathfrak{a\varsubsetneq}R$ be an ideal. The following assertions are
equivalent: $\left(  i\right)  $ there is a tuple $y$ generating
$\mathfrak{a}$ such that $k\left[  y\right]  \subseteq R$ is integral;
$\left(  ii\right)  $ there is a tuple $y$ generating $\mathfrak{a}$ such that
$i\left(  y\right)  <\infty$ for every Noetherian $R$-module $M$ and the
$R$-module structure on every $H_{p}\left(  y,M\right)  $ is triangularizable;
$\left(  iii\right)  $ the radical $\sqrt{\mathfrak{a}}$ is a finite
intersection of maximal ideals of $R$.
\end{corollary}

\begin{proof}
The implication $\left(  i\right)  \Rightarrow\left(  ii\right)  $ is due to
Theorem \ref{thNP11}. Let prove the implication $\left(  ii\right)
\Rightarrow\left(  iii\right)  $. Suppose that $\mathfrak{a=}\left\langle
y\right\rangle $ for some tuple $y$ such that $i\left(  y\right)  <\infty$ for
every Noetherian $R$-module $M$, and the $R$-module structure on every
$H_{p}\left(  y,M\right)  $ is triangularizable. By Lemma \ref{lemB4},
$k\left[  y\right]  $-module structure on $H_{p}\left(  y,M\right)  $ is
diagonalizable. In particular, $\mathfrak{a=}\left\langle y\right\rangle
\subseteq\operatorname{Ann}_{R}\left(  H_{p}\left(  y,M\right)  \right)  $,
where $\operatorname{Ann}_{R}\left(  H_{p}\left(  y,M\right)  \right)  $ is
the annihilator of the $R$-module $H_{p}\left(  y,M\right)  $. Taking into
account that $R$-module $H_{p}\left(  y,M\right)  $ is semilocal (or
triangularizable), we deduce that
\begin{align*}
\sqrt{\mathfrak{a}}  &  \subseteq\sqrt{\operatorname{Ann}_{R}\left(
H_{p}\left(  y,M\right)  \right)  }=\operatorname{nil}_{R}\left(  H_{p}\left(
y,M\right)  \right)  =\cap\operatorname{Supp}_{R}\left(  H_{p}\left(
y,M\right)  \right)  =\cap\operatorname{Ass}_{R}\left(  H_{p}\left(
y,M\right)  \right) \\
&  =\cap_{j=1}^{s}\mathfrak{m}_{j},
\end{align*}
that is, $\sqrt{\mathfrak{a}}\subseteq\cap_{j=1}^{s}\mathfrak{m}_{j}$ (see to
the proof of Theorem \ref{thNP11}, and \cite[13.6]{AK}). If $p=0$ and $M=R$,
then we obtain that $H_{0}\left(  y,R\right)  =R/\left\langle y\right\rangle
=R/\mathfrak{a}$ and $\cap_{j=1}^{s}\mathfrak{m}_{j}=\operatorname{nil}%
_{R}\left(  R/\mathfrak{a}\right)  =\sqrt{\mathfrak{a}}$.

Finally prove that $\left(  iii\right)  \Rightarrow\left(  i\right)  $.
Suppose that $\sqrt{\mathfrak{a}}=\cap_{j=1}^{s}\mathfrak{m}_{j}$ for some
maximal ideals $\left\{  \mathfrak{m}_{j}\right\}  $. It is a primary
decomposition and $\operatorname{Ass}\left(  R/\sqrt{\mathfrak{a}}\right)
\subseteq\left\{  \mathfrak{m}_{j}\right\}  \subseteq\operatorname{Max}\left(
R/\sqrt{\mathfrak{a}}\right)  $ (see \cite[18.17]{AK}). Since $R/\sqrt
{\mathfrak{a}}$ is Noetherian, we deduce that $\operatorname{Spec}\left(
R/\sqrt{\mathfrak{a}}\right)  =\operatorname{Max}\left(  R/\sqrt{\mathfrak{a}%
}\right)  $ (see \cite[17.14]{AK}) or the Krull dimension $\dim\left(
R/\sqrt{\mathfrak{a}}\right)  =0$. By Akizuki-Hopkins Theorem \cite[19.8]{AK},
we conclude that $R/\sqrt{\mathfrak{a}}$ is an Artinian ring. Using again
Zariski Nullstellensatz, we deduce that $\dim_{k}\left(  R/\sqrt{\mathfrak{a}%
}\right)  <\infty$ (the gaps of a Jordan-H\"{o}lder series are $R/\mathfrak{m}%
_{j}=k$). But $R=k\left[  x\right]  $ for some $n$-tuple $x$, and every
$x_{i}$ defines a linear transformation on $R/\sqrt{\mathfrak{a}}$. By
Cayley-Hamilton Theorem, $q_{i}\left(  x_{i}\right)  \in\sqrt{\mathfrak{a}}$
for a monic polynomial $q_{i}\in k\left[  t\right]  $. It follows that
$y_{i}=p_{i}\left(  x_{i}\right)  \in\mathfrak{a}$ for a monic polynomial
$p_{i}\in k\left[  t\right]  $, $1\leq i\leq n$. Put $y=\left(  y_{1}%
,\ldots,y_{n}\right)  $ and consider the subalgebra $R^{\prime}=k\left[
y\right]  \subseteq R$. For every $i$ we have $\left(  p_{i}-y_{i}\right)
\left(  x_{i}\right)  =0$, that is, $x_{i}$ is integral over $R^{\prime}$ and
$R=R^{\prime}\left[  x\right]  $. Hence $R$ is module finite over $R^{\prime}$
or $R^{\prime}\subseteq R$ is an integral extension \cite[10.18]{AK}. But $y$
can easily be extended up to generators of the ideal $\mathfrak{a}$. Just
consider $\mathfrak{b=}\left\langle y\right\rangle \subseteq\mathfrak{a}$ and
pick up generators from $\mathfrak{b/a}$, or just add up another generators of
$\mathfrak{a}$ to $y$ being a Noetherian ideal. Whence $\mathfrak{a=}%
\left\langle y\right\rangle $ and $k\left[  y\right]  \subseteq R$ is an
integral extension.
\end{proof}

Recall that if $M$ is a module over a $k$-algebra $R$ and $t\in R$, which
defines a linear transformation $t|M$ over the $k$-vector space $M$, then its
spectrum $\sigma\left(  t|M\right)  $ is defined to be a subset of those
$\lambda\in k$ such that $t-\lambda$ is not an invertible linear
transformation over $M$. Actually, it is just the Taylor spectrum
$\sigma\left(  t,M\right)  $ of the single tuple $t$ from the algebra $R$.
Note also that it is not expected that $t^{-1}\in R$ in the case of an
invertible linear transformation $t$ over $M$, that is, $0\notin\sigma\left(
t\right)  $ (see below Example \ref{ex2}).

\begin{proposition}
\label{propNP1}(Non voidness) Let $R/k$ be an algebra finite extension and $M$
a Noetherian $R$-module. Then $\sigma\left(  t|M\right)  \neq\varnothing$ for
every $t\in R$. For every chain $0=M_{0}\varsubsetneq M_{1}\varsubsetneq
\cdots\varsubsetneq M_{n-1}\varsubsetneq M_{n}=M$ of submodules with
$M_{i}/M_{i-1}=R/\mathfrak{p}_{i}$ and $\operatorname{Ass}\left(  M\right)
\subseteq\left\{  \mathfrak{p}_{1},\ldots,\mathfrak{p}_{n}\right\}
\subseteq\operatorname{Supp}\left(  M\right)  $, we have $\sigma\left(
t|M\right)  =\cup_{i}\sigma\left(  t|R/\mathfrak{p}_{i}\right)  $. In
particular, $\sigma\left(  t|M\right)  $ is either finite set or it is a dense
subset of $\mathbb{A}^{1}$.
\end{proposition}

\begin{proof}
We can assume that $M\neq0$ and $R=k\left[  x\right]  $ for an $n$-tuple $x$.
Then $\operatorname{Supp}\left(  M\right)  \neq\varnothing$. Actually,
$\operatorname{Supp}\left(  M\right)  =V\left(  \operatorname{Ann}\left(
M\right)  \right)  $ (see \cite[13.4]{AK}) to be the set of all primes
containing the annihilator $\operatorname{Ann}\left(  M\right)  $ of $M$. In
particular, there is a maximal ideal $\mathfrak{m\subseteq}R$ from
$\operatorname{Supp}\left(  M\right)  $, that is, $M_{\mathfrak{m}}\neq0$. By
Zariski Nullstellensatz, $\mathfrak{m=}\left\langle x-a\right\rangle $ for
some $a\in\mathbb{A}^{n}$. Take $t=t\left(  x\right)  \in R$. Then $t-t\left(
a\right)  \in\mathfrak{m}$ and $\left\langle \left(  t/1\right)  -t\left(
a\right)  \right\rangle \subseteq\mathfrak{m}R_{\mathfrak{m}}%
=\operatorname{rad}\left(  R_{\mathfrak{m}}\right)  $. By Nakayama lemma,
$\left(  M/\left\langle t-t\left(  a\right)  \right\rangle M\right)
_{\mathfrak{m}}=M_{\mathfrak{m}}/\left\langle \left(  t/1\right)  -t\left(
a\right)  \right\rangle M_{\mathfrak{m}}\neq0$. It follows that
$M/\left\langle t-t\left(  a\right)  \right\rangle M\neq0$ or
$\operatorname{im}\left(  t-t\left(  a\right)  \right)  \neq M$, that is,
$t\left(  a\right)  \in\sigma\left(  t|M\right)  $.

Further, consider a chain $0=M_{0}\varsubsetneq M_{1}\varsubsetneq
\cdots\varsubsetneq M_{n-1}\varsubsetneq M_{n}=M$ of submodules with
$M_{i}/M_{i-1}=R/\mathfrak{p}_{i}$ and $\operatorname{Ass}\left(  M\right)
\subseteq\left\{  \mathfrak{p}_{1},\ldots,\mathfrak{p}_{n}\right\}
\subseteq\operatorname{Supp}\left(  M\right)  $ (see \cite[17.16]{AK}). If
$t-\lambda$ is invertible on $M$ then $\left(  t-\lambda\right)  ^{-1}$ is an
$R$-linear map too. Hence $t-\lambda$ is invertible iff so are all $\left(
t-\lambda\right)  |M_{i}/M_{i-1}$, that is, $\sigma\left(  t|M\right)
=\cup_{i}\sigma\left(  t|R/\mathfrak{p}_{i}\right)  $.

Finally, consider the case of $M=R/\mathfrak{p}$ for some prime $\mathfrak{p}%
$, which is the coordinate ring of the variety $Y=V\left(  \mathfrak{p}%
\right)  \cap\mathbb{A}^{n}\subseteq\mathbb{A}^{n}$. Then $\operatorname{zdiv}%
\left(  M\right)  =\left\{  0\right\}  $, and $\lambda\in\sigma\left(
t|M\right)  $ iff $t-\lambda\in\mathfrak{m}$ for some maximal ideal
$\mathfrak{m}$ containing $\mathfrak{p}$. By Zariski Nullstellensatz,
$\mathfrak{m=}\left\langle x-a\right\rangle $ for some $a\in Y$, that is,
$\lambda=t\left(  a\right)  $. Hence $\sigma\left(  t|M\right)  =t\left(
Y\right)  $ is an irreducible subset of $\mathbb{A}^{1}$ (see \cite[2.4.4]%
{BurCA}). But $\sigma\left(  t|M\right)  $ is irreducible iff so is its
closure $\sigma\left(  t|M\right)  ^{-}$ in $\mathbb{A}^{1}$ (see
\cite[2.4.2]{BurCA}). So $\sigma\left(  t|M\right)  ^{-}$ is either singleton
or $\sigma\left(  t|M\right)  ^{-}=\mathbb{A}^{1}$. But $\sigma\left(
t|M\right)  \neq\varnothing$, therefore $\sigma\left(  t|M\right)  $ is either
singleton or it is a dense irreducible subset of $\mathbb{A}^{1}$.
\end{proof}

\begin{example}
\label{ex1}As an example of the case $\sigma\left(  t|M\right)  =\mathbb{A}%
^{1}$ consider the action of $X$ on the PID $k\left[  X\right]  $. If
$\ell_{R}\left(  M\right)  <\infty$ then $\operatorname{Ass}\left(  M\right)
=\operatorname{Max}\left(  M\right)  $ and $M$ is a finite dimensional
$k$-vector space (see \cite[19.4]{AK}). In this case, $t$ is a
triangularizable whose diagonal entries consist of its finite spectrum
$\sigma\left(  t|M\right)  $.
\end{example}

\begin{example}
\label{ex2}The spectrum can be empty set unless $M$ is Noetherian. Put
$R=k\left[  X\right]  $ and $M=k\left(  X\right)  =\operatorname{Frac}\left(
R\right)  $. Then $M$ is an $R$-module and $X$ is acting as a multiplication
operator over $M$. For every $\lambda\in k$ we have $X-\lambda$ is invertible
and $\left(  X-\lambda\right)  ^{-1}$ is the multiplication operator on $M$ by
$1/\left(  X-\lambda\right)  $. Note also that $\left(  X-\lambda\right)
^{-1}\notin R$.
\end{example}

\begin{example}
\label{ex3}The spectrum $\sigma\left(  t|M\right)  $ can be an open subset of
$\mathbb{A}^{1}$. Let $M$ be the coordinate ring of the hyperbola $Y=\left\{
xy=1\right\}  $ in $\mathbb{A}^{2}$, and $t=x$. Then $\sigma\left(
t|M\right)  =t\left(  Y\right)  =\mathbb{A}^{1}-\left\{  0\right\}  $ is an
open dense subset of $\mathbb{A}^{1}$.
\end{example}

Now we can prove a key result of the present section.

\begin{theorem}
\label{thNP1}(The projection property) Let $R/k$ be an algebra finite
extension of the field $k$, $M$ a Noetherian $R$-module, $y=\left(
y_{1},\ldots,y_{m}\right)  $ an $m$-tuple in $R$, and let $y^{\prime}=\left(
y_{1},\ldots,y_{m-1}\right)  $. Then $\sigma\left(  y^{\prime},M\right)
=\pi\left(  \sigma\left(  y,M\right)  \right)  $, where $\pi:\mathbb{A}%
^{m}\rightarrow\mathbb{A}^{m-1}$ is the canonical projection onto first $m-1$ coordinates.
\end{theorem}

\begin{proof}
The inclusion $\pi\left(  \sigma\left(  y,M\right)  \right)  \subseteq
\sigma\left(  y^{\prime},M\right)  $ holds in the general case thanks to
Corollary \ref{corProj1}. Conversely, take $a^{\prime}\in\sigma\left(
y^{\prime},M\right)  $. Then $H_{p}\left(  y^{\prime}-a^{\prime},M\right)
\neq0$ for some $p$. But $H_{p}\left(  y^{\prime}-a^{\prime},M\right)  $ is a
Noetherian $R$-module, and $y_{m}$ defines a linear transformation on it. By
Proposition \ref{propNP1}, $\sigma\left(  y_{m}|H_{p}\left(  y^{\prime
}-a^{\prime},M\right)  \right)  \neq\varnothing$, that is,
\[
y_{m}-a_{m}:H_{p}\left(  y^{\prime}-a^{\prime},M\right)  \rightarrow
H_{p}\left(  y^{\prime}-a^{\prime},M\right)
\]
is not invertible for some $a_{m}\in k$. If $\operatorname{coker}\left(
y_{m}-a_{m}|H_{p}\left(  y^{\prime}-a^{\prime},M\right)  \right)  \neq0$ then
using the following exact (Lemma \ref{lemB3}) sequence
\[
0\leftarrow\ker\left(  y_{m}-a_{m}|H_{p-1}\left(  y^{\prime}-a^{\prime
},M\right)  \right)  \longleftarrow H_{p}\left(  y-a,M\right)  \longleftarrow
\operatorname{coker}\left(  y_{m}-a_{m}|H_{p}\left(  y^{\prime}-a^{\prime
},M\right)  \right)  \leftarrow0,
\]
we conclude that $H_{p}\left(  y-a,M\right)  \neq0$. If $\ker\left(
y_{m}-a_{m}|H_{p}\left(  y^{\prime}-a^{\prime},M\right)  \right)  \neq0$ then
using the exact sequence
\[
0\leftarrow\ker\left(  y_{m}-a_{m}|H_{p}\left(  y^{\prime}-a^{\prime
},M\right)  \right)  \longleftarrow H_{p+1}\left(  y-a,M\right)
\longleftarrow\operatorname{coker}\left(  y_{m}-a_{m}|H_{p+1}\left(
y^{\prime}-a^{\prime},M\right)  \right)  \leftarrow0,
\]
we conclude that $H_{p+1}\left(  y-a,M\right)  \neq0$. Hence either
$H_{p}\left(  y-a,M\right)  \neq0$ or $H_{p+1}\left(  y-a,M\right)  \neq0$.
Anyway $a\in\sigma\left(  y,M\right)  $ and $\pi\left(  a\right)  =a^{\prime}$.
\end{proof}

\begin{corollary}
\label{corNE1}Let $R/k$ be an algebra finite extension of the field $k$ and
let $M$ be a Noetherian $R$-module. Then $\sigma\left(  y,M\right)
\neq\varnothing$ for every tuple $y$ in $R$.
\end{corollary}

\begin{proof}
Indeed, using Theorem \ref{thNP1}, we conclude that $\pi_{1}\left(
\sigma\left(  y,M\right)  \right)  =\sigma\left(  y_{1},M\right)
=\sigma\left(  y_{1}|M\right)  $, where $\pi_{1}:\mathbb{A}^{m}\rightarrow
\mathbb{A}^{1}$ is the canonical projection onto the first coordinate. By
Proposition \ref{propNP1}, $\sigma\left(  y_{1}|M\right)  \neq\varnothing$.
Therefore $\sigma\left(  y,M\right)  \neq\varnothing$.
\end{proof}

Nonetheless the point spectrum $\sigma_{\operatorname{p}}\left(  y,M\right)  $
could be an empty set, which detects only zero-divisors in the module $M$.

\subsection{The spectral mapping property}

Finally, let us prove the main result for the present section. As above we fix
an algebra finite extension $R/k$, an $n$-tuple with $x=\left(  x_{1}%
,\ldots,x_{n}\right)  $ from $R$, and an $m$-tuple $p\left(  x\right)
=\left(  p_{1}\left(  x\right)  ,\ldots,p_{m}\left(  x\right)  \right)  $ from
the subalgebra $k\left[  x\right]  $. If $M$ is an $R$-module then $k\left[
p\left(  x\right)  \right]  \subseteq R$ is a subalgebra and $M$ has the
natural $k\left[  p\left(  x\right)  \right]  $-module structure. Moreover,
the tuple $p\left(  x\right)  $ defines a morphism mapping $p:\mathbb{A}%
^{n}\rightarrow\mathbb{A}^{m}$, $a\mapsto p\left(  a\right)  =\left(
p_{1}\left(  a\right)  ,\ldots,p_{m}\left(  a\right)  \right)  $.

\begin{theorem}
\label{thSMP}If $R/k$ is an algebra finite extension of the field $k$, $M$ a
Noetherian $R$-module, then $\sigma\left(  p\left(  x\right)  ,M\right)
=p\left(  \sigma\left(  x,M\right)  \right)  $ for all tuples $x$ and
$p\left(  x\right)  $ from $R$.
\end{theorem}

\begin{proof}
Consider the $n+m$-tuple $y=\left(  x,p\left(  x\right)  \right)  $ in $R$. By
Theorem \ref{thNP1}, $\sigma\left(  x,M\right)  =\pi_{n}\left(  \sigma\left(
y,M\right)  \right)  $, where $\pi_{n}:\mathbb{A}^{n+m}\rightarrow
\mathbb{A}^{n}$ is the canonical projection onto the first $n$ coordinates.
Take $\left(  a,b\right)  \in\sigma\left(  y,M\right)  $. By Corollary
\ref{corBB1}, $b_{j}=p_{j}\left(  a\right)  $ for all $j$, $1\leq j\leq m$,
that is, $b=p\left(  a\right)  $. Thus $\sigma\left(  y,M\right)  =\left\{
\left(  a,p\left(  a\right)  \right)  :a\in\sigma\left(  x,M\right)  \right\}
$. Using again Theorem \ref{thNP1}, we conclude that $\sigma\left(  p\left(
x\right)  ,M\right)  =\pi_{m}\left(  \sigma\left(  y,M\right)  \right)
=p\left(  \sigma\left(  x,M\right)  \right)  $, where $\pi_{m}:\mathbb{A}%
^{n+m}\rightarrow\mathbb{A}^{m}$ is the canonical projection onto the last $m$ coordinates.
\end{proof}

\begin{corollary}
If $R/k$ is an algebra finite extension of the field $k$, $M$ a Noetherian
$R$-module, then $\sigma\left(  p\left(  x\right)  |M\right)  =p\left(
\sigma\left(  x,M\right)  \right)  $ for every polynomial $p\left(  x\right)
\in R$.
\end{corollary}

\begin{corollary}
\label{corSMP1}Let $R/k$ be an algebra, $R^{\prime}=k\left[  x\right]
\subseteq R$ a $k$-subalgebra with an $n$-tuple $x$ such that $R/R^{\prime}$
is integral, and let $M$ be a Noetherian $R$-module. If $y$ is an $m$-tuple in
$R^{\prime}$ related to $x$ with $m\leq n$, and $0\in\sigma\left(  x,M\right)
$, then $d\left(  x\right)  =f_{n-m-1}\ldots f_{m}\left(  d\left(  y\right)
\right)  $ and $i\left(  x\right)  =\delta_{nm}i\left(  y\right)  $.
\end{corollary}

\begin{proof}
By Theorem \ref{thSMP}, $\left(  0,\lambda\right)  \in\sigma\left(  \left(
x,y\right)  ,M\right)  $ for a certain $\lambda\in\mathbb{A}^{m}$. Actually,
$\lambda=y\left(  0\right)  =0$, and $0\in\sigma\left(  \left(  x,y\right)
,M\right)  $. Since $R^{\prime}=k\left[  x\right]  $ and $y$ generates the
maximal ideal $\left\langle x\right\rangle $ in $R^{\prime}$, we conclude that
$i\left(  x\right)  <\infty$ and $i\left(  y\right)  <\infty$ thanks to
Theorem \ref{thNP11}. It remains to use Corollary \ref{corind1}.
\end{proof}

\section{Spectra of a module over a scheme and extension property\label{Sec2}}

In the present section we review some key facts on spectrum of a module over a
scheme from \cite{DMMJ}, and use them regarding our framework of spectra. We
also introduce the point spectrum in the general setting and investigate its
invariance under integral extensions. Thereafter we switch to the case of the
affine scheme $\mathbb{A}_{k}^{n}$ over a field $k$.

\subsection{Spectra of a module over scheme\label{subsecSpMod}}

Let $\left(  \mathfrak{X},\mathcal{O}_{\mathfrak{X}}\right)  $ be a scheme
with the ring $R=\Gamma\left(  \mathfrak{X},\mathcal{O}_{\mathfrak{X}}\right)
$ of global sections, and let $M\in R$-$\operatorname{mod}$. We say that a
point $x\in\mathfrak{X}$ belongs to \textit{a resolvent set}
$\operatorname{res}\left(  \mathfrak{X},M\right)  $ if there is an affine
neighborhood $U$ of $x$ such that $\mathcal{O}_{\mathfrak{X}}\left(  U\right)
\perp_{R}M$, that is, $\operatorname{Tor}_{i}^{R}\left(  \mathcal{O}%
_{\mathfrak{X}}\left(  U\right)  ,M\right)  =0$ for all $i\geq0$. Since $U$ is
an affine neighborhood for all points close to $x$, it follows that
$\operatorname{res}\left(  \mathfrak{X},M\right)  $ is an open set, whose
complement set $\sigma\left(  \mathfrak{X},M\right)  $ is called \textit{the
spectrum of }$M$\textit{ over the scheme.} The property $\mathcal{O}%
_{\mathfrak{X}}\left(  U\right)  \perp_{R}M$ implies that $\mathcal{O}%
_{\mathfrak{X}}\left(  V\right)  \perp_{R}M$ for every open affine $V\subseteq
U$. Indeed, suppose $U=\operatorname{Spec}\left(  B\right)  $ and take a free
resolution
\[
\mathcal{P=}R\otimes G_{\circ},\quad0\leftarrow R\otimes G_{0}%
\overset{\partial_{0}}{\longleftarrow}R\otimes G_{1}\overset{\partial
_{1}}{\longleftarrow}\cdots
\]
of the module $M$, where $G_{j}$ are free abelian groups. Since $\mathcal{O}%
_{\mathfrak{X}}$ is a quasi-coherent sheaf, it follows that $\mathcal{O}%
_{\mathfrak{X}}|_{U}=\widetilde{B}$ with $\mathcal{O}_{\mathfrak{X}}\left(
U\right)  =B$, and $B\otimes_{R}\mathcal{P}=B\otimes G_{\circ}$, which is
exact iff so are $B_{x}\otimes G_{\circ}$, $x\in\operatorname{Spec}\left(
B\right)  $ (see \cite[Lemma 2.1]{DMMJ}). But $B_{x}=\mathcal{O}%
_{\mathfrak{X},x}$ for all $x\in U$. Thus $\mathcal{O}_{\mathfrak{X}}\left(
U\right)  \perp_{R}M$ is equivalent to $\mathcal{O}_{\mathfrak{X},x}\perp
_{R}M$ for all $x\in U$. In particular, $\mathcal{O}_{\mathfrak{X},x}\perp
_{R}M$ for all $x\in V$ imply that $\mathcal{O}_{\mathfrak{X}}\left(
V\right)  \perp_{R}M$.

If $\mathfrak{X}=\operatorname{Spec}\left(  R\right)  $ is an affine scheme,
then $\mathcal{O}_{\mathfrak{X},x}\perp_{R}M$ means that $R_{x}\otimes
_{R}\mathcal{P}$ is exact. But $R_{x}$ is a flat $R$-module, therefore the
sequence $0\leftarrow R_{x}\otimes_{R}M\leftarrow R_{x}\otimes_{R}\mathcal{P}$
remains exact. Hence $\mathcal{O}_{\mathfrak{X},x}\perp_{R}M$ iff $M_{x}=0$.
Thus $x\in\operatorname{res}\left(  \mathfrak{X},M\right)  $ iff $M_{y}=0$ for
all $y$ from a neighborhood of $x$. Since $\sigma\left(  \mathfrak{X}%
,M\right)  =\mathfrak{X}-\operatorname{res}\left(  \mathfrak{X},M\right)  $,
we derive that $x\in\sigma\left(  \mathfrak{X},M\right)  $ such that property
$M_{y}\neq0$ holds for some $y$ in every neighborhood of $x$, that is,
$\sigma\left(  \mathfrak{X},M\right)  $ is the closure of $\operatorname{Supp}%
\left(  M\right)  $ in $\mathfrak{X}$. Thus
\begin{equation}
\sigma\left(  \mathfrak{X},M\right)  =\operatorname{Supp}\left(  M\right)
^{-}\subseteq V\left(  \operatorname{Ann}\left(  M\right)  \right)  .
\label{cl}%
\end{equation}
If $M$ is a finitely generated $R$-module then $\operatorname{Supp}\left(
M\right)  $ is closed, $\operatorname{res}\left(  \mathfrak{X},M\right)
=\left\{  x\in\mathfrak{X}:R_{x}\perp_{R}M\right\}  =\left\{  x\in
\mathfrak{X}:M_{x}=0\right\}  $, that is,
\begin{equation}
\sigma\left(  \mathfrak{X},M\right)  =\operatorname{Supp}\left(  M\right)
=V\left(  \operatorname{Ann}\left(  M\right)  \right)  . \label{SV}%
\end{equation}
In the case of a general scheme the spectrum can be localized using an affine
covering. In particular, $\sigma\left(  \mathbb{P}_{k}^{r},M\right)
=\mathbb{P}_{k}^{r}$ whenever $M$ is a nonzero vector space over a field $k$
(see \cite{DMMJ}). The related spectral mapping formula was mentioned in
(\ref{2}).

\subsection{The point spectrum}

Let $\mathfrak{X}=\operatorname{Spec}\left(  R\right)  $ be an affine scheme
and let $M$ be an $R$-module. The set $\operatorname{Ass}_{R}\left(  M\right)
$ of all primes associated to $M$ is called \textit{the point spectrum}
$\sigma_{\operatorname{p}}\left(  \mathfrak{X},M\right)  $ \textit{of }$M$. If
$R$ is Noetherian then $\operatorname{Supp}\left(  M\right)  =\cup\left\{
V\left(  \mathfrak{p}\right)  :\mathfrak{p}\in\sigma_{\operatorname{p}}\left(
\mathfrak{X},M\right)  \right\}  $ and the minimal primes $\operatorname{Min}%
\left(  M\right)  $ from $\operatorname{Supp}\left(  M\right)  $ belong to
$\sigma_{\operatorname{p}}\left(  \mathfrak{X},M\right)  $ \cite[17.14]{AK}.
In particular, $\sigma_{\operatorname{p}}\left(  \mathfrak{X},M\right)  $ is
dense in $\operatorname{Supp}\left(  M\right)  $, which in turn implies that
$\sigma_{\operatorname{p}}\left(  \mathfrak{X},M\right)  ^{-}=\sigma\left(
\mathfrak{X},M\right)  $ in $\mathfrak{X}$ (see (\ref{cl})).

Now let $R^{\prime}\subseteq R$ be a ring extension, $\iota:R^{\prime
}\rightarrow R$ is the realted inclusion map, and $\mathfrak{X}^{\prime
}=\operatorname{Spec}\left(  R^{\prime}\right)  $ with the canonical mapping
$\iota^{\ast}:\mathfrak{X}\rightarrow\mathfrak{X}^{\prime}$, $\iota^{\ast
}\left(  \mathfrak{p}\right)  =\iota^{-1}\left(  \mathfrak{p}\right)
=\mathfrak{p\cap}R$. If $\mathfrak{p}^{\prime}\mathfrak{=p}\cap R=\iota^{\ast
}\left(  \mathfrak{p}\right)  $ we say that $\mathfrak{p}$ lies over the prime
$\mathfrak{p}^{\prime}$. Note also that $M$ has the canonical $R^{\prime}%
$-module structure along the embedding $\iota$. If $\mathfrak{a}%
=\operatorname{Ann}\left(  m\right)  $ is the annihilator of $m\in M$ in $R$
then $\mathfrak{a}^{\prime}\mathfrak{=a}\cap R^{\prime}=\operatorname{Ann}%
^{\prime}\left(  m\right)  $ is the annihilator of $m$ in $R^{\prime}$. It
follows that $\iota^{\ast}\left(  \sigma_{\operatorname{p}}\left(
\mathfrak{X},M\right)  \right)  \subseteq\sigma_{\operatorname{p}}\left(
\mathfrak{X}^{\prime},M\right)  $.

Notice also that $\iota^{\ast}\left(  \operatorname{Supp}\left(  M\right)
\right)  \subseteq\operatorname{Supp}^{\prime}\left(  M\right)  $, where
$\operatorname{Supp}^{\prime}\left(  M\right)  =\operatorname{Supp}%
_{R^{\prime}}\left(  M\right)  $. Indeed, if $M_{\mathfrak{p}}\neq\left\{
0\right\}  $ for some $\mathfrak{p}\in\mathfrak{X}$, then $sm\neq0$ for some
$m\in M$ and all $s\in R-\mathfrak{p}$. But $\mathfrak{p}^{\prime
}\mathfrak{=p\cap}R^{\prime}$, and for $s^{\prime}\in R^{\prime}$ we have
$s^{\prime}\in R^{\prime}-\mathfrak{p}^{\prime}$ iff $s^{\prime}\in
R-\mathfrak{p}$. Hence $s^{\prime}m\neq0$ for all $s^{\prime}\in R^{\prime
}-\mathfrak{p}^{\prime}$, which means that $m/1\neq0$ in $M_{\mathfrak{p}%
^{\prime}}$ or $\mathfrak{p}^{\prime}\in\operatorname{Supp}^{\prime}\left(
M\right)  $.

Using (\ref{cl}) and continuity of the map $\iota^{\ast}$, we deduce that
$\iota^{\ast}\left(  \sigma\left(  \mathfrak{X},M\right)  \right)
\subseteq\sigma\left(  \mathfrak{X}^{\prime},M\right)  $. Actually
$\iota^{\ast}\left(  \sigma\left(  \mathfrak{X},M\right)  \right)  $ is dense
in $\sigma\left(  \mathfrak{X}^{\prime},M\right)  $ due to the spectral
mapping property (\ref{2}).

\begin{theorem}
\label{propEx1}If $R^{\prime}\subseteq R$ is a ring extension and $M\in
R$-$\operatorname{mod}$ then $\sigma_{\operatorname{p}}\left(  \mathfrak{X}%
^{\prime},M\right)  =\iota^{\ast}\left(  \sigma_{\operatorname{p}}\left(
\mathfrak{X},M\right)  \right)  $ whenever $R$ is Noetherian.
\end{theorem}

\begin{proof}
First note that $M\neq\left\{  0\right\}  $ iff $\sigma_{\operatorname{p}%
}\left(  \mathfrak{X},M\right)  \neq\varnothing$ (see \cite[17.10]{AK}), and
the result follows in the case of $\sigma_{\operatorname{p}}\left(
\mathfrak{X},M\right)  =\varnothing$. Thus we can assume that $M\neq\left\{
0\right\}  $, and take $\mathfrak{q}\in\sigma_{\operatorname{p}}\left(
\mathfrak{X}^{\prime},M\right)  $. Then $\iota_{\mathfrak{p}}:R_{\mathfrak{q}%
}^{\prime}\rightarrow R_{\mathfrak{q}}$ is a ring extension and
$M_{\mathfrak{q}}=R_{\mathfrak{q}}^{\prime}\otimes_{R^{\prime}}%
M=R_{\mathfrak{q}}^{\prime}\otimes_{R^{\prime}}R\otimes_{R}M=R_{\mathfrak{q}%
}\otimes_{R}M$ is an $R_{\mathfrak{q}}$-module either, where $R_{\mathfrak{q}%
}=\iota\left(  R^{\prime}-\mathfrak{q}\right)  ^{-1}R$. Moreover,
$\mathfrak{q}R_{\mathfrak{q}}^{\prime}\in\operatorname{Ass}_{R_{\mathfrak{q}%
}^{\prime}}\left(  M_{\mathfrak{q}}\right)  $ \cite[17.8]{AK}, say
$\mathfrak{q}R_{\mathfrak{q}}^{\prime}=\operatorname{Ann}^{\prime}\left(
m/1\right)  $ for some $m/1\in M_{\mathfrak{q}}\backslash\left\{  0\right\}
$. If $\operatorname{Ann}\left(  m/1\right)  $ is the annihilator of $m/1$ in
$R_{\mathfrak{q}}$ then $\operatorname{Ann}\left(  m/1\right)  \cap
R_{\mathfrak{q}}^{\prime}=\operatorname{Ann}^{\prime}\left(  m/1\right)
=\mathfrak{q}R_{\mathfrak{q}}^{\prime}$. Pick%
\[
\mathfrak{S=}\left\{  \operatorname{Ann}\left(  n\right)  :n\in
M_{\mathfrak{q}},\operatorname{Ann}\left(  n\right)  \cap R_{\mathfrak{q}%
}^{\prime}=\mathfrak{q}R_{\mathfrak{q}}^{\prime}\right\}  ,
\]
which is a nonempty set of ideals of the ring $R_{\mathfrak{q}}$. Since $R$ is
Noetherian, so is $R_{\mathfrak{q}}$ and $\mathfrak{S}$ has a maximal element
$\mathfrak{l}=\operatorname{Ann}\left(  n\right)  $ for some $n\in
M_{\mathfrak{q}}$. If $x\in R_{\mathfrak{q}}-\mathfrak{l}$ then $xn\neq0$ (or
$1\notin\operatorname{Ann}\left(  xn\right)  $), $\mathfrak{l}\subseteq
\operatorname{Ann}\left(  xn\right)  \neq R_{\mathfrak{q}}$, which in turn
implies that
\[
\mathfrak{q}R_{\mathfrak{q}}^{\prime}=\operatorname{Ann}\left(  n\right)  \cap
R_{\mathfrak{q}}^{\prime}\subseteq\operatorname{Ann}\left(  xn\right)  \cap
R_{\mathfrak{q}}^{\prime}\subsetneqq R_{\mathfrak{q}}^{\prime}.
\]
But $R_{\mathfrak{q}}^{\prime}$ is local with its unique maximal ideal
$\mathfrak{q}R_{\mathfrak{q}}^{\prime}$, therefore $\mathfrak{q}%
R_{\mathfrak{q}}^{\prime}=\operatorname{Ann}\left(  xn\right)  \cap
R_{\mathfrak{q}}^{\prime}$ or $\operatorname{Ann}\left(  xn\right)
\in\mathfrak{S}$. It follows that $\mathfrak{l=}\operatorname{Ann}\left(
xn\right)  $ whenever $x\in R_{\mathfrak{q}}-\mathfrak{l}$. In particular, if
$xy\in\mathfrak{l}$ with $x\in R_{\mathfrak{q}}-\mathfrak{l}$ then $yxn=0$ or
$y\in\operatorname{Ann}\left(  xn\right)  =\mathfrak{l}$, which means that
$\mathfrak{l}$ is a prime or $\mathfrak{l}\in\operatorname{Ass}%
_{R_{\mathfrak{q}}}\left(  M_{\mathfrak{q}}\right)  $, and it is lying over
$\mathfrak{q}R_{\mathfrak{q}}^{\prime}$. But $\mathfrak{l}=\mathfrak{p}%
R_{\mathfrak{q}}$ for some $\mathfrak{p}\in\operatorname{Ass}_{R}\left(
M\right)  $ \cite[17.8]{AK}, \cite[4.1.5]{BurCA} with $\mathfrak{p}\cap\left(
R^{\prime}-\mathfrak{q}\right)  =\varnothing$ (or $\mathfrak{p}\cap R^{\prime
}\subseteq\mathfrak{q}$). Prove that $\mathfrak{p}\cap R^{\prime}%
=\mathfrak{q}$. Take $x^{\prime}\in\mathfrak{q}$. Then $x^{\prime}%
/1\in\mathfrak{q}R_{\mathfrak{q}}^{\prime}\subseteq\operatorname{Ann}\left(
n\right)  =\mathfrak{p}R_{\mathfrak{q}}$ or $x^{\prime}/1=x/t^{\prime}$ for
some $x\in\mathfrak{p}$ and $t^{\prime}\in R^{\prime}-\mathfrak{q}$. It
follows that $sx^{\prime}\in\mathfrak{p}$ for some $s\in R^{\prime
}-\mathfrak{q}$. Taking into account that $s\in R-\mathfrak{p}$, we conclude
that $x^{\prime}\in\mathfrak{p}$ or $x^{\prime}\in\mathfrak{p}\cap R^{\prime}%
$. Hence $\mathfrak{q=p}\cap R^{\prime}=\iota^{\ast}\left(  \mathfrak{p}%
\right)  $ and $\mathfrak{p}\in\sigma_{\operatorname{p}}\left(  \mathfrak{X}%
,M\right)  $.
\end{proof}

\begin{remark}
\label{rem00}Note that $\operatorname{Supp}^{\prime}\left(  M\right)  $ can be
much larger than $\iota^{\ast}\left(  \operatorname{Supp}\left(  M\right)
\right)  $. For example, put $R^{\prime}=\mathbb{Z\subset Q=}R$ and
$M=\mathbb{Q}$. Then $\sigma\left(  \mathfrak{X},M\right)
=\operatorname{Supp}\left(  M\right)  =\left\{  0\right\}  $ whereas
$\sigma\left(  \mathfrak{X}^{\prime},M\right)  =\operatorname{Supp}^{\prime
}\left(  M\right)  ^{-}=\operatorname{Spec}\left(  \mathbb{Z}\right)
=\mathfrak{X}^{\prime}$. But $\sigma_{\operatorname{p}}\left(  \mathfrak{X}%
^{\prime},M\right)  =\operatorname{Ass}_{\mathbb{Z}}\left(  \mathbb{Q}\right)
=\left\{  0\right\}  =\iota^{\ast}\left(  \left\{  0\right\}  \right)
=\iota^{\ast}\left(  \operatorname{Ass}_{\mathbb{Q}}\left(  \mathbb{Q}\right)
\right)  =\iota^{\ast}\left(  \sigma_{\operatorname{p}}\left(  \mathfrak{X}%
,M\right)  \right)  $, and it is dense in $\mathfrak{X}^{\prime}$.
\end{remark}

\begin{corollary}
\label{corEx1}Let $k$ be a field and let $k\subseteq R^{\prime}\subseteq R$ be
$k$-algebra extensions such that $R/k$ is algebra-finite. Then $\sigma
_{\operatorname{p}}\left(  \mathfrak{X}^{\prime},M\right)  =\iota^{\ast
}\left(  \sigma_{\operatorname{p}}\left(  \mathfrak{X},M\right)  \right)  $
for every $R$-module $M$.
\end{corollary}

\begin{proof}
The ring $R$ is Noetherian by Hilbert Basis Theorem. It remains to use Theorem
\ref{propEx1}.
\end{proof}

\begin{corollary}
Let $R^{\prime}\subseteq R$ be a ring extension with Noetherian $R$, $M\in
R$-$\operatorname{mod}$ and let $Q\subseteq M$ be an $R$-submodule. If $Q$ is
$\mathfrak{p}$-primary then $Q$ is $\mathfrak{p}^{\prime}$-primary $R^{\prime
}$-submodule of $M$, where $\mathfrak{p}\in\mathfrak{X}$ and $\mathfrak{p}%
^{\prime}=\iota^{\ast}\left(  \mathfrak{p}\right)  $. If $M$ is a Noetherian
$R$-module with its submodule $Q$ then $Q$ is $\mathfrak{p}^{\prime}$-primary
$R^{\prime}$-submodule of $M$ iff $Q=Q_{1}\cap\cdots\cap Q_{r}$ admits an
irredundant primary decomposition in $R$-$\operatorname{mod}$ such that every
$Q_{j}$ is $\mathfrak{p}^{\prime}$-primary $R^{\prime}$-module.
\end{corollary}

\begin{proof}
Suppose that $Q$ is a $\mathfrak{p}$-primary submodule, that is,
$\operatorname{Ass}_{R}\left(  M/Q\right)  =\left\{  \mathfrak{p}\right\}  $
or $\sigma_{\operatorname{p}}\left(  \mathfrak{X},M/Q\right)  =\left\{
\mathfrak{p}\right\}  $. Since $\mathfrak{p}^{\prime}=\iota^{\ast}\left(
\mathfrak{p}\right)  \in\iota^{\ast}\left(  \sigma_{\operatorname{p}}\left(
\mathfrak{X},M/Q\right)  \right)  \subseteq\sigma_{\operatorname{p}}\left(
\mathfrak{X}^{\prime},M/Q\right)  $, it follows that $\mathfrak{p}^{\prime}%
\in\operatorname{Ass}_{R^{\prime}}\left(  M/Q\right)  $. Conversely, if
$\mathfrak{q\in}\operatorname{Ass}_{R^{\prime}}\left(  M/Q\right)  $ then
\[
\mathfrak{q\in}\sigma_{\operatorname{p}}\left(  \mathfrak{X}^{\prime
},M/Q\right)  =\iota^{\ast}\left(  \sigma_{\operatorname{p}}\left(
\mathfrak{X},M/Q\right)  \right)  =\iota^{\ast}\left(  \operatorname{Ass}%
_{R}\left(  M/Q\right)  \right)  =\iota^{\ast}\left(  \left\{  \mathfrak{p}%
\right\}  \right)  =\left\{  \mathfrak{p}^{\prime}\right\}
\]
thanks to Theorem \ref{propEx1}, that is, $\mathfrak{q=p}^{\prime}$. Hence
$\operatorname{Ass}_{R^{\prime}}\left(  M/Q\right)  =\left\{  \mathfrak{p}%
^{\prime}\right\}  $, which means that $Q$ is a $\mathfrak{p}^{\prime}%
$-primary $R^{\prime}$-submodule of $M$.

Finally, assume that $M$ is a finitely generated $R$-module. By Lasker-Noether
Theorem \cite[18.19]{AK}, $Q=Q_{1}\cap\cdots\cap Q_{r}$ admits an irredundant
primary decomposition in $R$-$\operatorname{mod}$, where $Q_{j}$ is
$\mathfrak{p}_{j}$-primary. In this case, $\left\{  \mathfrak{p}_{1}%
,\ldots,\mathfrak{p}_{r}\right\}  $ are uniquely defined, in fact they are all
distinct primes of $\operatorname{Ass}_{R}\left(  M/Q\right)  $ (First
Uniqueness \cite[18.18]{AK}). In particular,%
\[
\operatorname{Ass}_{R^{\prime}}\left(  M/Q\right)  =\sigma_{\operatorname{p}%
}\left(  \mathfrak{X}^{\prime},M/Q\right)  =\iota^{\ast}\left(
\operatorname{Ass}_{R}\left(  M/Q\right)  \right)  =\iota^{\ast}\left(
\left\{  \mathfrak{p}_{1},\ldots,\mathfrak{p}_{r}\right\}  \right)
\]
by virtue of Theorem \ref{propEx1}. If $Q$ is $\mathfrak{p}^{\prime}$-primary
$R^{\prime}$-submodule then $\operatorname{Ass}_{R^{\prime}}\left(
M/Q\right)  =\left\{  \mathfrak{p}^{\prime}\right\}  $ and
\[
\operatorname{Ass}_{R^{\prime}}\left(  M/Q_{j}\right)  =\sigma
_{\operatorname{p}}\left(  \mathfrak{X}^{\prime},M/Q_{j}\right)  =\iota^{\ast
}\left(  \operatorname{Ass}_{R}\left(  M/Q_{j}\right)  \right)  =\iota^{\ast
}\left(  \left\{  \mathfrak{p}_{j}\right\}  \right)  =\left\{  \mathfrak{p}%
^{\prime}\right\}
\]
for every $j$. Thus every $Q_{j}$ is $\mathfrak{p}^{\prime}$-primary
$R^{\prime}$-module.

Conversely, suppose that $Q=Q_{1}\cap\cdots\cap Q_{r}$ admits an irredundant
primary decomposition in $R$-$\operatorname{mod}$ such that every $Q_{j}$ is
$\mathfrak{p}^{\prime}$-primary $R^{\prime}$-module. Then $\iota^{\ast}\left(
\left\{  \mathfrak{p}_{j}\right\}  \right)  =\left\{  \mathfrak{p}^{\prime
}\right\}  $ for every $j$, and $\operatorname{Ass}_{R^{\prime}}\left(
M/Q\right)  =\iota^{\ast}\left(  \left\{  \mathfrak{p}_{1},\ldots
,\mathfrak{p}_{r}\right\}  \right)  =\left\{  \mathfrak{p}^{\prime}\right\}
$, which means that $Q$ is a $\mathfrak{p}^{\prime}$-primary $R^{\prime}$-module.
\end{proof}

\subsection{Integral extensions}

Now assume that $R^{\prime}\subseteq R$ is an integral extension of rings and
$M\in R$-$\operatorname{mod}$. In this case, the mapping $\iota^{\ast
}:\mathfrak{X}\rightarrow\mathfrak{X}^{\prime}$ is surjective due to
Krull-Cohen-Seidenberg Theory \cite[Ch. 14]{AK}, \cite[5.2]{BurCA}.

\begin{proposition}
\label{propEx2}If $R^{\prime}\subseteq R$ is integral and $M\in R$%
-$\operatorname{mod}$ then $\operatorname{Supp}^{\prime}\left(  M\right)
=\iota^{\ast}\left(  \operatorname{Supp}\left(  M\right)  \right)  $.
\end{proposition}

\begin{proof}
The inclusion $\iota^{\ast}\left(  \operatorname{Supp}\left(  M\right)
\right)  \subseteq\operatorname{Supp}^{\prime}\left(  M\right)  $ was proved
above. Conversely, take $\mathfrak{p}^{\prime}\in\operatorname{Supp}^{\prime
}\left(  M\right)  $. Then $M_{\mathfrak{p}^{\prime}}\neq\left\{  0\right\}
$, which means that $s^{\prime}m\neq0$ for all $s^{\prime}\in R^{\prime
}-\mathfrak{p}^{\prime}$ and some $m\in M$. Put $\mathfrak{a=}%
\operatorname{Ann}\left(  m\right)  \subseteq R$ and $\mathfrak{a}^{\prime
}=\operatorname{Ann}^{\prime}\left(  m\right)  \subseteq R^{\prime}$ to be
ideals with $\mathfrak{a}^{\prime}=\mathfrak{a}\cap R^{\prime}$. Then
$\mathfrak{a}^{\prime}\subseteq\mathfrak{p}^{\prime}$, and using the key
property Going Up \cite[14.3]{AK}, we deduce that $\mathfrak{p}^{\prime}%
=\iota^{\ast}\left(  \mathfrak{p}\right)  $ for some $\mathfrak{p}%
\in\mathfrak{X}$ such that $\mathfrak{a}\subseteq\mathfrak{p}$. The latter
means that $sm\neq0$ for all $s\in R-\mathfrak{p}$, that is, $m/1\neq0$ in
$M_{\mathfrak{p}}$. Thus $\mathfrak{p}\in\operatorname{Supp}\left(  M\right)
$ and $\mathfrak{p}^{\prime}=\iota^{\ast}\left(  \mathfrak{p}\right)  \in
\iota^{\ast}\left(  \operatorname{Supp}\left(  M\right)  \right)  $.
\end{proof}

\begin{remark}
In the case of an integral algebra finite extension $R/R^{\prime}$ the set
$\left(  \iota^{\ast}\right)  ^{-1}\left(  \mathfrak{q}\right)  $ is finite
for every $\mathfrak{q}\in\operatorname{Supp}^{\prime}\left(  M\right)  $.
Indeed, since $R/R^{\prime}$ is module finite, it follows that so is
$R_{\mathfrak{q}}/R_{\mathfrak{q}}^{\prime}$, which in turn implies that $A/k$
is module finite either, where $A=R_{\mathfrak{q}}/\mathfrak{q}R_{\mathfrak{q}%
}$ and $k=R_{\mathfrak{q}}^{\prime}/\mathfrak{q}R_{\mathfrak{q}}^{\prime}$ is
the residue field of the local ring $R_{\mathfrak{q}}^{\prime}$. In
particular, $\dim_{k}\left(  A\right)  <\infty$, and $A$ turns out to be an
Artinian ring. By Akizuki-Hopkins Theorem \cite[19.8]{AK}, $\dim\left(
A\right)  =0$, $\operatorname{Spec}\left(  A\right)  $ consists of maximal
ideals and it is finite. But $\left(  \iota^{\ast}\right)  ^{-1}\left(
\mathfrak{q}\right)  $ is canonically identified with a subset of
$\operatorname{Spec}\left(  A\right)  $.
\end{remark}

\begin{corollary}
\label{corFD1}If $R^{\prime}\subseteq R$ is integral with Noetherian $R$ and
$M\in R$-$\operatorname{mod}$ with finite length $\ell_{R^{\prime}}\left(
M\right)  <\infty$ as an $R^{\prime}$-module then $\ell_{R}\left(  M\right)
<\infty$.
\end{corollary}

\begin{proof}
Since $\ell_{R^{\prime}}\left(  M\right)  <\infty$, it follows that
$\operatorname{Supp}^{\prime}\left(  M\right)  $ consists of maximal ideals by
Jordan-H\H{o}lder Theorem \cite[19.3]{AK}. By Proposition \ref{propEx2},
$\operatorname{Supp}^{\prime}\left(  M\right)  =\iota^{\ast}\left(
\operatorname{Supp}\left(  M\right)  \right)  $ and $\operatorname{Supp}%
\left(  M\right)  $ consists of primes lying over the maximal ideals from
$\operatorname{Supp}^{\prime}\left(  M\right)  $. By Maximality \cite[14.3]%
{AK}, $\operatorname{Supp}\left(  M\right)  $ consists of maximal ideals of
$R$ either. But $M$ is a finitely generated $R^{\prime}$-module having finite
length $\ell_{R^{\prime}}\left(  M\right)  $, therefore so is $M$ as an
$R$-module. It follows that $\ell_{R}\left(  M\right)  <\infty$ (see
\cite[19.4 or 17.16]{AK}).
\end{proof}

\begin{remark}
\label{remFD1}Let $k$ be a field, $k\subseteq R^{\prime}\subseteq R$ algebra
extensions such that $R/k$ is algebra-finite and $R/R^{\prime}$ is integral,
and let $M\in R$-$\operatorname{mod}$. Then $\ell_{R}\left(  M\right)
<\infty$ iff $\ell_{R^{\prime}}\left(  M\right)  <\infty$. By Artin-Tate lemma
\cite[16.21]{AK}, $R^{\prime}/k$ is algebra-finite too. In particular,
$R^{\prime}$ is Noetherian by Hilbert Basis. If $\left\{  M_{i}:0\leq i\leq
n\right\}  $ is a Jordan-H\H{o}lder chain of submodules in $R^{\prime}%
$-$\operatorname{mod}$ (or $R$-$\operatorname{mod}$) then $M_{i}%
/M_{i-1}=R^{\prime}/\mathfrak{m}_{i}^{\prime}$ for some maximal ideals
$\mathfrak{m}_{i}^{\prime}$, $i\geq1$. But $k\subseteq R^{\prime}%
/\mathfrak{m}_{i}^{\prime}$ is an algebra finite extension, which is finite by
Zariski Nullstellensatz \cite[15.4]{AK}. It follows that $\dim_{k}\left(
M_{i}/M_{i-1}\right)  <\infty$ for every $i$. Hence $\dim_{k}\left(  M\right)
<\infty$. Thus $\ell_{R^{\prime}}\left(  M\right)  <\infty$ iff $\dim
_{k}\left(  M\right)  <\infty$, which in turn is equivalent to $\ell
_{R}\left(  M\right)  <\infty$.
\end{remark}

\subsection{Algebraic varieties and spectra\label{SubsecAVS}}

Now let $P=k\left[  X\right]  $ be the algebra of all polynomials in $n$
variables $X=\left(  X_{1},\ldots,X_{n}\right)  $ over the field $k$. The
$P$-bimodule $P\otimes_{k}P$ has the $n$-tuple $T=\left(  T_{1},\ldots
,T_{n}\right)  $ of mutually commuting operators $T_{i}=1\otimes X_{i}%
-X_{i}\otimes1$ acting on, and we have the Koszul complex $\operatorname{Kos}%
\left(  T,P\otimes_{k}P\right)  $ augmented with the multiplication
(bi)operator $\pi:P\otimes_{k}P\rightarrow P$ provides a free $P$-bimodule
resolution of the algebra $P$. Actually, $\operatorname{Kos}\left(
T,P\otimes_{k}P\right)  $ splits as a complex of $P$-bimodules. If $M\in
P$-$\operatorname{mod}$ with the actions $x_{i}\left(  m\right)  =X_{i}\cdot
m$, $1\leq i\leq n$, then derive that the complex $0\leftarrow P\otimes
_{P}M\overset{1\otimes\pi}{\longleftarrow}\operatorname{Kos}\left(
P\otimes_{k}P,T\right)  \otimes_{P}M$ remains exact. But $P\otimes_{P}M=M$ and
$\operatorname{Kos}\left(  T,P\otimes_{k}P\right)  \otimes_{P}%
M=\operatorname{Kos}\left(  t,P\otimes_{k}M\right)  $ with $t_{i}=1\otimes
x_{i}-X_{i}\otimes1$, and $1\otimes\pi$ is reduced to the homomorphism
$\pi_{M}:P\otimes_{k}M\rightarrow M$, $\pi_{M}\left(  r\otimes_{k}m\right)
=rm$. Thus $\mathcal{P=}\operatorname{Kos}\left(  t,P\otimes_{k}M\right)  $
provides a finite free resolution of the module $M$. If $R/k$ is an algebra
finite extension with $R=k\left[  x\right]  $ for an $n$-tuple $x$, then $R$
is a quotient of $P$ and $\operatorname{Kos}\left(  t,P\otimes_{k}R\right)  $
turns out to be a free $R$-module resolution of $R$ either. Therefore it
splits in $R$-$\operatorname{mod}$.

Consider the affine space $\mathfrak{X}=\mathbb{A}_{k}^{n}$ over $k$. Then
$\mathfrak{p}\in\operatorname{res}\left(  \mathfrak{X},M\right)  $ iff there
is an open affine neighborhood $U=\operatorname{Spec}\left(  B\right)  $ of
$\mathfrak{p}$ such that $B\perp_{P}M$. It means that
\[
0\leftarrow B\otimes_{P}M\longleftarrow B\otimes_{P}\mathcal{P=}%
\operatorname{Kos}\left(  t^{B},B\otimes_{k}M\right)
\]
with $t_{i}^{B}=1\otimes x_{i}-\left(  X_{i}|_{U}\right)  \otimes1$ is exact.
Put $k\left(  \mathfrak{p}\right)  =P_{\mathfrak{p}}/\mathfrak{p}%
P_{\mathfrak{p}}=\left(  P/\mathfrak{p}\right)  _{\mathfrak{p}}%
=\operatorname{Frac}\left(  P/\mathfrak{p}\right)  $ to be residue field of
$\mathfrak{p}$. Then $k\left(  \mathfrak{p}\right)  \otimes_{P}\mathcal{P=}%
\operatorname{Kos}\left(  t^{\mathfrak{p}},k\left(  \mathfrak{p}\right)
\otimes_{k}M\right)  $ with $t_{i}^{\mathfrak{p}}=1\otimes x_{i}-X_{i}\left(
\mathfrak{p}\right)  \otimes1$, $1\leq i\leq n$.

The following key result was proved in \cite{DMMJ}.

\begin{theorem}
\label{tcal1}Let $\mathfrak{X}=\mathbb{A}_{k}^{n}$ be the affine space over
$k$, $P=k\left[  X_{1},\ldots,X_{n}\right]  $ and let $M\in P$%
-$\operatorname{mod}$. If $\mathfrak{p}\in\operatorname{res}\left(
\mathfrak{X},M\right)  $ then $k\left(  \mathfrak{p}\right)  \perp_{P}M$, that
is, the complex $\operatorname{Kos}\left(  t^{\mathfrak{p}},k\left(
\mathfrak{p}\right)  \otimes_{k}M\right)  $ is exact. Thus
\[
\operatorname{res}\left(  \mathfrak{X},M\right)  \subseteq\left\{
\mathfrak{p}\in\mathfrak{X}:k\left(  \mathfrak{p}\right)  \perp_{P}M\right\}
\subseteq\left\{  \mathfrak{p}\in\mathfrak{X}:M_{\mathfrak{p}}=\mathfrak{m}%
_{\mathfrak{p}}M_{\mathfrak{p}}\right\}  ,
\]
where $\mathfrak{m}_{\mathfrak{p}}=\operatorname{rad}\left(  P_{\mathfrak{p}%
}\right)  $. If $M$ is a Noetherian $P$-module then $\operatorname{res}\left(
\mathfrak{X},M\right)  =\left\{  \mathfrak{p}\in\mathfrak{X}:k\left(
\mathfrak{p}\right)  \perp_{P}M\right\}  $.
\end{theorem}

If $\mathfrak{p=}\left\langle X-a\right\rangle $ is a maximal ideal
corresponding to a closed point $a\in\mathfrak{X}$ then $k\left(
\mathfrak{p}\right)  =\operatorname{Frac}\left(  P/\mathfrak{p}\right)
=k\mathfrak{\ }$and $k\left(  \mathfrak{p}\right)  \otimes_{P}\mathcal{P=}%
\operatorname{Kos}\left(  t^{\mathfrak{p}},M\right)  $ with $t_{i}%
^{\mathfrak{p}}=x_{i}-a_{i}$, $1\leq i\leq n$, that is, $t^{\mathfrak{p}}%
=x-a$. Thus if $M$ is a finitely generated $P$-module, then $\sigma\left(
\mathfrak{X},M\right)  =\operatorname{Supp}\left(  M\right)  $ (see
(\ref{SV})), and $a\in\operatorname{res}\left(  \mathfrak{X},M\right)  $ iff
$\operatorname{Kos}\left(  x-a,M\right)  $ is exact thanks to Theorem
\ref{tcal1}.

\begin{lemma}
\label{lemEq}Let $R/k$ be an algebra finite extension of the field $k$,
$\mathfrak{X=}\operatorname{Spec}\left(  R\right)  $, $M$ an $R$-module, and
let $R=k\left[  x\right]  $ for an $n$-tuple $x$. Then $M$ is a $P$-module,
$\mathfrak{X}\subseteq\mathbb{A}_{k}^{n}$ up to a homeomorphism,
$\operatorname{Supp}_{P}\left(  M\right)  =\operatorname{Supp}_{R}\left(
M\right)  $, $\sigma\left(  \mathfrak{X},M\right)  =\sigma\left(
\mathbb{A}_{k}^{n},M\right)  $, $\sigma_{\operatorname{p}}\left(
\mathfrak{X},M\right)  =\sigma_{\operatorname{p}}\left(  \mathbb{A}_{k}%
^{n},M\right)  $, and
\[
\sigma_{\operatorname{p}}\left(  x,M\right)  =\sigma_{\operatorname{p}}\left(
\mathbb{A}_{k}^{n},M\right)  \cap\mathbb{A}^{n}.
\]
If $M$ is a Noetherian module then
\[
\sigma\left(  x,M\right)  =\sigma\left(  \mathbb{A}_{k}^{n},M\right)
\cap\mathbb{A}^{n}%
\]
and it is a nonempty closed subset of $\mathbb{A}^{n}$ whose closure in the
scheme $\mathbb{A}_{k}^{n}$ is reduced to $\sigma\left(  \mathfrak{X}%
,M\right)  $.
\end{lemma}

\begin{proof}
Put $R=P/\mathfrak{a}$ for some ideal $\mathfrak{a\subseteq}P$. Then $M$ turns
out to be a $P$-module with the actions $X_{i}u=x_{i}u$, $u\in M$, $1\leq
i\leq n$ and $\mathfrak{a\subseteq}\operatorname{Ann}_{P}\left(  M\right)  $.
Note that $\mathfrak{X=}V\left(  \mathfrak{a}\right)  \subseteq
\operatorname{Spec}\left(  P\right)  =\mathbb{A}_{k}^{n}$ up to a
homeomorphism. Moreover, $M_{\mathfrak{p}}=M_{\mathfrak{p}^{\prime}}$ whenever
$\mathfrak{p\in}V\left(  \mathfrak{a}\right)  $ and $\mathfrak{p}^{\prime
}=\mathfrak{p/a}$. Indeed, $R_{\mathfrak{p}^{\prime}}=\left(  R-\mathfrak{p}%
^{\prime}\right)  ^{-1}R=\left(  P-\mathfrak{p}\right)  ^{-1}R=R_{\mathfrak{p}%
}=P_{\mathfrak{p}}\otimes_{P}R$ (see \cite[(11.15.1)]{AK}) and%
\[
M_{\mathfrak{p}^{\prime}}=R_{\mathfrak{p}^{\prime}}\otimes_{R}%
M=P_{\mathfrak{p}}\otimes_{P}R\otimes_{R}M=P_{\mathfrak{p}}\otimes
_{P}M=M_{\mathfrak{p}}.
\]
In particular, $\operatorname{Supp}_{P}\left(  M\right)  =\operatorname{Supp}%
_{R}\left(  M\right)  $. But $R/\mathfrak{p}^{\prime}=\left(  P/\mathfrak{a}%
\right)  /\left(  \mathfrak{p/a}\right)  =P/\mathfrak{p}$, which means that
$R/\mathfrak{p}^{\prime}\hookrightarrow M$ iff $P/\mathfrak{p}\hookrightarrow
M$. Thus $\operatorname{Ass}_{P}\left(  M\right)  =\operatorname{Ass}%
_{R}\left(  M\right)  $ (see \cite[17.3]{AK}) or $\sigma_{\operatorname{p}%
}\left(  \mathbb{A}_{k}^{n},M\right)  =\sigma_{\operatorname{p}}\left(
\mathfrak{X},M\right)  $. Using (\ref{cl}), we obtain that $\sigma\left(
\mathfrak{X},M\right)  =\operatorname{Supp}_{R}\left(  M\right)
^{-}=\operatorname{Supp}_{P}\left(  M\right)  ^{-}=\sigma\left(
\mathbb{A}_{k}^{n},M\right)  $.

The variety $\mathbb{A}^{n}$ is the set of all closed points in $\mathbb{A}%
_{k}^{n}$ obtained by means of the natural homeomorphism $\beta:\mathbb{A}%
^{n}\rightarrow\mathbb{A}_{k}^{n}$, $\beta\left(  a\right)  =\left\langle
X-a\right\rangle $ (see \cite[2.2.6]{Harts}) onto the closed points. Put
$\sigma_{c}\left(  \mathbb{A}_{k}^{n},M\right)  $ to be $\sigma\left(
\mathbb{A}_{k}^{n},M\right)  \cap\mathbb{A}^{n}$ or $\beta^{-1}\left(
\sigma\left(  \mathbb{A}_{k}^{n},M\right)  \right)  $. Since $\sigma\left(
\mathbb{A}_{k}^{n},M\right)  $ is closed, the set $\sigma_{c}\left(
\mathbb{A}_{k}^{n},M\right)  $ turns out to be a closed subset of
$\mathbb{A}^{n}$. If $\mathfrak{p}\in\sigma\left(  \mathbb{A}_{k}%
^{n},M\right)  $ then $V\left(  \mathfrak{p}\right)  =\overline{\left\{
\mathfrak{p}\right\}  }\subseteq\sigma\left(  \mathbb{A}_{k}^{n},M\right)  $
and $\mathfrak{m}\in\sigma_{c}\left(  \mathbb{A}_{k}^{n},M\right)  $ whenever
$\mathfrak{m}\in V\left(  \mathfrak{p}\right)  \cap\mathbb{A}^{n}$ is a
maximal ideal. By Hilbert Nullstellensatz \cite[15.7]{AK}, $\sigma\left(
\mathbb{A}_{k}^{n},M\right)  $ is the closure of $\sigma_{c}\left(
\mathbb{A}_{k}^{n},M\right)  $ in $\mathbb{A}_{k}^{n}$.

Notice that $\mathfrak{m\in}\sigma_{\operatorname{p}}\left(  \mathbb{A}%
_{k}^{n},M\right)  \cap\mathbb{A}^{n}$ iff $\mathfrak{m=}\left\langle
X-a\right\rangle =\operatorname{Ann}\left(  u\right)  $ for some point
$a\in\mathbb{A}^{n}$ and $u\neq0$ from $M$. The latter means that
$x_{i}u=a_{i}u$ for all $i$, which means that $a\in\sigma_{\operatorname{p}%
}\left(  x,M\right)  $ is a joint eigenvalue of $x$ with the related
eigenvector $u$.

Finally, assume that $M$ is a Noetherian $R$-module (or $P$-module). Based on
Theorem \ref{tcal1}, we conclude that $\operatorname{res}\left(
\mathbb{A}_{k}^{n},M\right)  \cap\mathbb{A}^{n}=\left\{  a\in\mathbb{A}%
^{n}:k\left(  \left\langle X-a\right\rangle \right)  \perp_{P}M\right\}  $ or
\[
\sigma_{c}\left(  \mathbb{A}_{k}^{n},M\right)  =\left\{  a\in\mathbb{A}%
^{n}:\operatorname{Kos}\left(  x-a,M\right)  \text{ is not exact}\right\}
=\sigma\left(  x,M\right)
\]
is the Taylor spectrum of the tuple $x$ on $M$. Hence $\sigma\left(
x,M\right)  $ is a closed subset of $\mathbb{A}^{n}$ and $\sigma\left(
x,M\right)  ^{-}=\sigma\left(  \mathbb{A}_{k}^{n},M\right)  =\sigma\left(
\mathfrak{X},M\right)  $. By Corollary \ref{corNE1}, $\sigma\left(
x,M\right)  \neq\varnothing$.
\end{proof}

\begin{corollary}
\label{corFLM1}Let $R/k$ be an algebra finite extension of the field $k$, $M$
an $R$-module with $\ell_{R}\left(  M\right)  <\infty$, and let $x$ be an
$n$-tuple in $R$ such that $k\left[  x\right]  \subseteq R$ is integral. Then
$\sigma_{\operatorname{p}}\left(  x,M\right)  =\sigma_{\operatorname{p}%
}\left(  \mathbb{A}_{k}^{n},M\right)  =\sigma\left(  \mathbb{A}_{k}%
^{n},M\right)  =\sigma\left(  x,M\right)  $ and the $k\left[  x\right]
$-module structure on $M$ is triangularizable. In this case, $i\left(
x-a\right)  =0$ for all $a\in\mathbb{A}^{n}$.
\end{corollary}

\begin{proof}
Based on Corollary \ref{corFD1}, we can assume that $R=k\left[  x\right]  $
(see also Remark \ref{remFD1}). In this case, $\operatorname{Ass}\left(
M\right)  =\operatorname{Supp}\left(  M\right)  =\operatorname{Max}\left(
M\right)  \subseteq\mathbb{A}^{n}$ \cite[19.4]{AK}. It remains to use Lemma
\ref{lemEq}. Further, all the gaps of a Jordan-H\"{o}lder chain is $k$ by
Zariski Nullstellensatz. There is a $k$-basis $\omega=\left(  \omega
_{1},\ldots,\omega_{s}\right)  $ for $M$ such that $\operatorname{Ass}\left(
M\right)  =\left\{  a^{\left(  1\right)  },\ldots,a^{\left(  s\right)
}\right\}  $ and
\[
x=\left[
\begin{array}
[c]{ccc}%
a^{\left(  1\right)  } &  & 0\\
& \ddots & \\
\ast &  & a^{\left(  s\right)  }%
\end{array}
\right]  ,
\]
which means that the $k\left[  x\right]  $-module structure on $M$ is triangularizable.

Finally, for every $a\in\mathbb{A}^{n}$ the complex $\operatorname{Kos}\left(
x-a,M\right)  $ consists of finite dimensional $k$-vector space. It follows
that its index coincides with its Euler characteristics. Hence
\begin{align*}
i\left(  x-a\right)   &  =\sum_{j=0}^{n}\left(  -1\right)  ^{j+1}\dim
_{k}\left(  H_{j}\left(  x-a,M\right)  \right)  =\sum_{j=0}^{n}\left(
-1\right)  ^{j+1}\dim_{k}\left(  M\otimes_{k}\wedge^{j}k^{n}\right) \\
&  =\dim_{k}\left(  M\right)  \sum_{j=0}^{n}\left(  -1\right)  ^{j+1}%
\dbinom{n}{j}=0,
\end{align*}
that is, $i\left(  x-a\right)  =0$ for all $a\in\mathbb{A}^{n}$.
\end{proof}

The following reformulation of Theorem \ref{propEx1} results in spectral
mapping theorem for the joint point spectrum.

\begin{theorem}
\label{thsmtps}Let $R=k\left[  x\right]  $ be an algebra finite extension with
an $n$-tuple $x$, and let $M$ be an $R$-module. Then
\[
\sigma_{\operatorname{p}}\left(  p\left(  x\right)  ,M\right)  =p^{\ast
}\left(  \sigma_{\operatorname{p}}\left(  \mathbb{A}_{k}^{n},M\right)
\right)  \cap\mathbb{A}^{m}%
\]
for every $m$-tuple $p\left(  x\right)  $ from $R$. If $k\left[  p\left(
x\right)  \right]  \subseteq R$ is integral then $\sigma_{\operatorname{p}%
}\left(  p\left(  x\right)  ,M\right)  =p\left(  \sigma_{\operatorname{p}%
}\left(  x,M\right)  \right)  $ and $\sigma\left(  p\left(  x\right)
,M\right)  $ is a closed set.
\end{theorem}

\begin{proof}
Put $\mathfrak{X=}\operatorname{Spec}\left(  R\right)  $, $R^{\prime}=k\left[
p\left(  x\right)  \right]  $ to be a subalgebra of $R$ generated by an
$m$-tuple $p\left(  x\right)  $ from $R$ with the inclusion map $\iota
:R^{\prime}\rightarrow R$, and $\iota^{\ast}:\mathfrak{X}\rightarrow
\mathfrak{X}^{\prime}$ is the natural map with $\mathfrak{X}^{\prime
}=\operatorname{Spec}\left(  R^{\prime}\right)  $. There are canonical
surjective maps $k\left[  X\right]  \rightarrow R$ and $k\left[  Y\right]
\rightarrow R^{\prime}$, where $Y$ is an $m$-tuple. In particular,
$\mathfrak{X\subseteq}\mathbb{A}_{k}^{n}$ and $\mathfrak{X}^{\prime}%
\subseteq\mathbb{A}_{k}^{m}$ up to a canonical homeomorphisms. We have the
polynomial ring map $p:k\left[  Y\right]  \rightarrow k\left[  X\right]  $,
$Y_{j}\mapsto p_{j}\left(  X\right)  $, $1\leq j\leq m$, which in turn
generates the map $p^{\ast}:\mathbb{A}_{k}^{n}\rightarrow\mathbb{A}_{k}^{m}$
of affine spaces. If $\mathfrak{m}_{a}=\left\langle X-a\right\rangle
\in\mathbb{A}^{n}$ then $p^{\ast}\left(  \mathfrak{m}_{a}\right)
=p^{-1}\left(  \mathfrak{m}_{a}\right)  \supseteq\left\langle Y-p\left(
a\right)  \right\rangle =\mathfrak{m}_{p\left(  a\right)  }$, for $p\left(
\left\langle Y-p\left(  a\right)  \right\rangle \right)  \subseteq\left\langle
p\left(  X\right)  -p\left(  a\right)  \right\rangle \subseteq\mathfrak{m}%
_{a}$. Since $\mathfrak{m}_{p\left(  a\right)  }\in\mathbb{A}^{m}$ is maximal,
we conclude that $p^{\ast}\left(  \mathfrak{m}_{a}\right)  =\mathfrak{m}%
_{p\left(  a\right)  }$. Thus $p^{\ast}\left(  \mathbb{A}^{n}\right)
\subseteq\mathbb{A}^{m}$ and $p^{\ast}:\mathbb{A}^{n}\rightarrow\mathbb{A}%
^{m}$, $a\mapsto p\left(  a\right)  $ is a polynomial morphism. Moreover,
$p^{\ast}\left(  \mathbb{A}_{k}^{n}\right)  \cap\mathbb{A}^{m}=p^{\ast}\left(
\mathbb{A}^{n}\right)  $. Indeed, if $p^{\ast}\left(  \mathfrak{q}\right)
=\mathfrak{m}_{b}$ for some $\mathfrak{q}\in\mathbb{A}_{k}^{n}$ and
$b\in\mathbb{A}^{m}$, then $\mathfrak{q\subseteq m}_{a}$ for some
$a\in\mathbb{A}^{n}$ and $\mathfrak{m}_{b}\subseteq p^{\ast}\left(
\mathfrak{m}_{a}\right)  =\mathfrak{m}_{p\left(  a\right)  }$, which implies
that $\mathfrak{m}_{b}=\mathfrak{m}_{p\left(  a\right)  }$ or $b=p\left(
a\right)  $.

Using the duality correspondence between affine schemes and rings, we obtain
the following commutative diagrams
\[%
\begin{array}
[c]{ccccccc}%
k\left[  Y\right]  & \overset{p}{\longrightarrow} & k\left[  X\right]  &  &
\mathbb{A}_{k}^{m} & \overset{p^{\ast}}{\longleftarrow} & \mathbb{A}_{k}^{n}\\
\downarrow &  & \downarrow & \Leftrightarrow & \uparrow &  & \uparrow\\
R^{\prime} & \overset{\iota}{\longrightarrow} & R &  & \mathfrak{X}^{\prime} &
\overset{\iota^{\ast}}{\longleftarrow} & \mathfrak{X}%
\end{array}
\]
supporting each other. Using Theorem \ref{propEx1} and Lemma \ref{lemEq}, we
deduce that
\begin{align*}
\sigma_{\operatorname{p}}\left(  p\left(  x\right)  ,M\right)   &
=\sigma_{\operatorname{p}}\left(  \mathbb{A}_{k}^{m},M\right)  \cap
\mathbb{A}^{m}=\sigma_{\operatorname{p}}\left(  \mathfrak{X}^{\prime
},M\right)  \cap\mathbb{A}^{m}=\iota^{\ast}\left(  \sigma_{\operatorname{p}%
}\left(  \mathfrak{X},M\right)  \right)  \cap\mathbb{A}^{m}\\
&  =p^{\ast}\left(  \sigma_{\operatorname{p}}\left(  \mathbb{A}_{k}%
^{n},M\right)  \right)  \cap\mathbb{A}^{m}.
\end{align*}
If $R^{\prime}\subseteq R$ is integral, then $\sigma_{\operatorname{p}}\left(
\mathfrak{X}^{\prime},M\right)  \cap\mathbb{A}^{m}=\iota^{\ast}\left(
\sigma_{\operatorname{p}}\left(  \mathfrak{X},M\right)  \cap\mathbb{A}%
^{n}\right)  $ by Maximality \cite[14.3]{AK} and Theorem \ref{propEx1}. Using
again Lemma \ref{lemEq}, we deduce that $\sigma_{\operatorname{p}}\left(
p\left(  x\right)  ,M\right)  =p^{\ast}\left(  \sigma_{\operatorname{p}%
}\left(  \mathbb{A}_{k}^{n},M\right)  \cap\mathbb{A}^{n}\right)  =p\left(
\sigma_{\operatorname{p}}\left(  x,M\right)  \right)  $.

Finally, prove that Taylor spectrum $\sigma\left(  p\left(  x\right)
,M\right)  $ is a closed set. Since $R/R^{\prime}$ is integral, $M$ is module
finite over $R^{\prime}$. By Lemma \ref{lemEq}, $\sigma\left(  p\left(
x\right)  ,M\right)  =\sigma\left(  \mathbb{A}_{k}^{m},M\right)
\cap\mathbb{A}^{m}$ is a closed set.
\end{proof}

\begin{remark}
The same argument from the proof of Theorem \ref{thsmtps}, and Proposition
\ref{propEx2}, result in the spectral mapping property $\sigma\left(  p\left(
x\right)  ,M\right)  =p\left(  \sigma\left(  x,M\right)  \right)  $ for Taylor
spectrum whenever $k\left[  p\left(  x\right)  \right]  \subseteq R$ is
integral and $M$ is Noetherian. But as we have seen above in Theorem
\ref{thSMP}, the formula holds in the general case of all tuples $p\left(
x\right)  $.
\end{remark}

If $\mathfrak{a\subseteq}P$ is a radical ideal with the related algebraic set
$Y=Z\left(  \mathfrak{a}\right)  =V\left(  \mathfrak{a}\right)  \cap
\mathbb{A}^{n}$ (that is $I\left(  Y\right)  =\mathfrak{a}$), then the
coordinate ring $R$ of $Y$ is the reduced to the Noetherian ring
$P/\mathfrak{a}$. Put $\mathfrak{Y=}\operatorname{Spec}\left(  R\right)  $. By
(\ref{SV}), we obtain that $\sigma\left(  \mathbb{A}_{k}^{n},R\right)
=\operatorname{Supp}\left(  R\right)  =V\left(  \mathfrak{a}\right)
=\mathfrak{Y}$ is the set of all subvarieties of $Y$, and by Lemma
\ref{lemEq}, $Y=\mathfrak{Y}\cap\mathbb{A}^{n}=\sigma\left(  x,R\right)  $ is
the Taylor spectrum of the operator tuple $x$ in $R$, which consists of all
coordinate functions. For a Noetherian $k\left[  x\right]  $-module $M$, we
obtain that the tuple $x$ on the module $M$ is triangularizable whose diagonal
entries are varieties (see \cite[4.4]{DMMJ}).

Let $p=\left(  p_{1},\ldots,p_{m}\right)  :Y\rightarrow\mathbb{A}^{m}$ be a
morphism given by means of a ring extension $R^{\prime}=k\left[  y\right]
\subseteq R$ with $y=p\left(  x\right)  $ to be an $m$-tuple. In particular,
we have the spectrum $\sigma\left(  \mathbb{A}_{k}^{m},R\right)  $ of the
$R^{\prime}$-module $R$. If $R/R^{\prime}$ is integral then $p^{\ast}$ turns
out to be a finite morphism (see \cite[2.3]{Harts}). Using (\ref{2}), we
deduce that $\sigma\left(  \mathbb{A}_{k}^{m},R\right)  \cap\mathbb{A}%
^{m}=p\left(  Y\right)  ^{-}$ (see \cite{DMMJ} for the details).

\begin{corollary}
\label{propASMT}If $Y\subseteq\mathbb{A}^{n}$ is an algebraic set over $k$
with its coordinate ring $R$, then $\sigma_{\operatorname{p}}\left(
x,R\right)  $ is the set $Y_{is}$ of all isolated points of $Y$. If
$p:Y\rightarrow\mathbb{A}^{m}$ is a finite morphism given by a ring extension
$k\left[  p\left(  x\right)  \right]  \subseteq R$, then $\sigma\left(
p\left(  x\right)  ,R\right)  =\sigma\left(  \mathbb{A}_{k}^{m},R\right)
\cap\mathbb{A}^{m}=p\left(  Y\right)  $ and $\sigma_{\operatorname{p}}\left(
p\left(  x\right)  ,R\right)  =p\left(  Y_{is}\right)  $.
\end{corollary}

\begin{proof}
Note that $a\in\sigma_{\operatorname{p}}\left(  x,R\right)  $ iff
$\left\langle X-a\right\rangle g\left(  X\right)  \subseteq\mathfrak{a}$ for
some $g\left(  X\right)  \notin\mathfrak{a}$. It means that $Y=Z\left(
\mathfrak{a}\right)  \subseteq\left\{  a\right\}  \cup Z\left(  g\right)  $
and $Y\nsubseteq Z\left(  g\right)  $. Since $a\in Y$, we conclude that
$a\notin Z\left(  g\right)  $ and $\left\{  a\right\}  =Y\cap\left(
\mathbb{A}^{n}-Z\left(  g\right)  \right)  $ is open in $Y$, that is,
$\left\{  a\right\}  $ is an isolated point. Conversely, take $a\in Y_{is}$.
Since the complements to hypersurfaces from $\mathbb{A}^{n}$ is a topology
base in $\mathbb{A}^{n}$, it follows that $\left\{  a\right\}  =Y\cap\left(
\mathbb{A}^{n}-Z\left(  g\right)  \right)  $ for a certain hypersurface
$Z\left(  g\right)  $. Then $Y=Z\left(  \mathfrak{a}\right)  \subseteq\left\{
a\right\}  \cup Z\left(  g\right)  $, $Y\nsubseteqq Z\left(  g\right)  $ and%
\begin{align*}
\left\langle X-a\right\rangle g  &  \subseteq\left\langle X-a\right\rangle
\cap\sqrt{g}=I\left(  \left\{  a\right\}  \right)  \cap I\left(  Z\left(
g\right)  \right)  =I\left(  \left\{  a\right\}  \cup Z\left(  g\right)
\right) \\
&  \subseteq I\left(  Y\right)  =\sqrt{\mathfrak{a}}=\mathfrak{a}\text{,}%
\end{align*}
which means that $a\in\sigma_{\operatorname{p}}\left(  x,R\right)  $. Further,
if $p$ is a finite morphism then $R$ is a Noetherian $k\left[  p\left(
x\right)  \right]  $-module and $\sigma\left(  p\left(  x\right)  ,R\right)
=\sigma\left(  \mathbb{A}_{k}^{m},R\right)  \cap\mathbb{A}^{m}$ is closed (see
Lemma \ref{lemEq}). Using Lemma \ref{lemEq} and Theorem \ref{thSMP}, we derive
that $\sigma\left(  \mathbb{A}_{k}^{m},R\right)  \cap\mathbb{A}^{m}%
=\sigma\left(  p\left(  x\right)  ,R\right)  =p\left(  \sigma\left(
x,R\right)  \right)  =p\left(  Y\right)  $, and $\sigma_{\operatorname{p}%
}\left(  p\left(  x\right)  ,R\right)  =p\left(  \sigma_{\operatorname{p}%
}\left(  x,R\right)  \right)  =p\left(  Y_{is}\right)  $ thanks to Theorem
\ref{thsmtps}.
\end{proof}

\section{Koszul homology groups of a variety\label{Sec3}}

In the present section we focus on the case of a module $M$ which is the
coordinate ring $R$ of a variety $Y$, and investigate its Koszul homology
groups. We are targeting to find out a link between the dimensions of Koszul
homology groups of the standard tuple in $R$ and the dimension $d=\dim\left(
Y\right)  $.

\subsection{Multiplicity formula of Serre}

Let $P=k\left[  X\right]  $ be the polynomial algebra over the field $k$ with
$n$-tuple $X=\left(  X_{1},\ldots,X_{n}\right)  $, and put $h\left(  r\right)
=\dim_{k}\left(  \mathfrak{t}^{r}/\mathfrak{t}^{r+1}\right)  =\dbinom
{r+n-1}{n-1}$ to be the Hilbert polynomial ($\deg\left(  h\right)  =n-1$) of
the graded $P$-algebra $G^{\circ}P$ associated with the filtration $\left\{
\mathfrak{t}^{r}\right\}  $, where $\mathfrak{t=}\left\langle X\right\rangle
\subseteq P$ is the maximal ideal generated by $X$, and $\dbinom{z}{s}%
=\dfrac{1}{s!}z\left(  z-1\right)  \cdots\left(  z-s+1\right)  $ is the
binomial coefficient. If $\mathfrak{s}\left(  P,r\right)  =\dim_{k}\left(
P/\mathfrak{t}^{r}\right)  $ is the Samuel polynomial of the filtration then
$\mathfrak{s}\left(  P,r\right)  =\sum_{i=0}^{r-1}h\left(  i\right)
=\sum_{i=0}^{r-1}\dbinom{i+n-1}{n-1}=\dbinom{r+n-1}{n}$. If $M$ is a
$P$-module given by an $n$-tuple $x$ of $k$-linear maps on $M$, then
$P/\mathfrak{t}^{r}\otimes_{k}M=M^{\binom{r+n-1}{n}}$ and there are nilpotent
(shift) operators $X_{i}\otimes1$ acting on $M^{\binom{r+n-1}{n}}$. Put
$x_{i}^{\left(  r\right)  }=\left(  1\otimes x_{i}\right)  -X_{i}\otimes1$ and
$x^{\left(  r\right)  }=\left(  x_{1}^{\left(  r\right)  },\ldots
,x_{n}^{\left(  r\right)  }\right)  $ to be an $n$-tuple of mutually commuting
operators on the inflation $M^{\binom{r+n-1}{n}}$, which defines its
$P$-module structure. For every $a\in\mathbb{A}^{n}$ the $P$-module structure
on $M^{\binom{r+n-1}{n}}$ given by $x^{\left(  r\right)  }-a$ is the same
inflation of the $P$-module $M$ given by the tuple $x-a$. Just notice that
$x_{i}^{\left(  r\right)  }-a_{i}=1\otimes\left(  x_{i}-a_{i}\right)
-X_{i}\otimes1=\left(  x_{i}-a_{i}\right)  ^{\left(  r\right)  }$ for all $i$,
or $x^{\left(  r\right)  }-a=\left(  x-a\right)  ^{\left(  r\right)
}=1\otimes\left(  x-a\right)  -X\otimes1$. In this case one can also replace
$\mathfrak{t}$ by $\left\langle X-a\right\rangle $.

\begin{lemma}
\label{lemIM1}If $M$ is a Noetherian $P$-module then so is $M^{\binom
{r+n-1}{n}}$ and $\sigma\left(  x,M\right)  =\sigma\left(  x^{\left(
r\right)  },M^{\binom{r+n-1}{n}}\right)  $ for all $r$. Thus $i\left(  \left(
x-a\right)  ^{\left(  r\right)  }\right)  <\infty$ for all $a\in\mathbb{A}%
^{n}$ and $r\geq1$.
\end{lemma}

\begin{proof}
For every $r\geq1$ consider the following canonical exact sequence%
\begin{equation}
0\leftarrow P/\mathfrak{t}^{r}\longleftarrow P/\mathfrak{t}^{r+1}%
\longleftarrow\mathfrak{t}^{r}/\mathfrak{t}^{r+1}\leftarrow0, \label{sR}%
\end{equation}
in $P$-$\operatorname{mod}$. Notice that the $P$-module structure of
$\mathfrak{t}^{r}/\mathfrak{t}^{r+1}$ is reduced to its $k$-vector space one
with $\dim_{k}\left(  \mathfrak{t}^{r}/\mathfrak{t}^{r+1}\right)  =h\left(
r\right)  $. Using (\ref{sR}), we generate the following exact sequences
\[
0\leftarrow P/\mathfrak{t}^{r}\otimes_{k}M\longleftarrow P/\mathfrak{t}%
^{r+1}\otimes_{k}M\longleftarrow\mathfrak{t}^{r}/\mathfrak{t}^{r+1}\otimes
_{k}M\leftarrow0
\]
such that the $P$-module structure of $\mathfrak{t}^{r}/\mathfrak{t}%
^{r+1}\otimes_{k}M$ is diagonal and therefore it is Noetherian. By induction
on $r$ we deduce that $P/\mathfrak{t}^{r}\otimes_{k}M$ is Noetherian. By
Spectral Mapping Theorem \ref{thSMP}, we have
\begin{align*}
\sigma\left(  x^{\left(  r\right)  },M^{\binom{r+n-1}{n}}\right)   &
=\sigma\left(  1\otimes x-X\otimes1,M^{\binom{r+n-1}{n}}\right) \\
&  =\left\{  a-b:\left(  a,b\right)  \in\sigma\left(  \left(  1\otimes
x,X\otimes1\right)  ,M^{\binom{r+n-1}{n}}\right)  \right\}
\end{align*}
and $b\in\sigma\left(  X\otimes1,M^{\binom{r+n-1}{n}}\right)  $. For every
$a\in\sigma\left(  1\otimes x,M^{\binom{r+n-1}{n}}\right)  $ there is $b$ such
that $\left(  a,b\right)  \in\sigma\left(  \left(  1\otimes x,X\otimes
1\right)  ,M^{\binom{r+n-1}{n}}\right)  $ (see Theorem \ref{thNP1}). Moreover,
$b_{i}\in\sigma\left(  \left(  X_{i}\otimes1\right)  |M^{\binom{r+n-1}{n}%
}\right)  =\left\{  0\right\}  $, for $X_{i}\otimes1$ is nilpotent. Hence
$\sigma\left(  x^{\left(  r\right)  },M^{\binom{r+n-1}{n}}\right)
=\sigma\left(  1\otimes x,M^{\binom{r+n-1}{n}}\right)  $. Finally,%
\[
\operatorname{Kos}\left(  1\otimes\left(  x-a\right)  ,M^{\binom{r+n-1}{n}%
}\right)  =P/\mathfrak{t}^{r}\otimes_{k}\operatorname{Kos}\left(
x-a,M\right)  ,
\]
which is exact iff so is $\operatorname{Kos}\left(  x-a,M\right)  $, that is,
$\sigma\left(  1\otimes x,M^{\binom{r+n-1}{n}}\right)  =\sigma\left(
x,M\right)  $. It remains to use Theorem \ref{thNP11}.
\end{proof}

Let $A/k$ be a local Noetherian algebra with an $n$-tuple $x$ contained in the
maximal ideal of $A$, and let $M$ be a Noetherian $A$-module such that
$\ell\left(  M/xM\right)  <\infty$, that is, $x$ is a system of parameters for
$M$. The filtration $\left\{  \left\langle x\right\rangle ^{r}M\right\}  $ in
$M$ defines the Samuel polynomial $\mathfrak{s}\left(  M,r\right)
=\ell\left(  M/\left\langle x\right\rangle ^{r}M\right)  $ whose degree is at
most $n$, and $\mathfrak{s}\left(  M,r\right)  =e_{x,M}\left(  n\right)
\dfrac{r^{n}}{n!}+q\left(  r\right)  $ with $e_{x,M}\left(  n\right)
=\Delta^{n}\mathfrak{s}\left(  M,\circ\right)  $ ($n$th difference operator)
and $\deg\left(  q\right)  <n$. By Theorem \ref{thNP11}, $\ell\left(
H_{p}\left(  x,M\right)  \right)  =\dim_{k}\left(  H_{p}\left(  x,M\right)
\right)  $ for all $p$, and the Euler characteristic $\chi\left(  x,M\right)
=\sum_{p=0}^{n}\left(  -1\right)  ^{p}\ell\left(  H_{p}\left(  x,M\right)
\right)  $ from \cite[4.A.3]{Se} is reduced to the opposite index $-i\left(
x\right)  $ of the tuple $x$. The multiplicity formula of Serre \cite[4.A.3.
Theorem 1]{Se} is expressed in the following way
\begin{equation}
i\left(  x\right)  =-e_{x,M}\left(  n\right)  . \label{Se}%
\end{equation}
If $\dim\left(  M\right)  =n$ then $i\left(  x\right)  <0$, whereas $i\left(
x\right)  =0$ in the case of $\dim\left(  M\right)  <n$.

\subsection{The numerical Tor-polynomial}

Now let $\mathfrak{p}\subseteq P$ be a nonzero prime ideal of a variety
$Y\subseteq\mathbb{A}^{n}$ with its coordinate ring $R=P/\mathfrak{p}$. The
actions of $X_{i}$ on $R$ are denoted by $x_{i}$, and $x=\left(  x_{1}%
,\ldots,x_{n}\right)  $ is an operator tuple on $R$. Thus $R=k\left[
x\right]  $ is an algebra finite extension of $k$ which is a domain. In
particular, $\sigma_{\operatorname{p}}\left(  x,R\right)  =\varnothing$ or
$H_{n}\left(  x\right)  =0$. For every $r$ we have the inflated $P$-module
$R^{\binom{r+n-1}{n}}$ given by means of the $n$-tuple $x^{\left(  r\right)
}$. Using Lemma \ref{lemEq} and Lemma \ref{lemIM1}, we deduce that
\[
Y=\operatorname{Spec}\left(  R\right)  \cap\mathbb{A}^{n}=\sigma\left(
\mathbb{A}_{k}^{n},R\right)  \cap\mathbb{A}^{n}=\sigma\left(  x,R\right)
=\sigma\left(  x^{\left(  r\right)  },R^{\binom{r+n-1}{n}}\right)
\]
for all $r$. For every $a\in Y$ the homology groups $H_{i}\left(  x^{\left(
r\right)  }-a,R^{\binom{r+n-1}{n}}\right)  $ denoted by $H_{i}\left(
x^{\left(  r\right)  }-a\right)  $ are finite dimensional $k$-vector spaces
(Lemma \ref{lemIM1}). Put $d_{p}\left(  a^{\left(  r\right)  }\right)
=\dim_{k}\left(  H_{p}\left(  x^{\left(  r\right)  }-a\right)  \right)  $ and
define the tuple $d\left(  a^{\left(  r\right)  }\right)  =\left(
d_{p}\left(  a^{\left(  r\right)  }\right)  \right)  _{p}\in\mathbb{Z}%
_{+}^{n+1}$. In particular, $i\left(  x^{\left(  r\right)  }-a\right)
=\sum_{p=0}^{n}\left(  -1\right)  ^{p+1}d_{p}\left(  a^{\left(  r\right)
}\right)  $ and it defines the function $i_{Y}^{\left(  r\right)
}:Y\rightarrow\mathbb{Z}$, $a\mapsto i\left(  x^{\left(  r\right)  }-a\right)
$. We also put $i_{Y}$ instead of $i_{Y}^{\left(  1\right)  }$.

\begin{lemma}
\label{lemAp}If $Y\subseteq\mathbb{A}^{n}$ is a variety then $i_{Y}%
=-\delta_{nd}$, where $d=\dim\left(  Y\right)  $.
\end{lemma}

\begin{proof}
If $Y=\mathbb{A}^{n}$ then $\operatorname{Kos}\left(  X-a,P\right)  $ is a
free $P$-module resolution for the $P$-module $k$ (see Corollary \ref{corRH}),
that is, $H_{0}\left(  X-a,P\right)  =k$ and $H_{i}\left(  X-a,P\right)  =0$,
$i\geq1$. In particular, $d_{0}\left(  a\right)  =1$, $d_{i}\left(  a\right)
=0$, $i\geq1$, and $i_{Y}\left(  a\right)  =-1$ for all $a\in Y$.

Now assume that $Y\subseteq\mathbb{A}^{n}$ is a variety of dimension $d<n$,
and $a=0\in Y$, which responds to the maximal ideal $\left\langle
x\right\rangle \subseteq R$. Then $R_{\left\langle x\right\rangle }$ is a
local Noetherian algebra with its maximal ideal $\mathfrak{m=}\left\langle
x\right\rangle R_{\left\langle x\right\rangle }$. By Lemma \ref{lemNP0}, we
have $H_{p}\left(  x,R\right)  =H_{p}\left(  x,R\right)  _{\left\langle
x\right\rangle }=H_{p}\left(  x/1,R_{\left\langle x\right\rangle }\right)  $
for all $p\geq0$, and $\mathfrak{m=}\left\langle x/1\right\rangle $ for the
$n$-tuple $x/1$. Thus $i\left(  x\right)  =i\left(  x/1\right)  $, $x/1$ is a
system of parameters for $R_{\left\langle x\right\rangle }$ and the Samuel
polynomial $\mathfrak{s}\left(  R_{\left\langle x\right\rangle },r\right)  $
has the degree $\dim\left(  R_{\left\langle x\right\rangle }\right)  $
\cite[Theorem 21.4]{AK}. But $\dim\left(  R_{\left\langle x\right\rangle
}\right)  =\dim\left(  R\right)  =d<n$. Using Serre's formula (\ref{Se}), we
conclude that $i\left(  x\right)  =0$.
\end{proof}

Similar result for the inflated $P$-module $R^{\binom{r+n-1}{n}}$ is not
trivial (at least is not straightforward), for its $P$-module structure is not
diagonal. The related result is proved below in Proposition \ref{propDH1}.

\begin{remark}
\label{remNSp}In the case of a nonsingular point $a=0\in Y$ of a variety
$Y\subseteq\mathbb{A}^{n}$ with $d<n$, one can skip Serre's formula. Namely,
$d=\dim_{k}\left(  \mathfrak{m}/\mathfrak{m}^{2}\right)  $ is the minimal
number of generators of the $R_{\left\langle x\right\rangle }$-module
$\mathfrak{m}$ \cite[10.9]{AK}. As above $\mathfrak{m=}\left\langle
x/1\right\rangle $ for the $n$-tuple $x/1$, and there is a new $d$-tuple
$y\subseteq\mathfrak{m}$ generating $\mathfrak{m}$ due to the regularity of
$R_{\left\langle x\right\rangle }$, that is, $y$ is a tuple related to $x/1$.
Note that $0\in Y=\sigma\left(  x,R\right)  $ or $H_{p}\left(  x,R\right)
\neq0$ for some $p$, which in turn implies that $0\in\sigma\left(
x/1,R_{\left\langle x\right\rangle }\right)  $. Using Corollary \ref{corSMP1},
we deduce that $i\left(  x\right)  =i\left(  x/1\right)  =\delta_{nd}i\left(
y\right)  =0$.
\end{remark}

\begin{remark}
\label{remSp}A singular point $a$ of a variety $Y$ is said to be
\textit{integral} if $x_{i}-a_{i}$ is integral over the subalgebra $k\left[
x^{\prime}-a^{\prime}\right]  \subseteq R$ with $i$-th term skipped tuple
$x^{\prime}$. Suppose that $a=0$, and say $x_{n}$ is integral over $R^{\prime
}=k\left[  x^{\prime}\right]  $ with $x^{\prime}=\left(  x_{1},\ldots
,x_{n-1}\right)  $. Then $R^{\prime}\subseteq R$ is integral and $i\left(
x^{\prime}\right)  <\infty$ thanks to Theorem \ref{thNP11}. Moreover,
$0\in\sigma\left(  x^{\prime},R\right)  $ thanks to Corollary \ref{corProj1}.
By Lemma \ref{lemB3}, $i\left(  x\right)  =0$. Again we can avoid (\ref{Se}).
\end{remark}

\begin{example}
\label{exPp}In the case of a point $Y=\left\{  a\right\}  $, we have $R=k$ and
$i=0$ thanks to Corollary \ref{corFLM1}. Actually, $\operatorname{Kos}\left(
x,k\right)  $ is the following complex
\[
0\leftarrow k\leftarrow k^{n}\leftarrow\cdots\leftarrow\wedge^{p}%
k^{n}\leftarrow\cdots\leftarrow k\leftarrow0
\]
with trivial morphisms, $H_{j}\left(  x,k\right)  =\wedge^{j}k^{n}$,
$d_{j}\left(  a\right)  =\dbinom{n}{j}$, $j\geq0$, and $i_{Y}=\sum_{j=0}%
^{n}\left(  -1\right)  ^{j+1}\dbinom{n}{j}=0$.
\end{example}

Now assume that $Y$ is a variety and $a=0\in Y$, that is,
$\mathfrak{p\subseteq}\left\langle X\right\rangle $. For every $r\geq1$
consider again the sequence (\ref{sR}). Notice that $\operatorname{Tor}%
_{i}^{P}\left(  \mathfrak{t}^{r}/\mathfrak{t}^{r+1},R\right)
=\operatorname{Tor}_{i}^{P}\left(  k,R\right)  ^{h\left(  r\right)  }%
=H_{i}\left(  x\right)  ^{h\left(  r\right)  }$ for all $i\geq0$. The functor
$\circ\otimes_{P}R$ generates the following long homology exact sequence
\begin{align*}
0  &  \leftarrow P/\mathfrak{t}^{r}\otimes_{P}R\leftarrow P/\mathfrak{t}%
^{r+1}\otimes_{P}R\leftarrow k^{h\left(  r\right)  }\leftarrow
\operatorname{Tor}_{1}^{P}\left(  P/\mathfrak{t}^{r},R\right)  \leftarrow
\operatorname{Tor}_{1}^{P}\left(  P/\mathfrak{t}^{r+1},R\right)  \leftarrow
H_{1}\left(  x\right)  ^{h\left(  r\right)  }\\
\cdots &  \leftarrow H_{i-1}\left(  x\right)  ^{h\left(  r\right)
}\longleftarrow\operatorname{Tor}_{i}^{P}\left(  P/\mathfrak{t}^{r},R\right)
\leftarrow\operatorname{Tor}_{i}^{P}\left(  P/\mathfrak{t}^{r+1},R\right)
\leftarrow H_{i}\left(  x\right)  ^{h\left(  r\right)  }\leftarrow\cdots
\end{align*}
But $P/\mathfrak{t}^{r}\otimes_{P}R=R/\mathfrak{t}^{r}R=R/\left\langle
x\right\rangle ^{r}$ and $0\leftarrow R/\left\langle x\right\rangle
^{r}\longleftarrow R/\left\langle x\right\rangle ^{r+1}\longleftarrow
\left\langle x\right\rangle ^{r}/\left\langle x\right\rangle ^{r+1}%
\leftarrow0$ is exact. Thus we come with the exact sequence
\[
0\leftarrow\left\langle x\right\rangle ^{r}/\left\langle x\right\rangle
^{r+1}\leftarrow k^{h\left(  r\right)  }\leftarrow\operatorname{Tor}_{1}%
^{P}\left(  P/\mathfrak{t}^{r},R\right)  \leftarrow\operatorname{Tor}_{1}%
^{P}\left(  P/\mathfrak{t}^{r+1},R\right)  \leftarrow H_{1}\left(  x\right)
^{h\left(  r\right)  }\leftarrow\cdots.
\]
Note that $\mathcal{P=}\operatorname{Kos}\left(  t,P\otimes_{k}R\right)  $ is
a free resolution of $R$, and $P/\mathfrak{t}^{r}\otimes_{P}\mathcal{P=}%
\operatorname{Kos}\left(  x^{\left(  r\right)  },P/\mathfrak{t}^{r}\otimes
_{k}R\right)  =\operatorname{Kos}\left(  x^{\left(  r\right)  },R^{\binom
{r+n-1}{n}}\right)  $ with $x^{\left(  r\right)  }=1\otimes x-X\otimes1$ (see
Lemma \ref{lemIM1}). In particular, $\operatorname{Tor}_{i}^{P}\left(
P/\mathfrak{t}^{r},R\right)  =H_{i}\left(  P/\mathfrak{t}^{r}\otimes
_{P}\mathcal{P}\right)  =H_{i}\left(  x^{\left(  r\right)  }\right)  $ for all
$i$ and $r$. Now we apply the localization functor $\circ\otimes
_{R}R_{\left\langle x\right\rangle }$ (or $\circ\otimes_{P}P_{\left\langle
X\right\rangle }$) at $\left\langle x\right\rangle $. By Lemma \ref{lemNP0},
we have%
\[
H_{i}\left(  x^{\left(  r\right)  }/1,R_{\left\langle x\right\rangle }%
^{\binom{r+n-1}{n}}\right)  =H_{i}\left(  x^{\left(  r\right)  }%
,R^{\binom{r+n-1}{n}}\right)  _{\left\langle x\right\rangle }=H_{i}\left(
x^{\left(  r\right)  },R^{\binom{r+n-1}{n}}\right)  =H_{i}\left(  x^{\left(
r\right)  }\right)
\]
for all $i$ and $r$. We obtain the following exact sequence
\begin{align}
0  &  \leftarrow\mathfrak{m}^{r}/\mathfrak{m}^{r+1}\leftarrow k^{h\left(
r\right)  }\leftarrow H_{1}\left(  x^{\left(  r\right)  }\right)  \leftarrow
H_{1}\left(  x^{\left(  r+1\right)  }\right)  \leftarrow H_{1}\left(
x\right)  ^{h\left(  r\right)  }\leftarrow\cdots\label{sq}\\
&  \leftarrow H_{n-1}\left(  x\right)  ^{h\left(  r\right)  }\leftarrow
H_{n}\left(  x^{\left(  r\right)  }\right)  \leftarrow H_{n}\left(  x^{\left(
r+1\right)  }\right)  \leftarrow H_{n}\left(  x\right)  ^{h\left(  r\right)
}=0,\nonumber
\end{align}
where $\mathfrak{m=}\left\langle x\right\rangle R_{\left\langle x\right\rangle
}$ is the maximal ideal of the local ring $R_{\left\langle x\right\rangle }$.

\begin{lemma}
\label{lemDH1}If $H_{j+1}\left(  x\right)  =0$ for some $j$, then
$H_{i}\left(  x^{\left(  r\right)  }\right)  =0$ for all $i>j$ and $r\geq1$,
and the sequence
\begin{align*}
0  &  \leftarrow\mathfrak{m}^{r}/\mathfrak{m}^{r+1}\leftarrow k^{h\left(
r\right)  }\leftarrow\cdots\leftarrow H_{i-1}\left(  x\right)  ^{h\left(
r\right)  }\leftarrow H_{i}\left(  x^{\left(  r\right)  }\right)  \leftarrow
H_{i}\left(  x^{\left(  r+1\right)  }\right)  \leftarrow H_{i}\left(
x\right)  ^{h\left(  r\right)  }\leftarrow\cdots\\
&  \leftarrow H_{j-1}\left(  x^{\left(  r\right)  }\right)  \leftarrow
H_{j}\left(  x^{\left(  r\right)  }\right)  \leftarrow H_{j}\left(  x\right)
^{h\left(  r\right)  }\leftarrow0
\end{align*}
remains exact. In particular, $H_{n}\left(  x^{\left(  r\right)  }\right)  =0$
for all $r$, and the sequence
\[
0\leftarrow\mathfrak{m}^{r}/\mathfrak{m}^{r+1}\leftarrow k^{h\left(  r\right)
}\leftarrow\cdots\leftarrow H_{n-1}\left(  x^{\left(  r\right)  }\right)
\leftarrow H_{n-1}\left(  x^{\left(  r+1\right)  }\right)  \leftarrow
H_{n-1}\left(  x\right)  ^{h\left(  r\right)  }\leftarrow0
\]
is exact.
\end{lemma}

\begin{proof}
Suppose that $H_{j+1}\left(  x\right)  =0$ for some $j$. Prove that
$H_{i}\left(  x^{\left(  r\right)  }\right)  =0$ for all $i>j$, $r\geq1$ by
induction on $r$. By Lemma \ref{lemNP0}, $H_{i}\left(  x\right)  =0$ for all
$i>j$. If $r=1$, then we have $\operatorname{Tor}_{i}^{P}\left(
P/\mathfrak{t},R\right)  =\operatorname{Tor}_{i}^{P}\left(  k,R\right)
=H_{i}\left(  x\right)  $, $i\geq0$ and we come with the exact pieces
\[
\cdots\leftarrow H_{i}\left(  x\right)  \leftarrow H_{i}\left(  x^{\left(
2\right)  }\right)  \leftarrow H_{i}\left(  x\right)  ^{h\left(  r\right)
}\leftarrow\cdots.
\]
of the sequence (\ref{sq}) for $r=1$. If $i>j$ then $H_{i}\left(  x^{\left(
2\right)  }\right)  =0$. By induction hypothesis, $H_{i}\left(  x^{\left(
r\right)  }\right)  =0$ for all $i>j$. The exact part%
\[
\cdots\leftarrow H_{i-1}\left(  x\right)  ^{h\left(  r\right)  }\leftarrow
H_{i}\left(  x^{\left(  r\right)  }\right)  \leftarrow H_{i}\left(  x^{\left(
r+1\right)  }\right)  \leftarrow H_{i}\left(  x\right)  ^{h\left(  r\right)
}\leftarrow\cdots
\]
of the sequence (\ref{sq}) for $i>j$, implies that $H_{i}\left(  x^{\left(
j+1\right)  }\right)  =0$, and the long homology sequence terminates at
$H_{j}\left(  x\right)  ^{h\left(  r\right)  }$. Finally, since $H_{n}\left(
x\right)  =0$ we conclude that $H_{n}\left(  x^{\left(  r\right)  }\right)
=0$ for all $r\geq1$.
\end{proof}

Based on Lemma \ref{lemIM1}, we have $i\left(  x^{\left(  r\right)  }\right)
<\infty$ for all $r$, and we define the following function
\[
p:\mathbb{Z}\longrightarrow\mathbb{Z}\text{,\quad}p\left(  r\right)
=\sum_{i=1}^{n}\left(  -1\right)  ^{i+1}\dim_{k}\left(  \operatorname{Tor}%
_{i}^{P}\left(  P/\mathfrak{t}^{r},R\right)  \right)
\]
called as of $\operatorname{Tor}$-polynomial. Note that $-\dim_{k}\left(
H_{0}\left(  x^{\left(  r\right)  }\right)  \right)  +p\left(  r\right)
=i\left(  x^{\left(  r\right)  }\right)  $ for all $r$. Recall that
$R_{\left\langle x\right\rangle }$ is a local Noetherian algebra, and
$d=\dim\left(  Y\right)  =\dim\left(  R_{\left\langle x\right\rangle }\right)
=\deg\left(  h_{R}\right)  +1=\deg\mathfrak{s}\left(  R_{\left\langle
x\right\rangle },\circ\right)  \leq n$, where $h_{R}\left(  r\right)
=\dim_{k}\left(  \mathfrak{m}^{r}/\mathfrak{m}^{r+1}\right)  $ is the Hilbert
polynomial of $R_{\left\langle x\right\rangle }$ associated with $\left\{
\mathfrak{m}^{r}\right\}  $.

\begin{proposition}
\label{propDH1}The equality $p\left(  r\right)  =\left(  1-\delta_{nd}\right)
\mathfrak{s}\left(  R_{\left\langle x\right\rangle },r\right)  $ holds for all
$r$, that is, the function $p$ is a numerical polynomial of degree at most
$d$. In particular, $i_{Y}^{\left(  r\right)  }=\dbinom{r+n-1}{n}i_{Y}$ for
all $r$.
\end{proposition}

\begin{proof}
Using the exact sequence from Lemma \ref{lemDH1} and its Euler
characteristics, we deduce that
\begin{align*}
0  &  =-h_{R}\left(  r\right)  +h\left(  r\right)  +\sum_{i=1}^{n-1}\left(
-1\right)  ^{i}\dim_{k}\left(  H_{i}\left(  x^{\left(  r\right)  }\right)
\right)  +\sum_{i=1}^{n-1}\left(  -1\right)  ^{i+1}\dim_{k}\left(
H_{i}\left(  x^{\left(  r+1\right)  }\right)  \right) \\
&  +h\left(  r\right)  \sum_{i=1}^{n-1}\left(  -1\right)  ^{i}\dim_{k}\left(
H_{i}\left(  x\right)  \right)  .
\end{align*}
Note that $H_{0}\left(  x\right)  =R/\left\langle x\right\rangle =k$,
$H_{n}\left(  x^{\left(  r\right)  }\right)  =0$, and $\sum_{i=0}^{n}\left(
-1\right)  ^{i}\dim_{k}\left(  H_{i}\left(  x\right)  \right)  =1-p\left(
1\right)  =-i\left(  x\right)  $. If $d<n$ then $i\left(  x\right)  =0$ thanks
to Lemma \ref{lemAp}. Thus $p:\mathbb{Z}\rightarrow\mathbb{Z}$ is a function
with the property
\[
p\left(  z+1\right)  -p\left(  z\right)  =h_{R}\left(  z\right)  ,
\]
that is, $p\left(  z+1\right)  -p\left(  z\right)  $ is a numerical polynomial
for large $z$. That means $p\left(  z\right)  \in\mathbb{Q}\left[  z\right]  $
is a numerical polynomial (see \cite[1.7.3]{Harts}) and $p\left(  r\right)
=\sum_{j=0}^{r-1}h_{R}\left(  j\right)  =\mathfrak{s}\left(  R_{\left\langle
x\right\rangle },r\right)  $ for all $r\geq1$. If $Y=\mathbb{A}^{n}$ then
$h_{R}\left(  r\right)  =h\left(  r\right)  $, $H_{i}\left(  x\right)  =0$,
$i\geq1$, and $i\left(  x\right)  =-1$ (see Lemma \ref{lemAp}). It follows
that $p\left(  r+1\right)  =p\left(  r\right)  $ for all $r\geq1$. But
$p\left(  1\right)  =i\left(  x\right)  +1=0$, therefore $p=0$.

Finally, taking into account that%
\[
H_{0}\left(  x^{\left(  r\right)  }\right)  =H_{0}\left(  x^{\left(  r\right)
}\right)  _{\left\langle x\right\rangle }=\operatorname{Tor}_{0}^{P}\left(
P/\mathfrak{t}^{r},R\right)  _{\left\langle x\right\rangle }=\left(
P/\mathfrak{t}^{r}\otimes_{P}R\right)  _{\left\langle x\right\rangle }=\left(
R/\left\langle x\right\rangle ^{r}\right)  _{\left\langle x\right\rangle
}=R_{\left\langle x\right\rangle }/\mathfrak{m}^{r},
\]
we deduce that $\dim_{k}\left(  H_{0}\left(  x^{\left(  r\right)  }\right)
\right)  =\dim_{k}\left(  R_{\left\langle x\right\rangle }/\mathfrak{m}%
^{r}\right)  =\mathfrak{s}\left(  R_{\left\langle x\right\rangle },r\right)
$. It follows that
\[
i\left(  x^{\left(  r\right)  }\right)  =-\dim_{k}\left(  H_{0}\left(
x^{\left(  r\right)  }\right)  \right)  +p\left(  r\right)  =-\mathfrak{s}%
\left(  R_{\left\langle x\right\rangle },r\right)  +\left(  1-\delta
_{nd}\right)  \mathfrak{s}\left(  R_{\left\langle x\right\rangle },r\right)
=-\delta_{nd}\mathfrak{s}\left(  R_{\left\langle x\right\rangle },r\right)  .
\]
Thus $i_{Y}^{\left(  r\right)  }=0$ whenever $d<n$, and $i_{Y}^{\left(
r\right)  }=-\mathfrak{s}\left(  P_{\left\langle X\right\rangle },r\right)
=-\mathfrak{s}\left(  P,r\right)  =-\dbinom{r+n-1}{n}$ for $Y=\mathbb{A}^{n}$.
\end{proof}

\begin{example}
If $d<n$ and $H_{2}\left(  x\right)  =0$, then $H_{i}\left(  x^{\left(
r\right)  }\right)  =0$ for all $i\geq2$, $r\geq1$ thanks to Lemma
\ref{lemDH1}. Then $p\left(  r\right)  =\dim_{k}\left(  H_{1}\left(
x^{\left(  r\right)  }\right)  \right)  $, $i\left(  x\right)  =-1+\dim
_{k}\left(  H_{1}\left(  x\right)  \right)  =0$ and $\dim_{k}\left(
H_{1}\left(  x^{\left(  r\right)  }\right)  \right)  =\mathfrak{s}\left(
R_{\left\langle x\right\rangle },r\right)  $, that is, $r\mapsto\dim
_{k}\left(  H_{1}\left(  x^{\left(  r\right)  }\right)  \right)  $ is a
polynomial of degree $d$. If $Y\subseteq\mathbb{A}^{2}$ is an irreducible
curve and $a\in Y$, then $d_{0}\left(  a\right)  =1$ and $d_{2}\left(
a\right)  =0$. By Lemma \ref{lemAp}, $i_{Y}\left(  a\right)  =0$, therefore
$d_{1}\left(  a\right)  =1$, and $r\mapsto\dim_{k}\left(  H_{1}\left(  \left(
x-a\right)  ^{\left(  r\right)  }\right)  \right)  $ is a linear polynomial.
\end{example}

\begin{remark}
The presence of nonzero $H_{1}\left(  x\right)  $ is well know. If
$H_{1}\left(  x\right)  =0$ then $R\overset{\partial_{0}}{\longleftarrow
}R\otimes k^{n}\overset{\partial_{1}}{\longleftarrow}R\otimes\wedge^{2}k^{n}$
is exact and so is its localization $R_{\left\langle x\right\rangle
}\overset{\partial_{0}}{\longleftarrow}R_{\left\langle x\right\rangle }\otimes
k^{n}\overset{\partial_{1}}{\longleftarrow}R_{\left\langle x\right\rangle
}\otimes\wedge^{2}k^{n}$, that is, $x/1\subseteq\operatorname{rad}\left(
R_{\left\langle x\right\rangle }\right)  $ with $H_{1}\left(  x/1\right)  =0$.
Using \cite[9.7, Theorem 1]{BurHA}, we conclude that $x/1$ is a regular
sequence in $R_{\left\langle x\right\rangle }$. In particular,
$\operatorname{depth}\left(  R_{\left\langle x\right\rangle }\right)  \geq n$.
But $\operatorname{depth}\left(  R_{\left\langle x\right\rangle }\right)
\leq\dim\left(  R_{\left\langle x\right\rangle }\right)  =d<n$, a contradiction.
\end{remark}

The evaluation map $a:P\rightarrow k$, $f\mapsto f\left(  a\right)  $ can be
lifted to a ring homomorphism $a:R\rightarrow k$, which is the quotient
mapping $R\rightarrow R/\left\langle x\right\rangle $. In particular, there
are $k$-linear maps $\wedge^{r}a:R\otimes_{k}\wedge^{r}k^{n}\rightarrow
\wedge^{r}k^{n}$, $r\geq1$, and $\left(  \wedge^{r}a\right)  \partial_{r-1}=0$
for all $r$, which in turn define $k$-linear maps $a^{\left(  r\right)
}:H_{r}\left(  x\right)  \rightarrow\wedge^{r}k^{n}$ on the homology groups.
Note that $H_{0}\left(  x\right)  =R/\operatorname{im}\left(  \partial
_{0}\right)  =k$, $a^{\left(  0\right)  }=1$, and $a^{\left(  1\right)
}:H_{1}\left(  x\right)  \rightarrow k^{n}$ is a $k$-linear map.

\subsection{Upper-triangular matrices}

Now let us analyze the first homology group $H_{1}\left(  x\right)  $. We
assume that $\operatorname{char}\left(  k\right)  \neq2$. Every $g=\sum
_{i<j}g_{ij}e_{i}\wedge e_{j}\in R\otimes_{k}\wedge^{2}k^{n}$ can be
represented by means of an upper triangular matrix $g=\left[  g_{ij}\right]
_{i,j}\in M_{n}\left(  R\right)  $ whose $i$-th column $C_{i}$ and $i$-th row
$R_{i}$ (they are $n$-tuples in $R$) have the shapes
\[%
\begin{array}
[c]{cccccc}
& g_{1i} &  &  &  & \\
& g_{2i} &  &  &  & \\
& \vdots &  &  &  & \\
& g_{i-1i} &  &  &  & \\
\cdots & 0 & g_{ii+1} & \cdots & g_{in-1} & g_{in}\\
& \vdots &  &  &  &
\end{array}
\]
for all $i>1$, and $C_{1}$, $R_{n}$ consist of zeros. For $n$-tuples $\alpha$
and $\beta$ from $R^{n}$ we write $\left\langle \alpha,\beta\right\rangle $
instead of the sum $\sum_{i=1}^{n}\alpha_{i}\beta_{i}$ in $R$. Note that
$\partial_{1}:R\otimes_{k}\wedge^{2}k^{n}\rightarrow R\otimes_{k}k^{n}$ is an
$R$-linear map acting by the rule%
\begin{equation}
\partial_{1}\left(  g\right)  =\sum_{i<j}x_{i}g_{ij}e_{j}-x_{j}g_{ij}%
e_{j}=\sum_{i=1}^{n}\left(  \sum_{s<i}x_{s}g_{si}-\sum_{i<s}x_{s}%
g_{is}\right)  e_{i}=\sum_{i=1}^{n}\left(  \left\langle C_{i},x\right\rangle
-\left\langle R_{i},x\right\rangle \right)  e_{i}. \label{D1F}%
\end{equation}
For example, if
\[
g=\left[
\begin{array}
[c]{ccccc}%
0 & \cdots & g_{j} & \cdots & g_{n}\\
0 & \cdots & 0 & \cdots & 0\\
&  & \vdots &  & \\
&  &  &  & \\
0 &  & 0 &  & 0
\end{array}
\right]  \text{ with }\sum_{k=j}^{n}x_{k}g_{k}=0
\]
for some $g_{j},\ldots,g_{n}$ then $C_{i}=0$, $1\leq i\leq j-1$, $R_{i}=0$,
$i>1$ and
\begin{align*}
\partial_{1}\left(  g\right)   &  =-\left\langle R_{1},x\right\rangle
e_{1}+\left\langle C_{j},x\right\rangle e_{j}+\cdots+\left\langle
C_{n},x\right\rangle e_{n}=-\left(  \sum_{k=j}^{n}x_{k}g_{k}\right)
e_{1}+x_{1}g_{j}e_{j}+\cdots+x_{1}g_{n}e_{n}\\
&  =x_{1}\sum_{k=j}^{n}g_{j}e_{j}.
\end{align*}
The following assertion is a special case of Lemma \ref{lemB4}. But for our
purposes we provide an independent proof in the case of $H_{1}\left(
x\right)  $ exploiting upper triangular matrices.

\begin{lemma}
\label{lemKZ2}The equality $\left\langle x\right\rangle H_{1}\left(  x\right)
=\left\{  0\right\}  $ holds.
\end{lemma}

\begin{proof}
Take $\omega=\sum_{j=1}^{n}g_{j}e_{j}\in\ker\left(  \partial_{0}\right)  $ or
$\sum_{j=1}^{n}x_{j}g_{j}=0$ in $R$. Fix $i$ and prove that $x_{i}\omega
\in\operatorname{im}\left(  \partial_{1}\right)  $. First note that
$x_{i}g_{i}=-\sum_{j\neq i}x_{j}g_{j}$ and consider the matrix
\[
g=\left[
\begin{array}
[c]{cccccc}%
0 & \cdots & -g_{1} & 0 & \cdots & 0\\
&  & \vdots & \vdots &  & \vdots\\
&  & -g_{i-1} & 0 &  & 0\\
0 & \cdots & 0 & g_{i+1} & \cdots & g_{n}\\
&  & \vdots & \vdots &  & \vdots\\
&  & 0 & 0 &  & 0
\end{array}
\right]
\]
with nontrivial $i$-th column and $i$-th row. Since
\[
\left\langle C_{i},x\right\rangle -\left\langle R_{i},x\right\rangle
=-\sum_{j=1}^{i-1}x_{j}g_{j}-\sum_{j=i+1}^{n}x_{j}g_{j}=-\sum_{j\neq i}%
x_{j}g_{j}=x_{i}g_{i},
\]
we deduce using (\ref{D1F}) that
\begin{align*}
\partial_{1}\left(  g\right)   &  =-x_{i}g_{1}e_{1}-\cdots-x_{i}g_{i-1}%
e_{i-1}+\left(  \left\langle C_{i},x\right\rangle -\left\langle R_{i}%
,x\right\rangle \right)  e_{i}+x_{i}g_{i+1}e_{i+1}+\cdots+x_{i}g_{n}e_{n}\\
&  =x_{i}\left(  \sum_{j<i}g_{j}e_{j}+\sum_{j\geq i}g_{j}e_{j}\right)
=x_{i}\omega,
\end{align*}
that is, $x_{i}H_{1}\left(  x\right)  =\left\{  0\right\}  $.
\end{proof}

\subsection{Minimal generators}

Now let $F=\left\{  f_{1},\ldots,f_{p}\right\}  $ be a set of generators of
the ideal $\mathfrak{p}$. Since $\mathfrak{p\subseteq}\left\langle
X\right\rangle $ in $P$, it follows that every $f\in\mathfrak{p}$ has a
standard representation $f=X_{1}^{m_{1}}h_{1}+\cdots+X_{n}^{m_{n}}h_{n}$ with
$h_{j}\left(  a\right)  \neq0$. In particular, $f_{j}=\sum_{i}X_{i}^{m_{ij}%
}h_{ij}$ with $h_{ij}\left(  a\right)  \neq0$ for all $i,j$. For every $i$,
$1\leq i\leq n$, we put $m_{i}=\min\left\{  m_{ij}:1\leq j\leq p\right\}  $. A
generator $f\in F$ is said to be \textit{minimal} if $f=f_{j}$ and
$m_{i}=m_{ij}$ for some $i$. Every $f=\sum_{i}X_{i}^{m_{i}}h_{i}%
\in\mathfrak{p}$ associates in turn $\omega=\sum_{i=1}^{n}x_{i}^{m_{i}-1}%
h_{i}e_{i}\in R\otimes_{k}k^{n}$ such that $\partial_{0}\left(  \omega\right)
=\sum_{i=1}^{n}x_{i}^{m_{i}}h_{i}=0$ in $R=P/\mathfrak{p}$. In particular,
$\omega^{\sim}\in H_{1}\left(  x\right)  $.

\begin{lemma}
\label{lemMinf}If $f\in F$ is a minimal generator then $\omega^{\sim}\neq0$ in
$H_{1}\left(  x\right)  $, and $a^{\left(  1\right)  }\left(  \omega^{\sim
}\right)  \neq0$ whenever $m_{i}=1$.
\end{lemma}

\begin{proof}
Suppose $\omega^{\sim}=0$, which means that $\omega=\partial_{1}\left(
g\right)  $ for some upper triangular matrix $g\in M_{n}\left(  R\right)  $.
It follows that $\left\langle C_{i},x\right\rangle -\left\langle
R_{i},x\right\rangle =x_{i}^{m_{i}-1}h_{i}$, $1\leq i\leq n$ in $R$ or
\[
\sum_{s<i}X_{s}g_{si}-X_{i}^{m_{i}-1}h_{i}-\sum_{i<s}X_{s}g_{is}=\sum
_{q=1}^{p}f_{q}l_{q}=\sum_{q=1}^{p}\sum_{k=1}^{n}X_{k}^{m_{kq}}h_{kq}l_{q}%
\]
in $P$ for some $l_{q}\in P$. But $f=f_{j}$ for some $j$, therefore
\[
\sum_{s<i}X_{s}g_{si}-X_{i}^{m_{i}-1}h_{i}-\sum_{i<s}X_{s}g_{is}=\sum
_{k=1}^{n}X_{k}^{m_{k}}h_{k}l_{j}+\sum_{q\neq j}\sum_{k=1}^{n}X_{k}^{m_{kq}%
}h_{kq}l_{q}.
\]
By passing to the quotient algebra $P/\left\langle X_{1},\ldots,\widehat{X_{i}%
},\ldots,X_{n}\right\rangle =k\left[  X_{i}\right]  $, we obtain that
$-X_{i}^{m_{i}-1}h_{i}^{\sim}=X_{i}^{m_{i}}h_{i}^{\sim}l_{j}^{\sim}%
+\sum_{q\neq j}X_{i}^{m_{iq}}h_{iq}^{\sim}l_{q}^{\sim}$. But $m_{i}%
=\min\left\{  m_{iq}:1\leq q\leq p\right\}  $, therefore $m_{iq}-m_{i}\geq0$
and
\[
h_{i}^{\sim}=-X_{i}h_{i}^{\sim}l_{j}^{\sim}-\sum_{q\neq j}X_{i}^{m_{iq}%
-m_{i}+1}h_{iq}^{\sim}l_{q}^{\sim}\in\left\langle X_{i}\right\rangle .
\]
It follows that $h_{i}\left(  a\right)  =0$, a contradiction. Finally,
$a\left(  x_{i}^{m_{i}-1}h_{i}\right)  =h_{i}\left(  a\right)  \neq0$ whenever
$m_{i}=1$. Therefore $a^{\left(  1\right)  }\left(  \omega^{\sim}\right)
=\left(  a\left(  x_{k}^{m_{k}-1}h_{k}\right)  \right)  _{k}\neq0$ in $k^{n}$.
\end{proof}

\begin{remark}
In the case of an irreducible curve $Y\subseteq\mathbb{A}^{2}$ given by
$f=X_{1}^{n}h_{1}+X_{2}^{m}h_{2}\in k\left[  X_{1},X_{2}\right]  $,
$h_{j}\notin\left\langle X_{1},X_{2}\right\rangle $, $n,m\geq1$, we have
$H_{1}\left(  x\right)  =k\omega^{\sim}$ with $\omega=\left(  x_{2}^{m-1}%
h_{2},x_{1}^{n-1}h_{1}\right)  \in\ker\left(  \partial_{0}\right)  $ and
$a^{\left(  1\right)  }\left(  \omega^{\sim}\right)  =\left(  a\left(
x_{2}^{m-1}h_{2}\right)  ,a\left(  x_{1}^{n-1}h_{1}\right)  \right)  $, which
is not trivial whenever $m=1$ or $n=1$. The latter is equivalent to the
presence of a linear part of $f$, which means that $f$ is nonsingular at $a$.
Thus $a^{\left(  1\right)  }:H_{1}\left(  x-a\right)  \rightarrow k^{2}$ is a
nonzero linear map iff $f$ is nonsingular at $a$.
\end{remark}

For every $i$ we put $S_{i}=\left\{  j:m_{ij}=m_{i}\right\}  $ and define the
vector $v_{j}=\left(  v_{1j},v_{2j},\ldots,v_{nj}\right)  \in k^{n}$ with
\[
v_{ij}=\left\{
\begin{array}
[c]{ccc}%
h_{ij}\left(  a\right)  & \text{if} & j\in S_{i},\\
0 & \text{if} & j\notin S_{i}.
\end{array}
\right.
\]
Thus $v_{j}=0$ whenever $j\notin\cup_{i}S_{i}$. If $j\in S_{i}$ for some $i$,
then $m_{ij}=m_{i}$, $f_{j}$ turns out to be a minimal generator and
$v_{ij}=h_{ij}\left(  a\right)  \neq0$. Hence $v_{j}\neq0$ in $k^{n}$ iff
$f_{j}$ is a minimal generator. Suppose $f_{1},\ldots,f_{t}$ are minimal
generators in $F$. Then we have the related nonzero vectors $\left\{
v_{1},\ldots,v_{t}\right\}  $ from $k^{n}$, and related nonzero vectors
$\left\{  \omega_{1}^{\sim},\ldots,\omega_{t}^{\sim}\right\}  $ (see Lemma
\ref{lemMinf}) from $H_{1}\left(  x\right)  $.

\begin{lemma}
\label{lemMinfi}If $\left\{  v_{1},\ldots,v_{t}\right\}  $ is a linearly
independent set of vectors in $k^{n}$ then so is the set $\left\{  \omega
_{1}^{\sim},\ldots,\omega_{t}^{\sim}\right\}  $ and $\dim_{k}\left(
H_{1}\left(  x\right)  \right)  \geq t$.
\end{lemma}

\begin{proof}
Suppose that $\sum_{j=1}^{t}\lambda_{j}\omega_{j}^{\sim}=0$, that is,
$\sum_{i=1}^{n}\left(  \sum_{j=1}^{t}\lambda_{j}x_{i}^{m_{ij}-1}h_{ij}\right)
e_{i}=\partial_{1}\left(  g\right)  $ for some $g$. As in the proof of Lemma
\ref{lemMinf}, we derive that
\[
\sum_{s<i}X_{s}g_{si}-\sum_{j=1}^{t}X_{i}^{m_{ij}-1}\lambda_{j}h_{ij}%
-\sum_{i<s}X_{s}g_{is}=\sum_{s=1}^{p}X_{i}^{m_{is}}h_{is}l_{s}+\sum_{s=1}%
^{p}\sum_{k\neq i}X_{k}^{m_{ks}}h_{ks}l_{s}%
\]
for some $l_{s}\in P$. Again by passing to the quotient algebra $k\left[
X_{i}\right]  $, we obtain that $\sum_{j=1}^{t}X_{i}^{m_{ij}-1}\lambda
_{j}h_{ij}^{\sim}=-\sum_{s=1}^{p}X_{i}^{m_{is}}h_{is}^{\sim}l_{s}^{\sim}$.
But
\[
\sum_{j=1}^{t}X_{i}^{m_{ij}-1}\lambda_{j}h_{ij}^{\sim}=\sum_{j\in S_{i}}%
+\sum_{j\notin S_{i}}=X_{i}^{m_{i}-1}\sum_{j\in S_{i}}\lambda_{j}h_{ij}^{\sim
}+\sum_{j\notin S_{i}}X_{i}^{m_{ij}-1}\lambda_{j}h_{ij}^{\sim},
\]
therefore
\[
X_{i}^{m_{i}-1}\sum_{j\in S_{i}}\lambda_{j}h_{ij}^{\sim}=-\sum_{j\notin S_{i}%
}X_{i}^{m_{ij}-1}\lambda_{j}h_{ij}^{\sim}-\sum_{s=1}^{p}X_{i}^{m_{is}}%
h_{is}^{\sim}l_{s}^{\sim}.
\]
Thus $\sum_{j\in S_{i}}\lambda_{j}h_{ij}^{\sim}\in\left\langle X_{i}%
\right\rangle $, which in turn implies that $\sum_{j\in S_{i}}\lambda
_{j}h_{ij}\left(  a\right)  =0$ for every $i$. In particular,
\[
\sum_{j=1}^{t}\lambda_{j}v_{ij}=\sum_{j\in S_{i}}\lambda_{j}h_{ij}\left(
a\right)  =0
\]
for every $i$, which means that $\sum_{j=1}^{t}\lambda_{j}v_{j}=0$ in $k^{n}$.
Using the fact of independence, we conclude that $\lambda_{j}=0$ for all $j$.
Hence $\left\{  \omega_{1}^{\sim},\ldots,\omega_{t}^{\sim}\right\}  $ is a
linearly independent set of vectors and the $R$-module structure of
$H_{1}\left(  x\right)  $ is reduced to its $k$-vector space one (see Lemma
\ref{lemKZ2}). Whence $\dim_{k}\left(  H_{1}\left(  x\right)  \right)  \geq t$.
\end{proof}

For every $j$ we have the well defined vector $f_{j}^{\prime}=\left(
\dfrac{\partial f_{j}}{\partial X_{1}}\left(  a\right)  ,\ldots,\dfrac
{\partial f_{j}}{\partial X_{n}}\left(  a\right)  \right)  \in k^{n}$. Notice
that $v_{j}\neq0$ whenever $f_{j}^{\prime}\neq0$, and in this case
$v_{j}=f_{j}^{\prime}$ and $a^{\left(  1\right)  }\left(  \omega_{j}^{\sim
}\right)  =f_{j}^{\prime}$.

\begin{corollary}
If the Jacobian matrix $J_{a}=\left[  \dfrac{\partial f_{j}}{\partial X_{i}%
}\left(  a\right)  \right]  _{j,i}$ has the rank $t$ then $\dim_{k}\left(
H_{1}\left(  x-a\right)  \right)  \geq t$. If $Y$ is nonsingular at point $a$
then $d_{1}\left(  a\right)  \geq n-d$.
\end{corollary}

\begin{proof}
There are linearly independent columns $f_{j_{1}}^{\prime},\ldots,f_{j_{t}%
}^{\prime}$ of the Jacobian matrix. In particular, so are vectors $v_{j_{1}%
},\ldots,v_{j_{t}}$. By Lemma \ref{lemMinfi}, $\left\{  \omega_{j_{1}}^{\sim
},\ldots,\omega_{j_{t}}^{\sim}\right\}  $ is an independent set of vectors
from $H_{1}\left(  x\right)  $, therefore $\dim_{k}\left(  H_{1}\left(
x\right)  \right)  \geq t$. If $Y$ is nonsingular at point $a$ then the rank
of the matrix $J_{a}$ is $n-d$, where $d=\dim\left(  Y\right)  $. It follows
that $\dim_{k}\left(  H_{1}\left(  x\right)  \right)  \geq n-d$.
\end{proof}

\subsection{Examples}

Let $Y\subseteq\mathbb{A}^{3}$ be a variety given by a prime
$\mathfrak{p\subseteq}P=k\left[  X_{1},X_{2},X_{2}\right]  $. For the actions
of $X_{j}$ on $R=P/\mathfrak{p}$ we use the notations $x,y$ and $z$,
respectively. If $a=0\in Y$ then the Koszul complex $\operatorname{Kos}\left(
\left(  x,y,z\right)  ,R\right)  $ looks like that
\[
0\leftarrow R\overset{\partial_{0}}{\longleftarrow}R^{3}\overset{\partial
_{1}}{\longleftarrow}R^{3}\overset{\partial_{2}}{\longleftarrow}R\leftarrow0
\]
with the operators
\[
\partial_{0}=\left[
\begin{array}
[c]{ccc}%
x & y & z
\end{array}
\right]  ,\partial_{1}=\left[
\begin{array}
[c]{ccc}%
-y & -z & 0\\
x & 0 & -z\\
0 & x & y
\end{array}
\right]  ,\partial_{2}=\left[
\begin{array}
[c]{c}%
z\\
-y\\
x
\end{array}
\right]
\]
whose homology groups are denoted by $H_{j}$ and $d_{j}=\dim_{k}\left(
H_{j}\right)  $, $0\leq j\leq3$. Recall that $d_{0}=1$, $d_{3}=0$.

\begin{proposition}
The equalities $\partial_{1}\left(  \partial_{1}^{2}+2xz-y^{2}\right)  =0$ and
$\operatorname{im}\left(  \partial_{1}^{2}+2xz-y^{2}\right)  \oplus
\operatorname{im}\left(  \partial_{2}|k\cdot1\right)  =\operatorname{im}%
\left(  \partial_{2}\right)  $ hold.
\end{proposition}

\begin{proof}
One can easily verify that $Q\left(  t\right)  =t\left(  t^{2}+2xz-y^{2}%
\right)  \in R\left[  t\right]  $ is the characteristic polynomial of
$\partial_{1}$. By Cayley-Hamilton Theorem, we obtain that $Q\left(
\partial_{1}\right)  =0$. In particular, $\operatorname{im}\left(
\partial_{1}^{2}+2xz-y^{2}\right)  \subseteq\ker\left(  \partial_{1}\right)
$. But
\[
\partial_{1}^{2}+2xz-y^{2}=\left[
\begin{array}
[c]{ccc}%
xz & yz & z^{2}\\
-xy & -y^{2} & -yz\\
x^{2} & xy & xz
\end{array}
\right]  ,
\]
which in turn implies that $\left(  \partial_{1}^{2}+2xz-y^{2}\right)
g=\partial_{2}\left(  xg_{1}+yg_{2}+zg_{3}\right)  $ for every $g=\left(
g_{1},g_{2},g_{3}\right)  \in R^{3}$. Thus $\operatorname{im}\left(
\partial_{1}^{2}+2xz-y^{2}\right)  \subseteq\operatorname{im}\left(
\partial_{2}\right)  $. Finally, for every $h\in R$ we have $h=h\left(
0\right)  +xh_{1}+yh_{2}+zh_{3}$ and
\[
\partial_{2}\left(  h\right)  =h\left(  0\right)  \partial_{2}\left(
1\right)  +\partial_{2}\left(  xh_{1}+yh_{2}+zh_{3}\right)  =\left(
\partial_{1}^{2}+2xz-y^{2}\right)  g+h\left(  0\right)  \partial_{2}\left(
1\right)  ,
\]
where $g=\left(  h_{1},h_{2},h_{3}\right)  \in R^{3}$. It remains to use the
fact $H_{2}=0$ and the $k$-linear projection $\partial_{2}\left(  h\right)
\mapsto h\left(  0\right)  \partial_{2}\left(  1\right)  $.
\end{proof}

Thus $g\in\ker\left(  \partial_{1}\right)  -\operatorname{im}\left(
\partial_{2}\right)  $ iff $xg_{1}=zg_{3}$, $xg_{2}=-yg_{3}$, $yg_{1}=-zg_{3}%
$, and
\[
g\neq\lambda\left[
\begin{array}
[c]{c}%
z\\
-y\\
x
\end{array}
\right]  +\left(  2xz-y^{2}\right)  h
\]
for all $\lambda\in k$ and $h\in R^{3}$, that is, the structure of $H_{2}$ is
really complicated. But Lemma \ref{lemAp} turns out to be useful in this
manner. Namely, consider the following hypersurface $Y=\left\{
xy+yz+zx=0\right\}  $, whose singularity at $a=0$ is not integral (see Remark
\ref{remSp}). Then
\[
H_{1}=k\omega^{\sim}\text{ with }\omega=\left[
\begin{array}
[c]{c}%
y+z\\
z\\
0
\end{array}
\right]  .
\]
Indeed, $\partial_{0}\left(  \omega\right)  =x\left(  y+z\right)  +yz=0$ means
that $\omega\in\ker\left(  \partial_{0}\right)  $. Take $g\in\ker\left(
\partial_{0}\right)  $, that is, $xg_{1}+yg_{2}+zg_{3}=0$ in $R$. Then
$X_{1}g_{1}+X_{2}g_{2}+X_{3}g_{3}=\left(  X_{1}\left(  X_{2}+X_{3}\right)
+X_{2}X_{3}\right)  h\in\mathfrak{p}$ or $X_{3}g_{3}=X_{1}\left(  \left(
X_{2}+X_{3}\right)  h-g_{1}\right)  +X_{2}\left(  X_{3}h-g_{2}\right)
\in\left\langle X_{1},X_{2}\right\rangle $ in $P$. Then $g_{3}=X_{1}%
q_{1}+X_{2}q_{2}\in\left\langle X_{1},X_{2}\right\rangle $ and
\begin{align*}
\left(  X_{2}+X_{3}\right)  h-g_{1} &  =X_{3}q_{1}-X_{2}l,\\
X_{3}h-g_{2} &  =X_{3}q_{2}+X_{1}l
\end{align*}
for some $q_{j}$, $l\in P$. It follows that%
\begin{align*}
g &  =\left[
\begin{array}
[c]{c}%
g_{1}\\
g_{2}\\
g_{3}%
\end{array}
\right]  =\left[
\begin{array}
[c]{c}%
\left(  y+z\right)  h+yl-zq_{1}\\
zh-xl-zq_{2}\\
xq_{1}+yq_{2}%
\end{array}
\right]  =h\left[
\begin{array}
[c]{c}%
y+z\\
z\\
0
\end{array}
\right]  +\left[
\begin{array}
[c]{c}%
\left(  -y\right)  \left(  -l\right)  +\left(  -z\right)  q_{1}\\
x\left(  -l\right)  -zq_{2}\\
xq_{1}+yq_{2}%
\end{array}
\right]  \\
&  =h\omega+\partial_{1}\left[
\begin{array}
[c]{c}%
-l\\
q_{1}\\
q_{2}%
\end{array}
\right]  ,
\end{align*}
that is, $g^{\sim}=h^{\sim}\omega^{\sim}=h\left(  0\right)  \omega^{\sim}$
(see Lemma \ref{lemKZ2}). Using Lemma \ref{lemMinf}, we deduce that
$\omega^{\sim}\neq0$ and $H_{1}=k\omega^{\sim}$, that is, $d_{1}=1$. By Lemma
\ref{lemAp}, we have $d_{2}=d_{0}-d_{1}+d_{2}-d_{3}=-i_{Y}\left(  a\right)
=0$.

Another example of a singular (at $a=0$) variety $Y$ is given by the following
generators
\[
f_{1}=X_{1}^{2}-X_{2}(1-X_{1}X_{3})+X_{3}^{3},\quad f_{2}=X_{1}^{3}X_{2}%
+X_{2}+\left(  X_{2}-1\right)  X_{3}^{2},\quad f_{3}=X_{1}^{3}+X_{2}^{2}%
X_{3}+X_{3}^{4}%
\]
in $P$. Notice that $f_{1}$ and $f_{2}$ are minimal generators, $m_{1}%
=m_{3}=2$, $m_{2}=1$, $S_{1}=\left\{  1\right\}  $, $S_{2}=\left\{
1,2\right\}  $, $S_{3}=\left\{  2\right\}  $ and $v_{1}=\left(  1,-1,0\right)
$, $v_{2}=\left(  0,1,-1\right)  $. Note also that these vectors can also be
represented by
\[
v_{1}=\left(  \dfrac{\partial^{2}f_{1}}{\partial X_{1}^{2}}\left(  0\right)
,\dfrac{\partial f_{1}}{\partial X_{2}}\left(  0\right)  ,\dfrac{\partial
^{2}f_{1}}{\partial X_{3}^{2}}\left(  0\right)  \right)  ,v_{2}=\left(
f_{2}\left(  0\right)  ,\dfrac{\partial f_{2}}{\partial X_{2}}\left(
0\right)  ,\dfrac{\partial^{2}f_{2}}{\partial X_{3}^{2}}\left(  0\right)
\right)
\]
up to constant multiplies. Based on Lemma \ref{lemMinfi}, we deduce that
$d_{1}\geq2$. Notice also that $x^{3}-g=0$ in $R=k\left[  x\right]  $, where
$g=-y^{2}z-z^{4}$ (since $f_{3}=0$ in $R$). That means that $x$ is integral
over $k\left[  y,z\right]  $ or $a$ is integral with $i_{Y}\left(  a\right)
=0$ (see Remark \ref{remSp}).

\end{document}